\documentclass[10pt,twoside]{amsart}
\usepackage{fullpage}
\usepackage[english]{babel}
\usepackage{graphicx}
\usepackage[T1]{fontenc}
\usepackage{color}
\usepackage{bigints}
\usepackage{soul}

\newtheorem{Ass}{Assumption}
\newtheorem{Thm}{Theorem}
\newtheorem{Prop}[Thm]{Proposition}

\newtheorem{Lem}{Lemma}
\newtheorem{Rem}{Remark}
\newtheorem*{Prop*}{Proposition}
\newtheorem*{Cor*}{Corollary}
\newtheorem*{Thm*}{Theorem}
\usepackage{hyperref}
\usepackage[utf8]{inputenc}
\usepackage[english]{babel}
\usepackage{etoolbox}
\apptocmd{\sloppy}{\hbadness 10000\relax}{}{}
\hypersetup{
	colorlinks=true,
	linkcolor=blue,
	urlcolor=cyan,
	citecolor=red
}

\begin{document}
	%%%%%%%%%%%%%%%%%%%%%%%%%%%%%%%%%%%%%%%%%%%%%%%%%%%%%%%%%%%%%%%%%%%%%
	\title[Degenerate triangular reaction-diffusion system]{Convergence to equilibrium for\\ a degenerate triangular reaction-diffusion system}
	
	%\bibliographystyle{alpha}
	
	%\author[Saumyajit Das]{Saumyajit Das}
	
	%\author[Harsha Hutridurga]{Harsha Hutridurga}	
	\bibliographystyle{plain}
	
	\author[Saumyajit Das]{Saumyajit Das}
	\address{S.D.: Department of Mathematics, Indian Institute of Technology Bombay, Powai, Mumbai 400076 India.}
	\email{194099001@iitb.ac.in}
	
	\author[Harsha Hutridurga]{Harsha Hutridurga}
	\address{H.H.: Department of Mathematics, Indian Institute of Technology Bombay, Powai, Mumbai 400076 India.}
	\email{hutridurga@iitb.ac.in}
	
	\maketitle
	
	\begin{abstract}
		In this article we study a reaction diffusion system with $m$ unknown concentration. The non-linearity in our study comes from an underlying reversible chemical reaction and triangular in nature. Our objective is to understand the large time behaviour of solution where there are degeneracies. In particular we treat those cases when one of the diffusion coefficient is zero and others are strictly positive. We prove convergence to equilibrium type of results under some condition on stoichiometric coefficients in dimension $1$,$2$ and $3$ in correspondence with the existence of classical solution. For dimension greater than 3 we prove similar result under certain closeness condition on the non-zero diffusion coefficients and with the same condition imposed on stoichiometric coefficients. All the constant occurs in the decay estimates are explicit.
	\end{abstract}

\section{Introduction}\label{sec:introduction}

Reaction-diffusion equations govern the evolution (in time) of species concentrations at various spatial locations that are simultaneously diffusing and undergoing chemical reactions. We consider a reaction-diffusion model that concerns the $m$ species $X_1$,$\cdots$. $X_{m-1}$ and $X_m$ undergoing the following reversible reaction:	
	\[
	\alpha_1 X_1+\alpha_2X_2+\cdot+\alpha_{m-1}X_{m-1}{\rightleftharpoons}X_m
	\]
Here the medium is some domain  $\Omega\in \mathbb{R}^n$, generally taken to be smooth and bounded with $\partial\Omega$ it's boundary. For all $1\leq i\leq m-1$,  $\alpha_i$ generally taken non-negative real number and most preferably non negative integer. For all $1\leq i\leq m$, we denote  $a_i(t,x)$ as the concentration of the compound  $X_i$ at a particular space and time. Furthermore $a_i(t,x)$  satisfies the following p.d.e.	
\begin{equation}\label{triangular  model}
	\left \{
	\begin{aligned}
		\partial_t a_i - d_i \Delta a_i=& a_m-\displaystyle\prod\limits_{j=1}^{m-1}(a_j)^{\alpha_j} \qquad \qquad \mbox{in}\ (0,T)\times\Omega\\
		\partial_t a_m-d_{m}\Delta a_m=& \displaystyle\prod\limits_{j=1}^{m-1}(a_j)^{\alpha_j}-a_m \qquad \qquad \mbox{in}\ (0,T)\times\Omega \\
		\nabla_{x}a_i.\gamma= & 0 \qquad \qquad \qquad \qquad \qquad \mbox{on}\ (0,T)\times \partial \Omega \\
		a_{i}(0,x)=& a_{i,0} \qquad \qquad \ \ \   \qquad \qquad \ \ \ \mbox{in}\  \Omega.
	\end{aligned}
	\right .
\end{equation}
Here, the initial condition is taken to be positive and smooth up to the boundary. For all $1\leq i\leq m$, the diffusion coefficients $d_i$ are assumed to be non-negative constant. Although here the rate functions, appears on the right hand side corresponding to each density are difference between two product type polynomials, one can express the general reaction-diffusion equation in the following way

\vspace{.3cm}

	General Reaction-Diffusion System: \ For all $1\leq i\leq m$:
	
	\begin{equation}\label{general reaction diffusion}
		\left \{
		\begin{aligned}
			\partial_{t}u_{i}-d_{i}\Delta u_{i}= & f_{i}(u_1,u_2,\cdots,u_m) \quad \   \mbox{in} (0,T) \times \Omega \\ 
			\nu_{i} \nabla_x u_i. \gamma+(1-\nu_{i})u_{i}=& 0 \qquad \qquad \qquad \qquad \ \ \mbox{on} \ (0,T) \times \partial\Omega   \\ 
			u_{i}(0, \cdot)= & u_{i0} \qquad \qquad \qquad \qquad \mbox{in} \ \Omega.
		\end{aligned}
		\right .
	\end{equation}
Here the unknown is $u=(u_1,u_2,\cdots,u_m):[0,T)\times\Omega\to\mathrm{R}^m$. The functions $f_i:\mathbb{R}^m\to\mathbb{R}$ for $1\leq i\leq m$ are $C^1$ functions. The coefficients $d_1,\cdots,d_m$ are non-negative constants. In the boundary condition, we take the constants $\nu_i\in[0,1]$ for $1\leq i\leq m$. We consider the spatial domain $\Omega\subset\mathbb{R}^{N}$ to be bounded and smooth with $\partial\Omega$ it's boundary. We emphasize on triangular reaction-diffusion system. 

\vspace{.2cm}

Triangular Reaction-Diffusion System: If the rate functions corresponding to \eqref{general reaction diffusion}, for all $1\leq i\leq m$, satisfy the following:
\[
\texttt{P}(f_1,\cdots,f_m)^{t}\leq \left[1+\sum\limits_{i=1}^{m}u_i\right]\texttt{b} \quad \mbox{when}\ u_i\geq 0,
\]
for some lower triangular positive matrix $\texttt{P}$ and for some positive vector  $\texttt{b}\in \mathbb{R}^m$, then the reaction-diffusion system \eqref{general reaction diffusion} is called a triangular one.  The rate functions corresponding to our system \eqref{triangular  model} are the following:

\begin{equation*}
	\left \{
	\begin{aligned}
		f_i = & a_m-\displaystyle\prod_{j=1}^{m-1}(a_j)^{\alpha_j} \quad \forall i=1,\cdots,m-1\\
		f_m=& -f_1.
	\end{aligned}
	\right . 
\end{equation*}

Consider the positive definite lower triangular matrix
\begin{equation*}
	P=(p_{ij})_{m\times m}= \left \{
	\begin{aligned}
		&1 \qquad i=j, i\in\{1,\cdots,m \}\\
		&1 \qquad j=i-1, i\in \{2,\cdots,m\}\\
		& 0 \qquad \text{otherwise}.
	\end{aligned}
	\right .
\end{equation*}

The action of this matrix on the rate function vector $(f_1,\cdots,f_m)$ is the the following product

\[
P \begin{pmatrix}
	f_1 \\ \cdot \\ \cdot\\ \cdot \\ f_m
\end{pmatrix} \leq \left[1+\sum\limits_{1}^{m}a_i\right] \begin{pmatrix}
	1 \\ 2\\ \cdot \\ \cdot \\ 2 \\ 0
\end{pmatrix},
\]
which makes our system a triangular one. 

\vspace{.3cm}

 When all the diffusion coefficients are strictly positive i.e., $d_i>0$ for all $i=1,\cdots,m$ (here onwards we call this as non-degenerate setting),
well posedness of this triangular system and global-in-time existence of classical positive solution  is established in  articles \cite{Pie2010} \cite{Pierre2003}. In the article \cite{DF15}, authors first analyze the simplest three species triangular reaction-diffusion system where one or more species stops diffusing (here onwards we call this as degenerate setting). Authors consider the following reversible reaction:
\[
X_1+X_2 \rightleftharpoons X_3,
\] 

where the concentrations $a_1,a_2,a_3$ satisfy the following system 

\begin{equation}\label{eq:model}
	\left \{
	\begin{aligned}
		\partial_t a_1 -d_1 \Delta a_1 =& a_3-a_1a_2 \qquad \mbox{in}\ (0,T)\times\Omega\\
		\partial_t a_2 -d_2 \Delta a_2 =& a_3-a_1a_2 \qquad \mbox{in}\ (0,T)\times\Omega\\
		\partial_t a_3 -d_3 \Delta a_3 =& a_1a_2-a_3 \qquad \mbox{in}\ (0,T)\times\Omega\\
		\nabla_x a_i.\gamma=& 0 \qquad \qquad \mbox{on}\ (0,T)\times\partial\Omega, \forall i=1,2,3\\
		a_{i}(0,x)=& a_{i,0}(\in C^{\infty}(\overline{\Omega}))\geq 0 \qquad \forall i=1,2,3.
	\end{aligned}
	\right .
\end{equation}
Authors established existence of global-in-time positive classical solution  for all degenerate cases in any dimension except when $d_3=0$ and $d_1$,$d_2$ strictly positive. However, this particular degeneracy also attains a global-in-time positive classical solution up to dimension $3$. The degenerate triangular reaction-diffusion system \eqref{triangular  model} is analysed in the article [SH]. The authors established barring the degenerate case where $d_m=0$ and rest of the diffusion coefficient strictly positive, system \eqref{triangular  model} always has a unique global-in-time positive classical solution. However, this particular degeneracy also attains a global-in-time positive classical solution up to dimension $2$.

\vspace{.2cm}

The natural question arises, how the solution behaves as time growing larger and larger. The large time behaviour of solutions to \eqref{eq:model} in the non-degenerate setting was addressed in \cite{DF06} using the method of entropy. In this method, a Lyapunov functional (termed entropy) is found for the evolution equation. The negative of the time derivative of this entropy functional is referred to as the entropy dissipation functional. The entropy dissipation functional is then related back to the relative entropy via a functional inequality. This will then be followed by a Gr\"onwall type argument to deduce convergence of relative entropy to zero. A Czisz\'ar-Kullback-Pinsker type inequality that relates relative entropy and the $\mathrm L^1$-norm helps the authors in \cite{DF06} to prove the convergence to equilibrium in the $\mathrm L^1(\Omega)$-norm and that the decay is exponentially fast in time.  On a similar note, authors in the article \cite{FLT20} obtained similar convergence to equilibrium in the $\mathrm L^1(\Omega)$-norm and that the decay is exponentially fast in time. In both the proof of \cite{DF06} and \cite{FLT20}, uniform boundedness of the solution  to \eqref{eq:model} and \eqref{triangular  model} plays an important role. This uniform bound is available in the non-degenerate case.

\vspace{.2cm}

The study of large time behaviour of solutions for the degenerate case is quite different as the uniform bound of solutions is not yet known. In the article \cite{EMT20}, authors analysed global-in-time existence and large time behaviour of solution of a particular $4\times 4$ degenerate reaction-diffusion system. Although, the same entropy method is used but to relate entropy with it's dissipation we need to go through various integral estimate in unit cylinder which is established under some assumption on the non-zero diffusion coefficients. Under the same assumption on the diffusion coefficients, authors proved a similar convergence to equilibrium in the $\mathrm L^1(\Omega)$-norm but, here the the decay is sub-exponentially fast in time. In the light of this article authors in the article [SH] derive  similar sub-exponentially fast convergence to equilibrium in the $\mathrm L^1(\Omega)$-norm for the three species degenerate case \eqref{eq:model}, when there is only one degeneracy. The result is unconditional up to dimension $3$. However, authors extended the result in any dimension for the case where $d_1=0$ and other two diffusion coefficient is strictly positive, under the assumption that the two non-zero diffusion coefficients are sufficiently close. In all the cases, the equilibrium states are constants and can be described from the mass conservation nature of the system. In this article, we deduce similar sub-exponentially fast convergence towards equilibrium results for the degenerate triangular reaction-diffusion system \eqref{triangular  model}  where there is only one degeneracy and stoichiometric coefficients being natural number. We further assume that the stoichiometric coefficient corresponding to degenerate diffusion coefficient is $1$, which is automatically satisfied when $d_m=0$. We first address the mass conservation nature of the triangular reaction-diffusion system and equilibrium states.

\vspace{.2cm}
\textbf{Mass Conservation Property, Equilibrium States}:\\

Considering the initial datum are strictly positive, we have the following mass conservation estimates for the  triangular reaction-diffusion system \eqref{triangular  model}:

\begin{align} 
	0<M_{i} :=\int_{\Omega}a_i(t,x)+a_m(t,x)=\int_{\Omega}a_i(0,x)+a_m(0,x) \qquad \forall i=1,\cdots,m-1. \label{tri mass conserve 1}
\end{align}

The equilibrium state must be a steady state. Let we denote  $(a_{1\infty},\cdots,a_{m\infty})$ as the equilibrium state for the triangular  reaction-diffusion system \eqref{triangular  model}. In particular, we search for constants where the source term vanish. Hence, this constant state is a steady state.  So, the equilibrium state satisfies:
\begin{align}\label{triangular equibrium state 2}
	a_{m\infty}=\displaystyle{\prod\limits_{j=1}^{m-1}a_{j\infty}^{\alpha_j}}
\end{align}
Moreover, the mass conservation estimates also satisfied by the equilibrium state. Keeping in mind that we consider the constant steady state as described above, we obtain
\begin{align}\label{triangular equilibrium state 1}
	a_{i\infty}+a_{m\infty}=\frac{M_{i}}{\vert \Omega\vert} \qquad \forall i=1,\cdots,m-1.
\end{align}

Consider the following function
\[
f(x)=\displaystyle{\prod\limits_{j=1}^{m-1}\left(\frac{M_j}{\vert \Omega\vert}-x\right)^{\alpha_j}-x}\qquad x\in\left[0,\min\limits_{i=1,\cdots,m-1}\left\{\frac{M_i}{\vert \Omega\vert}\right\}\right]
\]
$f$ is continuously differentiable and observe that $f(x)>0$ when $x=0$ and $f(x)<0$ when $x=\displaystyle{\min\limits_{i=1,\cdots,m-1}\left\{\frac{M_i}{\vert \Omega\vert}\right\}}$. Intermediate value property yields a root of $f$ in the interval $\left(0,\displaystyle{\min\limits_{i=1,\cdots,m-1}\left\{\frac{M_i}{\vert \Omega\vert}\right\}}\right)$. Furthermore, consider the first derivative of $f$
\[
f'(x)= -\prod\limits_{j=1}^{m-1}\left(\frac{M_j}{\vert \Omega \vert}-x\right)^{\alpha_j} \sum\limits_{i=1}^{m-1}\frac{\alpha_i}{\frac{M_i}{\vert \Omega\vert}-x} -1< 0 \qquad x\in \left(0,\displaystyle{\min\limits_{i=1,\cdots,m-1}\left\{\frac{M_i}{\vert \Omega\vert}\right\}}\right).
\]
Hence $f(x)$ has only one root in $\left(0,\displaystyle{\min\limits_{i=1,\cdots,m-1}\left\{\frac{M_i}{\vert \Omega\vert}\right\}}\right)$. This ensures existence of such unique equilibrium points.

For non-degenerate triangular system, authors in the article \cite{FLT20} have shown that the global-in-time positive classical solution decays around it's equilibrium point exponentially fast in time in $\mathrm{L}^1(\Omega)$  norm. For degenerate triangular reaction-diffusion system \eqref{triangular  model} we follow similar approach as in degenerate three species system as described in [SH]. We obtain various estimates inspired by the articles \cite{EMT20}\cite{Lam87}\cite{Tan18}. These estimates helps us to obtain a polynomial type time growth  of the solution and  enable us to relate entropy functional with it's dissipation. Thus provide  asymptotic decay results. We further use various Neumann Green's function results from \cite{FMT20} \cite{rothe06} \cite{Morra83} \cite{ML15}, to get various estimates. The triangular degenerate systems we analyze for asymptotic behavior are described below:

\vspace{.2cm}
\text{Degenerate Triangular Species Model, Convergence to Equilibrium}:
\begin{equation} \label{degenerate 1}
	\left\{
	\begin{aligned}
		\partial_t a_1 = & a_m-a_1\prod\limits_{j=2}^{m-1}a_j^{\alpha_j}  \qquad \mbox{ in } (0,T)\times\Omega \\
		\partial_t a_i-d_i \Delta a_i= & a_m-a_1\prod\limits_{j=2}^{m-1}a_j^{\alpha_j}  \qquad \mbox{ in } (0,T)\times\Omega \ \ \forall \ i=2,\cdots,m-1\\
		\partial a_m-d_m \Delta a_m= & a_1\prod\limits_{j=2}^{m-1}a_j^{\alpha_jj}-a_m  \qquad \mbox{ in } (0,T)\times\Omega \\
		\nabla_x a_i.\gamma= & 0  \qquad \qquad \qquad \qquad \  \mbox{ on } (0,T)\times\partial\Omega, \ \  \forall \ i=1,\cdots,m\\
		a_i(0,x)=& a_{i,0}  \qquad \qquad \ \ \qquad \ \  \mbox{ in } \Omega,\ \  \forall \ i=1,\cdots,m
	\end{aligned}
	\right.
\end{equation}
and 
\begin{equation}\label{triangular d_m=0}
	\left\{
	\begin{aligned}
		\partial_t a_i - d_i \Delta a_i=& a_m-\displaystyle\prod\limits_{j=1}^{m-1}(a_j)^{\alpha_j} \qquad \qquad \mbox{in}\ (0,T)\times\Omega, \forall i=1,\cdots,m-1\\
		\partial_t a_m=& \displaystyle\prod\limits_{j=1}^{m-1}(a_j)^{\alpha_j}-a_m \qquad \qquad \mbox{in} (0,T)\times\Omega \\
		\nabla_{x}a_i.\gamma= & 0 \qquad \qquad \qquad \qquad \qquad \ \mbox{on}\ (0,T)\times \partial \Omega, \forall i=1,\cdots,m \\
		a_{i}(0,x)=& a_{i,0} \qquad \qquad \qquad \qquad \ \ \ \ \ \mbox{in} \ \Omega, \forall i=1,\cdots,m.
	\end{aligned}
	\right .
\end{equation}

For both the cases, initial datum are taken to be strictly positive and smooth up to the boundary. For all $1\leq i\leq m-1$, stoichiometric coefficients $\alpha_i\in\mathbb{N}$. Let $\alpha_m=1$. We consider the following entropy functional associated with the system  \eqref{degenerate 1} and \eqref{triangular d_m=0}:

\begin{align}\label{entropy}
	E(a_1,\cdots,a_m)=\sum\limits_{i=1}^m\int_{\Omega}  \alpha_i(a_i(\ln{a_i}-1)+1) \rm{d}x
\end{align}
The entropy dissipation functional is the following:
\begin{equation}\label{dissipation functional} 
	\left \{
	\begin{aligned}
		\text{For} \ d_1&=0\\
		D(a_1,\cdots,a_m)=\sum\limits_{i=2}^m\int_{\Omega} &\alpha_i d_i  \frac {\vert \nabla a_i \vert^2}{a_i} \rm{d}x \\ 
		+& \int_{\Omega}   \left(a_m-a_1\displaystyle \prod\limits_{j=2}^{m-1} (a_j)^{\alpha_j}\right)\ln{\left(\frac{a_m}{a_1\prod\limits_{j=2}^{m-1} (a_j)^{\alpha_j}}\right)}\rm{d}x. \nonumber   \\
		\text{For} \ d_m&=0\\
		D(a_1,\cdots,a_m)=\sum\limits_{i=1}^{m-1}\int_{\Omega} & \alpha_i d_i \frac{\vert \nabla a_i \vert^2}{a_i}\rm{d}x\\
		& +\int_{\Omega}   \left(a_m-\displaystyle \prod_{j=1}^{m-1} (a_j)^{\alpha_j}\right)\ln{\left(\frac{a_m}{\prod\limits_{j=1}^{m-1} (a_j)^{\alpha_j}}\right)}\rm{d}x . 
	\end{aligned}
	\right .
\end{equation}
Differentiating entropy functional with respect to time yields entropy dissipation functional with a negative sign. To apply Gr\'onwal inequality, we relate these two functionals. To relate this for the system \eqref{degenerate 1}, the key step is to obtain an $\mathrm{L}^p$ integral estimate of the species concentrations $a_i$, for all $i=2,\cdots,m$ where $p>N$ and $\mathrm{L}^{\frac{N}{2}}$ integral estimate of $a_1$. These estimates yields a polynomial time growth on the supremum of the solutions.  Thus, provides a sub-exponential decay results. These estimates are unconditional upto the dimension $3$ for the system \eqref{degenerate 1}, whereas for the  dimensions $N\ge4$, we have been able to arrive at those estimates only under certain closeness assumption on the non-zero diffusion coefficients. The precise assumption is the following:

\begin{Ass}
	The non-zero diffusion coefficients $d_i$ and $d_m$ are said to satisfy the closeness assumption if
	\begin{align}\label{tri closeness condition imp}
		\vert{d_i-d_m}\vert < \frac{1}{C^{PRC}_{{\frac{d_i+d_m}{2}},p'}}
		\qquad 
		\text{ and } 
		\qquad \frac{\vert d_i-d_m \vert}{d_i+d_m}<\frac{1}{C_{SOR}(\Omega,N,p')},
	\end{align}
	where the constants $C^{PRC}_{{\frac{d_a+d_c}{2}},p'}$ and $C_{SOR}(\Omega,N,p')$, are the parabolic regularity constant $($see Appendix \ref{PE}$)$ and the second order regularity constant $($see Appendix \ref{estimation 1}$)$, respectively.
\end{Ass}

For the degenerate triangular system \eqref{triangular d_m=0}, we have same unconditional sub-exponential convergence result up to dimension $2$. In this report the two main results for the triangular  degenerate reaction-diffusion equation are described below:

\begin{Thm}\label{theorem convergence 1 tri}
	Let dimension $N\geq4$ and let $(a_1,\cdots,a_m)$ be the solution of the degenerate triangular reaction-diffusion system \eqref{degenerate 1}. Let $(a_{1\infty},\cdots,a_{m\infty})$ be equilibrium states given by the \eqref{triangular equilibrium state 1}-\eqref{triangular equibrium state 2}. Furthermore let  $d_i$ and $d_m$ satisfying closeness condition \eqref{tri closeness condition imp} for some $i\in\{2,\cdots,m-1\}$. Then, for positive $\epsilon\ll 1$, there exists a time $T_{\epsilon}$ and positive constants $\lambda_1,\lambda_2$ depends on the domain, dimension, initial condition and $\epsilon$ and a positive constant $C_{LE}$ explicitly defined in \eqref{lower bound of entropy}, such that:
	\[
	\sum\limits_{i=1}^{m}\Vert a_i-a_{i\infty}\Vert_{\mathrm{L}^1(\Omega)}^{2} \leq 
	\frac{\lambda_1}{C_{LE}}e^{-\lambda_2(1+t)^{\frac{1-\epsilon}{N-1}}} \qquad \forall t\geq T_{\epsilon}.
	\]
	
	For dimension $N=1,2,3$, there exists a time $T_{\epsilon}$, positive constants $\lambda_1,\lambda_2$ depends on the domain, dimension, initial condition and $\epsilon$ and a constant and a positive constant $C_{LE}$ explicitly defined \eqref{lower bound of entropy}, such that:
	
	\[
	\sum\limits_{i=1}^{m}\Vert a_i-a_{i\infty}\Vert_{\mathrm{L}^1(\Omega)}^{2} \leq 
	\frac{\lambda_1}{C_{LE}}e^{-\lambda_2(1+t)^{\frac{1-\epsilon}{2}}} \qquad \forall t\geq T_{\epsilon}
	\]
	Regardless of any closeness assumption on non-zero diffusion coefficients.  
\end{Thm}

\begin{Thm}\label{theorem comvergence 2 tri}
	Let dimension $N=1,2$ and $(a_1,\cdots,a_m)$ be the solution of the degenerate triangular reaction-diffusion system \eqref{triangular d_m=0}. Let  $(a_{1\infty},\cdots,a_{m\infty})$ be the equilibrium states \eqref{triangular equilibrium state 1}-\eqref{triangular equibrium state 2}. Then, for positive $\epsilon\ll 1$, there exists a time $T_{\epsilon}$ and positive constants $\lambda_1,\lambda_2$ depends on the domain, dimension, initial condition and $\epsilon$ and $C_{LE}$ explicitly defined in \eqref{lower bound of entropy}, such that:
	
	\[
	\sum\limits_{i=1}^{m}\Vert a_i-a_{i\infty}\Vert_{\mathrm{L}^1(\Omega)}^{2} \leq 
	\frac{\lambda_1}{C_{LE}}e^{-\lambda_2(1+t)^{1-\epsilon}} \qquad \forall t\geq T_{\epsilon}.
	\]
\end{Thm}

\vspace{.2cm}
Some Notation:
\begin{itemize}
	\item For a function $f:\Omega\to \mathbb{R}$, the average $\displaystyle{\overline{f}=\frac{1}{\vert \Omega\vert}\int_{\Omega}f}$.
	\item $\sqrt{a_i}=A_i$, $\sqrt{a_{i\infty}}=A_{i\infty}$, $\delta_{A_i}= A_i-\overline{A_i}$.
	\item $\Omega_{T}:=(0,T)\times \Omega, \qquad \Omega_{\tau,T}:=(\tau,T)\times\Omega, \ 0\leq\tau<T$,
	\item $\partial\Omega_{T}:=(0,T)\times\partial\Omega,\quad \partial\Omega_{\tau,T}:=(\tau,T)\times\partial\Omega, \ 0\leq\tau<T$,
	\item $\alpha_m=1$.
\end{itemize} 
Apart from the above notations we heavily use the following time cut-off function to derive various estimate in unit cylinder. 
\\
Time cut-off function: Let $\phi:[0,\infty)\rightarrow [0,1]$ be a smooth function such that 
\[
\phi(0)=0, \ \phi\big\vert_{[1,\infty)}=1 \text{\ and \ } \phi'\in[0,M_{\phi}],
\]
where $M_{\phi}$ is a positive constant. For any $\tau\geq 0$, define $\phi_{\tau}:[\tau, \infty)\to[0,1]$ as follows
\begin{align}\label{cut-off}
\phi_{\tau}(s)=\phi(s-\tau) \qquad \mbox{for} \ s\in[\tau, \infty).
\end{align}
We also used the following  Green's function estimate from \cite{Morra83}\cite{ML15}. There exists a constant $\tilde{\kappa}_{d_i}>0$, depending only on the domain, such that
\begin{align}\label{Heat kernel estimate}
0\leq G_{d_i}(t_1,s,x,y) \leq \frac{\tilde{\kappa}_{d_i}}{(4\pi( t_1-s))^\frac{N}{2}}e^{-\kappa_{d_i} \frac{\Vert x-y \Vert^2}{(t_1-s)}}\leq g_{d_i}(t_1-s,x-y) \qquad 0\leq s<t,
\end{align}
for some constant $\kappa_{d_i}>0$ depending only on $\Omega$ and the diffusion coefficient $d_i$. Furthermore, the function $g_{d_i}(t_1-s,x-y)$ is defined in the following way:
\begin{align}\label{Heat kernel bound function}
g_{d_i}(t,x):= \frac{\tilde{\kappa}_{d_i}}{(4\pi\vert t \vert)^\frac{N}{2}}e^{-\kappa_{d_i} \frac{\Vert x-y \Vert^2}{\vert t\vert}}, \qquad (t,x)\in \mathbb{R}\times\mathbb{R}^{N}.
\end{align}
Note that
\begin{align}\label{Heat kernel integral estimate}
\left\Vert g_{d_i}\right\Vert_{\mathrm{L}^{z}((-2,2)\times\mathbb{R}^N)} \leq \tilde{\kappa}_{1,d_i} \qquad \forall z\in\left[1,1+\frac{2}{N}\right)
\end{align}
for some constant positive $\tilde{\kappa}_{1,d_i}$ depending on $z$. Next, we give a brief outline of our proof. 

\vspace{.2cm}
The entropy dissipation functional for the degenerate system \eqref{degenerate 1}:
\begin{align}
	D(a_1,\cdots,a_m)=\sum\limits_{i=2}^{m}\int_{\Omega} \alpha_i d_i  \frac {\vert \nabla a_i \vert^2}{a_i} 
	+ \int_{\Omega}   \Big(a_m-a_1\displaystyle \prod\limits_{j=2}^{m-1} (a_j)^{\alpha_j}\Big)\ln{\left(\frac{a_m}{a_1\prod\limits_{j=2}^{m-1} (a_j)^{\alpha_j}}\right)} \nonumber.
\end{align}

We rewrite the terms in the above dissipation functional as follows:
\[
D(a_1,\cdots,a_m)=\sum\limits_{i=2}^{m}\int_{\Omega} \alpha_i d_i \vert \nabla \sqrt{a_i} \vert^2 +\int_{\Omega}   \Big(a_m-a_1\displaystyle \prod_{j=2}^{m-1} (a_j)^{\alpha_j}\Big)\ln{\left(\frac{a_m}{a_1\prod\limits_{j=2}^{m-1} (a_j)^{\alpha_j}}\right)}.
\]

We recall here an algebraic inequality:
\[ 
(x-y)(\ln{x}-\ln{y}) \geq 4(\sqrt{x}-\sqrt{y})^2, \ \ \forall x,y \geq 0.
\]

This algebraic inequality gives a lower bound on the last term of the above entropy dissipation functional whereas the classical Poincar\'e inequality gives a lower bound for the gradient terms. More precisely, we have
\begin{gather}
	D(a_1,\cdots,a_m) \geq \sum\limits_{i=2}^{m}\int_{\Omega} \frac{\alpha_i d_i}{P(\Omega)} \vert  \delta_{A_i} \vert^2 +\int_{\Omega}   \Big(A_m-A_1\displaystyle \prod_{j=2}^{m-1} (A_j)^{\alpha_j}\Big)^2,
	\label{dissipation-L2}
\end{gather}
where $P(\Omega)$ is the Poincar\'e constant (see Lemma \ref{Poincare-Wirtinger}) on the domain $\Omega$.

Next we apply Poincar\'e Writenger inequality (see Lemma \ref{Poincare-Wirtinger}) on the graident terms. It yields

\begin{equation}\label{dissipation-L_2N/N-2}
	\left \{
	\begin{aligned}
		D(a_1,\cdots,a_m) \geq \sum\limits_{i=2}^{m} & \frac{\alpha_i d_i}{P(\Omega)} \Vert  \delta_{A_i} \Vert_{\mathrm{L}^{\frac{2N}{N-2}}(\Omega)}^2\\
		& +\left\Vert A_m-A_1\displaystyle \prod_{j=2}^{m-1} (A_j)^{\alpha_j}\right\Vert_{\mathrm{L}^2(\Omega)}^2 \qquad \mbox{for}\ N\geq 4,
		\\
		D(a_1,\cdots,a_m) \geq \sum\limits_{i=2}^{m} &\frac{\alpha_i d_i}{P(\Omega)} \Vert  \delta_{A_i} \Vert_{\mathrm{L}^{6}(\Omega)}^2\\
		&+\left\Vert A_m-A_1\displaystyle \prod_{j=2}^{m-1} (A_j)^{\alpha_j}\right\Vert_{\mathrm{L}^2(\Omega)}^2 \qquad \mbox{for}\ N=1,2,3.
	\end{aligned}
	\right .
\end{equation}

Similar to above setting, we now consider the following dissipation functional for the degenerate system \eqref{triangular d_m=0}:
\begin{align}
	D(a_1,\cdots,a_m)=\sum\limits_{i=2}^{m-1}\int_{\Omega}  \alpha_i d_i \frac{\vert \nabla a_i \vert^2}{a_i}\label{dissipation functional d_m}
	+\int_{\Omega}   \Big(a_m-\displaystyle \prod_{j=1}^{m-1} (a_j)^{\alpha_j}\Big)\ln{\left(\frac{a_m}{\prod\limits_{j=1}^{m-1} (a_j)^{\alpha_j}}\right)}. 
\end{align}

It rewrites as
\[
D(a_1,\cdots,a_m)=\sum\limits_{i=1}^{m-1}\int_{\Omega} \alpha_i d_i \vert \nabla \sqrt{a_i} \vert^2 +\int_{\Omega}   \Big(a_m-\displaystyle \prod\limits_{j=1}^{m-1} (a_j)^{\alpha_j}\Big)\ln{\left(\frac{a_m}{\prod\limits_{j=1}^{m-1} (a_j)^{\alpha_j}}\right)}.
\]

Thanks to the algebraic identity mentioned earlier and the Poincar\'e inequality, we arrive at

\begin{gather}
	D(a_1,\cdots,a_m) \geq \sum\limits_{i=1}^{m-1}\int_{\Omega} \frac{\alpha_i d_i}{P(\Omega)} \vert  \delta_{A_i} \vert^2 +\int_{\Omega}   \Big(A_m-\displaystyle \prod_{j=1}^{m-1} (A_j)^{\alpha_j}\Big)^2.
	\label{dissipation-L2 d_m}
\end{gather}

To summarize, the solutions to the degenerate system \eqref{triangular d_m=0} satisfies:

\begin{equation}\label{dissipation-L_2N/N-2 d_m}
	\left \{
	\begin{aligned}
		D(a_1,\cdots,a_m) \geq \sum\limits_{i=1}^{m-1} &\frac{\alpha_i d_i}{P(\Omega)} \Vert  \delta_{A_i} \Vert_{\mathrm{L}^{\frac{2N}{N-2}}(\Omega)}^2\\
		&+\left\Vert A_m-\displaystyle \prod_{j=1}^{m-1} (A_j)^{\alpha_j}\right\Vert_{\mathrm{L}^2(\Omega)}^2 \qquad \ \mbox{for}\ N\geq 4,
		\\
		D(a_1,\cdots,a_m) \geq \sum\limits_{i=1}^{m-1} &\frac{\alpha_i d_i}{P(\Omega)} \Vert  \delta_{A_i} \Vert_{\mathrm{L}^{6}(\Omega)}^2\\
		&+\left\Vert A_m-\displaystyle \prod_{j=1}^{m-1} (A_j)^{\alpha_j}\right\Vert_{\mathrm{L}^2(\Omega)}^2 \ \qquad \mbox{for}\ N=1,2,3.
	\end{aligned}
	\right .
\end{equation}

We are particularly interested in the case $N=2$, i.e.
\begin{align}\label{dissipation-L_6 d_m}
	D(a_1,\cdots,a_m) \geq \sum\limits_{i=1}^{m-1} \frac{\alpha_i d_i}{P(\Omega)} \Vert  \delta_{A_i} \Vert_{\mathrm{L}^{6}(\Omega)}^2+\Vert A_m-\displaystyle \prod\limits_{j=1}^{m-1} (A_j)^{\alpha_j}\Vert_{\mathrm{L}^2(\Omega)}^2.
\end{align}

Note that both the entropy functional and entropy dissipation functional are positive functional. If $(a_1,\cdots,a_m)$ solves either of the system \eqref{degenerate 1} and \eqref{triangular d_m=0}, we obtain that
\[
\frac{dE}{dt}(a_1,\cdots,a_m)=-D(a_1,\cdots,a_m),
\]
which ensures the decay of entropy with time.
Let us consider the relative entropy(relative with respect to the homogeneous equilibrium states given by \eqref{triangular equilibrium state 1}-\eqref{triangular equibrium state 2}), we obtain that
\[
E(a_1,\cdots,a_m)-E(a_{1\infty},\cdots,a_{m\infty}) =\sum\limits_{i=1}^{m}\int_{\Omega}\alpha_i(a_i \ln{a_i}-a_i-a_{i\infty} \ln{a_{i\infty}}+a_{i\infty}).
\]
Thanks to the mass conservation property \eqref{triangular equilibrium state 1} of the equilibrium states, the relative entropy becomes
\begin{align}\label{Entropy Gamma function relation tri}
	E(a_1,\cdots,a_m)-E(a_{1\infty},\cdots,a_{m\infty}) =\sum\limits_{i=1}^{m}\int_{\Omega}\alpha_i (a_i \ln{\frac{a_i}{a_{i\infty}}}-a_i+a_{i\infty}). 
\end{align}

Let us define a function $\Gamma:(0,\infty)\times(0,\infty)\to\mathbb{R}$ as follows:
\begin{equation}
	\Gamma(x,y) :=
	\left\{
	\begin{aligned}
		&\frac{x \ln\left(\frac{x}{y}\right)-x+y}{\left(\sqrt{x}-\sqrt{y}\right)^2}  \qquad \mbox{ for }x\not=y,
		\\
		&2 \qquad  \qquad \qquad \qquad \ \ \ \mbox{ for }x=y.
	\end{aligned}\right.
\end{equation}
It can be shown (see \cite[Lemma 2.1, p.162]{DF06} for details) that the above defined function satisfies the following bound:
\begin{align}\label{eq:Gamma-bound}
	\Gamma(x,y)\leq C_{\Gamma}\max\left\{1,\ln\left(\frac{x}{y}\right)\right\}
\end{align}
for some positive constant $C_{\Gamma}$. Note that using the function $\Gamma$ defined above, the relative entropy can be rewritten as 
\[
E(a_1,\cdots,a_m)-E(a_{1\infty},\cdots,a_{1\infty})= \sum\limits_{1}^{m}\int_{\Omega}\alpha_i \Gamma(a_i,a_{i\infty})(A_i-A_{i\infty})^2
\]
Using the aforementioned bound for $\Gamma$, we obtain
\begin{align*}
	E(a_1,\cdots,a_m)&-E(a_{1\infty},\cdots,a_{m\infty})\\
	\leq & \max\limits_{i=1,\cdots,m}\{\alpha_i\}C_{\Gamma}\max\limits_{i=1,\cdots,m} \big\{1,\ln{(\Vert a_i\Vert_{\mathrm{L}^{\infty}(\Omega)}+1)}+\vert \ln{a_{i\infty}}\vert\big\}\sum\limits_{i=1}^{m}\Vert A_i-A_{i\infty}\Vert_{\mathrm{L}^2(\Omega)}^2.
\end{align*} 

So, other than the logarithmic growth of the supremum norm of the solution, the growth of relative entropy depends on the $\mathrm{L}^2(\Omega)$ norm of the deviation of the $A_i$ from $A_{i\infty}$. Observe from \eqref{dissipation-L_2N/N-2} and \eqref{dissipation-L_2N/N-2 d_m} that the dissipation functional is also related to this $\mathrm{L}^2(\Omega)$ norm of $\delta_{A_i}$. Observe that
\[
\Vert A_i-A_{i\infty}\Vert_{\mathrm{L}^2(\Omega)}^2 \leq 3\left( \Vert A_i-\overline{A_i}\Vert_{\mathrm{L}^2(\Omega)}^2+\Vert \overline{A_i}-\sqrt{\overline{A_i^2}}\Vert_{\mathrm{L}^2(\Omega)}^2+\Vert \sqrt{\overline{A_i^2}}-A_{i\infty}\Vert_{\mathrm{L}^2(\Omega)}^2\right). 
\]
The following observation says that the first term on the right hand side dominate the second term:
\begin{align*}
	\Vert \overline{A_i}-\sqrt{\overline{A_i^2}}\Vert_{\mathrm{L}^2(\Omega)}^2= &\vert\Omega\vert \big\vert
	\overline{A_i}-\sqrt{\overline{A_i^2}} \big\vert^2 = \vert\Omega\vert \left( \overline{A_i^2}+\overline{A_i}^2-2\overline{A_i}\sqrt{\overline{A_i^2}} \right)\\ &
	\leq\vert\Omega\vert \left( \overline{A_i^2}-\overline{A_i}^2\right) \leq \Vert A_i-\overline{A_i}\Vert_{\mathrm{L}^2(\Omega)}^2.
\end{align*}
Here we used the fact that $\displaystyle{\overline{A_i}\leq \sqrt{\overline{A_i^2}}}$, thanks to H\"older inequality. Hence
\[
\Vert A_i-A_{i\infty}\Vert_{\mathrm{L}^2(\Omega)}^2 \leq 6\left( \Vert A_i-\overline{A_i}\Vert_{\mathrm{L}^2(\Omega)}^2+\left\Vert \sqrt{\overline{A_i^2}}-A_{i\infty}\right\Vert_{\mathrm{L}^2(\Omega)}^2\right). 
\]

Hence we deduce
\[
\sum\limits_{i=1}^m\Vert A_i-A_{i\infty}\Vert_{\mathrm{L}^2(\Omega)}^2 \leq 6\left(\sum\limits_{i=1}^m \Vert A_i-\overline{A_i}\Vert_{\mathrm{L}^2(\Omega)}^2+\sum\limits_{i=1}^m\left\Vert \sqrt{\overline{A_i^2}}-A_{i\infty}\right\Vert_{\mathrm{L}^2(\Omega)}^2\right).
\]

Next we borrow a result from \cite{FLT20} (see Appendix, Theorem \ref{ED}) which says that there exists $C_{EB}>0$, depending only on the domain and the equilibrium state $(a_{1\infty},\cdots,a_{m\infty})$ such that

\[
\sum\limits_{i=1}^{m}\Vert \sqrt{\overline{A_i^2}}-A_{i\infty} \Vert_{\mathrm{L}^2(\Omega)}^2 \leq C_{EB} \left ( \sum_1^{m}  \Vert  \delta_{A_i} \Vert_{\mathrm{L}^{2}(\Omega)}^2+\left\Vert A_m-\displaystyle \prod_{j=1}^{m-1} (A_j)^{\alpha_j}\right\Vert_{\mathrm{L}^2(\Omega)}^2\right).
\]

It is evident that for the degenerate system \eqref{degenerate 1}, the gradient term $\displaystyle{\Vert \nabla \delta_{A_1} \Vert_{\mathrm{L}^2(\Omega)}}$ is missing in the dissipation expression. Similarly for the degenerate system \eqref{triangular d_m=0}, the term $\displaystyle{\Vert \nabla \delta_{A_m} \Vert_{\mathrm{L}^2(\Omega)}}$  is missing in the dissipation expression. If we relate dissipation with the corresponding missing terms, we will be able to produce entropy-entropy dissipation inequality provided we know how the logarithm of the supremum of the solution behaves. It turns out that the lower bound \eqref{dissipation-L_2N/N-2} for the entropy dissipation can be improved to include the gradient term $\displaystyle{\Vert \nabla \delta_{A_1} \Vert_{\mathrm{L}^2(\Omega)}}$ but with a price of the constant being time dependent. We prove in Proposition \ref{Dissipatation theorem d_1} and Proposition \ref{Dissipatation theorem d_1, N=3} that there exists a constant $\hat{C}>0$ such that 

\begin{equation}\label{new entity triangular one}
	\left \{
	\begin{aligned}
		\mbox{For}\ N\leq 3, \ \mbox{we have}
		\\
		D(a_1,\cdots,a_m) \geq \hat{C}(1+t)^{-\frac{1}{2}}& \left(\sum\limits_{i=1}^{m}  \Vert  \delta_{A_i} \Vert_{\mathrm{L}^{2}(\Omega)}^2+\left\Vert A_m-A_1\displaystyle \prod_{j=2}^{m-1} (A_j)^{\alpha_j}\right\Vert_{\mathrm{L}^2(\Omega)}^2\right).
		\\
		\mbox{for}\ N\geq 4, \ \mbox{we have}
		\\
		D(a_1,\cdots,a_m) \geq \hat{C}(1+t)^{-\frac{N-2}{N-1}}&\left(\sum\limits_{i=1}^{m}  \Vert  \delta_{A_i} \Vert_{\mathrm{L}^{2}(\Omega)}^2+\left\Vert A_m-A_1\displaystyle \prod_{j=2}^{m-1} (A_j)^{\alpha_j}\right\Vert_{\mathrm{L}^2(\Omega)}^2\right). 
	\end{aligned}
	\right .
\end{equation}

In the context of the degenerate system \eqref{triangular d_m=0}, we improve the lower bound in \eqref{dissipation-L_2N/N-2 d_m} to include the missing gradient term $\displaystyle{\Vert \nabla \delta_{A_m} \Vert_{\mathrm{L}^2(\Omega)}}$. More precisely in Proposition \ref{dissipation d_m}, we proved that for dimension $N\leq 2$, there exists a positive constant $\hat{C}$ such that

\[
D(a_1,\cdots,a_m) \geq \hat{C}\left(\sum\limits_{i=1}^{m}  \Vert  \delta_{A_i} \Vert_{\mathrm{L}^{2}(\Omega)}^2+\left\Vert A_m-\displaystyle \prod_{j=1}^{m-1} (A_j)^{\alpha_j}\right\Vert_{\mathrm{L}^2(\Omega)}^2\right).
\]

The result in Proposition \ref{Dissipatation theorem d_1} concerning the dimension $N\geq 4$ is conditional in the sense that \eqref{new entity triangular one} is valid under the assumption that the diffusion coefficients $d_m$ and at least one of tyhe diffusion coefficients $\{d_2,\cdots,d_{m-1}\}$ are close to each other. The result in Proposition \ref{Dissipatation theorem d_1, N=3} dealing with dimension $N\leq 3$ is however unconditional.

We briefly highlight the key points which leads to \eqref{new entity triangular one}. For $d_1=0$ with $\alpha_1=1$, we relate the missing term $\displaystyle{\Vert A_1-\overline{A_1}\Vert_{\mathrm{L}^2(\Omega)}}$ with the entropy dissipation. We describe the key steps below.

\vspace{.2cm}

Step 1: We obtain a large integral estimate of all the species except $a_1$ and $a_m$ in a parabolic cylinder with unit height up to dimension $N=3$. The result is as follows
\[
\Vert a_i(t,x) \Vert_{\mathrm{L}^9(\Omega_{\tau,\tau+1})} \leq {K_{\Gamma}},
\qquad i=2,\cdots,m-1, \tau \geq 0,
\]
where $K_{\Gamma}$ is a positive constant, independent of time.

\vspace{.2cm}

Step 2: We pass the integral regularity to $a_m$ up to dimension $N=3$.

\[
\Vert a_m(t,x) \Vert_{\mathrm{L}^9(\Omega_{\tau,\tau+1})} \leq \texttt{K},
\]
where $\texttt{K}$ is a positive constant, independent of time.

\vspace{.2cm}

Step 3: We will go reverse when dimension $N\geq 4$. We obtain the following relation
\[
\Vert a_m \Vert_{\mathrm{L}^{p}(\Omega_{\tau,\tau+1})} \leq C_0,
\]
where $C_0$ is a positive constant independent of time and we choose $p>N$. Here in this case we need the closeness condition\eqref{tri closeness condition imp} on non-zero diffusion coeeficients.

\vspace{.2cm}

Step 4: Above integral estimates in dimension $N\geq 4$ also yields the growth of supremum norm of the solution barring the species $a_1$ and $a_m$.  
\[
\Vert a_i \Vert_{\mathrm{L}^{\infty}(\Omega)} \leq K_{\infty}, \qquad \forall i=2,\cdots,m-1,
\]
where $K_{\infty}$ is a positive constant, independent of time.

\vspace{.2cm}

Step 5: Integrability of $a_m$ further yields the $\mathrm{L}^{\frac{N}{2}}(\Omega)$ norm estimate of the species $a_1$.
\begin{equation*}
	\left \{
	\begin{aligned}
		\mbox{For dimension}& \ N\geq 4, \ \exists \Tilde{K}>0\ \mbox{such that}\\
		& \Vert a_1 \Vert_{\mathrm{L}^{\frac{N}{2}}(\Omega)} \leq \Tilde{K}(1+t)^{\frac{N-2}{N-1}}. \\
		\mbox{For dimension} & \ N=1,2,3, \ \exists \Tilde{K}>0\ \mbox{such that}\\
		& \Vert a_1 \Vert_{\mathrm{L}^{\frac{3}{2}}(\Omega)} \leq \Tilde{K}(1+t)^{\frac{1}{2}}.
	\end{aligned}
	\right .
\end{equation*}

\vspace{.2cm}

Step 6: We now arrive at the position where we can link dissipation with the missing deviation $\displaystyle{\Vert A_1-\overline{A_1}\Vert_{\mathrm{L}^2(\Omega)}}$. The results read as
\[
D(a_1,\cdots,a_m) \geq \hat{C}\left(\sum_1^{m}  \Vert A_i-\overline{A_i}\Vert_{\mathrm{L}^2(\Omega)}^2+\left\Vert A_m-A_1\displaystyle \prod_{j=2}^{m-1} (A_j)^{\alpha_j}\right\Vert_{\mathrm{L}^2(\Omega)}^2\right),
\]
where $\hat{C}$ is a positive constant depends on domain and initial conditions only. The above estimates also yields that the space-time growth of the solutions varies polynomially with time. Hence we can relate relative entropy with the dissipation and obtain the convergence results.

\vspace{.3cm}

For the degeneracy $d_m=0$, we relate the missing term $\displaystyle{\Vert A_m-\overline{A_m}\Vert_{\mathrm{L}^2(\Omega)}}$. The result we obtain holds up to dimension $N=2$. We present the two dimensional case only. The key points are described below

\vspace{.2cm}

Step 1: We obtain a uniform integral estimate for $d_m=0$ case. The result reads as
\[
\sup\limits_{t>0}\Vert a_i \Vert_{\mathrm{L}^p(\Omega)}\leq C_{0,p}; \qquad \forall i=1,\cdot\cdot,m-1, p\in[1,\infty),
\]
where $C_{0,p}$ is a positive constant depends on $p$ but independent of time.

\vspace{.2cm}

Step 2: Using the above estimate, we can link dissipation with the missing deviation $\displaystyle{\displaystyle{\Vert A_m-\overline{A_m}\Vert_{\mathrm{L}^2(\Omega)}}}$. The results read as
\[
D(a_1,\cdots,a_m) \geq \hat{C}\left(\sum_1^{m}  \Vert A_i-\overline{A_i}\Vert_{\mathrm{L}^2(\Omega)}^2+\left\Vert A_m-\displaystyle \prod_{j=1}^{m-1} (A_j)^{\alpha_j}\right\Vert_{\mathrm{L}^2(\Omega)}^2\right),
\]
where $\hat{C}$ is a positive constant depends on domain and initial conditions only. The above estimates also yields that the space-time growth of the solutions varies polynomially with time. Hence we can relate relative entropy with the dissipation and obtain the convergence results for dimension $N=2$.

\section{Large time behaviour: degenerate system $d_1=0$}\label{d1=0 convergence section}

This section deals with the degenerate system \eqref{degenerate 1} where the diffusion coefficient $d_1$ vanishes. Our ultimate goal is to employ the method of entropy to demonstrate that the solution to the degenerate system \eqref{degenerate 1} converges to the equilibrium states given by \eqref{triangular equilibrium state 1}-\eqref{triangular equibrium state 2}. While carrying out the procedure, information on the $\mathrm{L}^{\infty}(\Omega_t)$ norm of the solution comes in handy. So, a good portion of this section is spent on analyzing the supremum norm of the solution. To begin with, we record an $\mathrm{L}^2$ estimate on the diffusive species' concentrations in dimension $N\leq 3$. 

\vspace{.2cm}

\begin{Prop}\label{a_i to a_m L^9 entropy}
	Let dimension $N\leq 3$. Let $(a_1,\cdots,a_m)$ be the smooth positive solution to the degenerate triangular reaction-diffusion system $d_1=0$\eqref{degenerate 1}. Then there exists a positive constant $K_{I}$, independent of $\tau$, such that
	\[
	\Vert a_i(t,x) \Vert_{\mathrm{L}^2(\Omega_{\tau,\tau+1})} \leq {K_{I}}
	\]
	for all $i=2,\cdots,m$ and for all $\tau\geq 0$.
\end{Prop}

\begin{proof} Recall the relation between the entropy and the entropy dissipation:
	\[
	\frac{dE}{dt}(a_1,\cdots,a_m)=-D(a_1,\cdots,a_m).
	\]
	Integrating the above relation in time on $[0,T]$ yields
	\begin{align*}
		&\sup\limits_{t\in[0,T]}\sum\limits_{i=1}^m \int_{\Omega}  \alpha_i(a_i(\ln{a_i}-1)+1)+ \sum\limits_{i=2}^{m}\int_{\Omega_T} \alpha_i d_i \frac{\vert \nabla a_i \vert^2}{a_i}\nonumber \\ & +\int_{\Omega_T}   \Big(a_m-a_1\displaystyle \prod\limits_{j=2}^{m-1} (a_j)^{\alpha_j}\Big)\ln{\left(\frac{a_m}{a_1\prod\limits_{j=2}^{m-1} (a_j)^{\alpha_j}}\right)}\leq E(a_{1,0},\cdots,a_{m,0}). 
	\end{align*}   
	
	The first two terms of the above inequality immediately yields the following bounds
	
	\begin{align}\label{mass bound}
		\int_{\Omega}a_i \leq \frac{1}{\min\limits_{1\leq i\leq m}\{\alpha_i\}}E(a_{1,0},\cdots,a_{m,0})+e^2\vert \Omega \vert=\Tilde{M}_1 \quad \mbox{for}\ i=1,\cdots,m
	\end{align}
	and 
	\begin{align}\label{Gradient bound}
		\int_{\Omega_T}\vert \nabla \sqrt{a_i}\vert^2\leq \frac{1}{\min\limits_{i=2,\cdots,m}\{d_i\alpha_i\}}E(a_{1,0},\cdots,a_{m,0})=K_1 \qquad \mbox{for}\ i=2,\cdots,m.
	\end{align}
	Employing the Gagliardo-Nirenberg interpolation inequality for $\sqrt{a_i}$ in dimension $N=1$, we obtain 
	\[
	\Vert \sqrt{a_i} \Vert^4_{\mathrm{L}^4(\Omega)} \leq C_{GN}\Vert \nabla \sqrt{a_i} \Vert_{\mathrm{L}^2(\Omega)}\Vert \sqrt{ a_i} \Vert^3_{\mathrm{L}^2(\Omega)}+C_{GN} \sqrt{ a_i} \Vert^4_{\mathrm{L}^2(\Omega)} 
	\]
	
	Integrating the above inequality in time on $[\tau,\tau+1]$ yields
	\[
	\int_{\tau}^{\tau+1} \int_{\Omega}a_i^2(t,x) \rm{d}x\rm{d}t \leq C_{GN}\left(\Tilde{M}_1^{\frac{3}{2}}K_1^{\frac12}+\Tilde{M}_1^2\right),
	\]

	where $K_1$ is defined in \eqref{Gradient bound}. Note that the above estimate holds true for all diffusive species.

	Following a similar strategy via the Gagliardo-Nirenberg interpolation inequality for $\sqrt{a_i}$ in dimension $N=2$, we obtain
	\[
	\int_{\tau}^{\tau+1} \int_{\Omega}a_i^2(t,x) \rm{d}x\rm{d}t \leq C_{GN}\left( K_1\Tilde{M}_1+\Tilde{M}_1^2\right)
	\] 
	for the same constants $K_1$ from \eqref{Gradient bound}.
	
	For dimension $N=3$, we borrow a result from \cite[Lemma 3.2]{FMT20} which gives the following bound
	
	\begin{align}\label{New L2}
		\int_{\Omega_{\tau,\tau+2}} a_i^2 \leq \sqrt{\Tilde{M}_1}\left( \sqrt{2}\frac{d_{max}}{d_{min}}\right)^{\frac{2}{3}}
	\end{align}
	for all $\tau\geq 0$ and for $i=2,\cdots,m$. Here, the constant $\Tilde{M}_1$ is defined in \eqref{mass bound} and $\displaystyle{d_{max}=\max\{d_2,\cdots,d_m\}}$ and $\displaystyle{d_{min}=\min\{d_2,\cdots,d_m\}}$.
	
	Let us choose a positive constant $K_{I}$ to be
	\[
	K_{I}=\max \left \{ \Tilde{M}_1^{\frac{3}{2}}K_1^{\frac12}+\Tilde{M}_1^2, M_1^4 C_{GN}\Vert \nabla \sqrt{a_i} \Vert^2_{\mathrm{L}^2(\Omega)}+C_{GN}M_1^3, \sqrt{\Tilde{M}_1}\left( \sqrt{2}\frac{d_{max}}{d_{min}}\right)^{\frac{2}{3}} \right\}
	\]
	
	This provides essential uniform bound for $\Vert a_i\Vert_{\mathrm{L}^2(\Omega_{\tau,\tau+1})}$ for $i=2,\cdots,m$.
\end{proof}

\begin{Rem}\label{L^3 dimension 1}
	For dimension $N=1$, we can further improve the integrability of $a_i$ via Gagliardo-Nirenberg interpolation. More precisely, we have for $i=2,\cdots,m$,
	\[
	\Vert \sqrt{a_i}\Vert^6_{\mathrm{L}^6(\Omega)} \leq C_{GN}\Vert \nabla \sqrt{a_i} \Vert^2_{\mathrm{L}^2(\Omega)}\Vert \sqrt{ a_i} \Vert^4_{\mathrm{L}^2(\Omega)}+C_{GN}\Vert \sqrt{ a_i} \Vert^6_{\mathrm{L}^2(\Omega)}\leq \Tilde{M}_1^4 C_{GN}\Vert \nabla \sqrt{a_i} \Vert^2_{\mathrm{L}^2(\Omega)}+C_{GN}\Tilde{M}_1^3,
	\]
	This helps us arrive at 
	\[
	\int_{\tau}^{\tau+1} \int_{\Omega} a_i^3(t,x) \rm{d}x\rm{d}t \leq C_{GN}\left( K_1\Tilde{M}_1^2+\Tilde{M}_1^3\right).
	\]
\end{Rem}

\begin{Rem}
	Even though we recall the estimate \eqref{New L2} for dimension $N=3$, the author in \cite{FMT20} derive a uniform $\mathrm{L}^2(\Omega_{\tau,\tau+2})$ estimate for dimension $N\geq 3$. They use duality approach.
\end{Rem}

\vspace{.2cm}

The following result improves the integrability of the species concentrations $a_2,\cdots,a_{m-1}$ to $\mathrm{L}^9(\Omega_{0,1})$. Note that this is for a parabolic cylinder of unit height at initial time.

\vspace{.2cm}

\begin{Prop}\label{a_i to a_m L^9 [0,1]}
	Let dimension $N\leq 3$. Let $(a_1,\cdots,a_m)$ be the smooth positive solution to the degenerate triangular reaction-diffusion system \eqref{degenerate 1}. Then there exists a positive constant $K_{0,1}$, independent of time, such that
	\[
	\Vert a_i(t,x) \Vert_{\mathrm{L}^9(\Omega_{0,1})} \leq {K_{0,1}} \qquad
	\quad i=2,\cdots,m-1.
	\]
\end{Prop}

\begin{proof} Thanks to the non-negativity of solution, for all $2\leq i\leq m$, the species $a_i$ is a positive subsolution of the following equation in the time interval $(0,2)$.
	\begin{equation}\label{eq:tilde a}
		\left \{
		\begin{aligned}
			\partial_{t}a_i(t,x)-d_i\Delta a_i(t,x)
			\leq &a_m \qquad \qquad \qquad  \mbox{ in }\Omega_{0,2}
			\\
			\nabla_x a_i(t,x).\gamma=&  0 \qquad   \qquad \quad   \ \  \ \mbox{ on } \partial\Omega_{0,2} 
			\\
			a_i(0,x) =& a_{i,0} \qquad   \quad  \qquad \ \ \ \mbox{ in } \Omega.
		\end{aligned}
		\right.
	\end{equation}
	The solution corresponding to the equation \eqref{eq:tilde a} can be expressed as:
	\[
	a_i(t,x)\leq \int_{\Omega}G_{d_a}(t,0,x,y)a_{i,0}(y) \rm{d}y+ \int_{0}^{t}\int_{\Omega} G_{d_i}(t,s,x,y)a_m(s,y)\rm{d}y\rm{d}s, \quad (t,x)\in\Omega_{0,2}.
	\]
	where $G_{d_i}$ denotes the Green's function associated with the operator $\partial_t-d_i\Delta$ with Neumann boundary condition. We use the fact that $\displaystyle{a_{i,0}(y)\in\mathrm{L}^{\infty}(\Omega)}$ and $\displaystyle{\int_{\Omega}G_{d_i}(t,0,x,y) \rm{d}y \leq 1}$ for all $t\in(0,2)$. It yields
	\[
	\left\vert \int_{\Omega}G_{d_a}(t,0,x,y)a_{i,0}(y) \rm{d}y \right\vert \leq \Vert a_0\Vert_{\mathrm{L}^{\infty}(\Omega)}.
	\]
	The species $a_i$ can be pointwise estimated as
	\[
	a_i\leq \Vert a_0\Vert_{\mathrm{L}^{\infty}(\Omega)} +\int_{0}^{t}\int_{\Omega} G_{d_i}(t,s,x,y)a_m(s,y)\rm{d}y\rm{d}s.
	\]
	Triangular inequality of the $\mathrm{L}^9$ norm yields:
	\[
		\Vert a_i\Vert_{\mathrm{L}^9(\Omega_{0,1})} \leq \Vert a_0\Vert_{\mathrm{L}^{\infty}(\Omega)} \vert \Omega\vert^{\frac{1}{9}}+\left\Vert \int_{0}^{t}\int_{\Omega} G_{d_i}(t,s,x,y)a_m(s,y)\rm{d}y\rm{d}s \right\Vert_ {\mathrm{L}^9(\Omega_{0,1})}.
	\]
	Green function estimate as in \eqref{Heat kernel estimate} and \eqref{Heat kernel bound function} yields
	\[
	0\leq G_{d_i}(t_1,s,x,y) \leq  g_{d_i}(t_1-s,x-y) \qquad 0\leq s<t_1.
	\]
	 Let us consider the following two functions
\begin{equation*}
	\Tilde{g}_{d_i}(s,x)=
	\left \{
	\begin{aligned} 
		&g_{d_i}(s,x)  \qquad (s,x)\in (-1,1)\times \mathbb{R}^N,\\
		& 0  \qquad \qquad \qquad \mbox{otherwise}
	\end{aligned}
\right .
\end{equation*}
and
\begin{equation*}
	\Tilde{a}_m(s,x)=
	\left \{
	\begin{aligned} 
		&a_m(s,x)  \qquad (s,x)\in (0,1)\times \Omega\\
		& 0  \qquad \qquad \qquad \mbox{otherwise}.
	\end{aligned}
	\right .
\end{equation*} 
Hence 
\[
\int_{0}^{t}\int_{\Omega} G_{d_i}(t,s,x,y)a_m(s,y)\rm{d}y\rm{d}s \leq \int_{\mathbb{R}}\int_{\mathbb{R}^N} \Tilde{g}_{d_i}(t-s,x-y) \tilde{a}_m(s,y) \rm{d}y \rm{d}s.
\]
We use Young's convolution inequality. It yields
\begin{align*}
\left\Vert \int_{0}^{t}\int_{\Omega} G_{d_i}(t,s,x,y)a_m(s,y)\rm{d}y\rm{d}s \right\Vert_ {\mathrm{L}^9(\Omega_{0,1})} &\leq \left \Vert \int_{\mathbb{R}}\int_{\mathbb{R}^N} \Tilde{g}_{d_i}(t-s,x-y) \tilde{a}_m(s,y) \rm{d}y \rm{d}s \right\Vert_ {\mathrm{L}^9(\mathbb{R}\times\mathbb{R}^N)}\\ & \leq \Big\Vert \Tilde{g}_{d_i} \Big\Vert_{\mathrm{L}^{\frac{18}{11}}(\mathbb{R}\times\mathbb{R}^N)}
	\Big\Vert \tilde{a}_m \Big\Vert_{\mathrm{L}^2(\mathbb{R}\times\mathbb{R}^N)}
\end{align*}
as $\displaystyle{1+\frac{1}{9}=\frac{1}{2}+\frac{11}{18}}$. Hence
\[
\left\Vert \int_{0}^{t}\int_{\Omega} G_{d_i}(t,s,x,y)a_m(s,y)\rm{d}y\rm{d}s \right\Vert_{\mathrm{L}^9(\Omega_{0,1})} \leq \Big\Vert g_{d_i} \Big\Vert_{\mathrm{L}^{\frac{18}{11}}((-1,1)\times\mathbb{R}^N)}
\Big\Vert a_m \Big\Vert_{\mathrm{L}^2(\Omega_{0,1}}. 
\]
Furthermore \eqref{Heat kernel integral estimate} yields the following integral estimate:
\[
\left\Vert g_{d_i}\right\Vert_{\mathrm{L}^{z}((-2,2)\times\mathbb{R}^N)} \leq \tilde{\kappa}_{1,d_i} \qquad \forall z\in\left[1,1+\frac{2}{N}\right), N=1,2,3,
\]
for some constant positive $\tilde{\kappa}_{1,d_i}$ depending on $z$. We in particular choose $z=\frac{18}{11}$ which is admissible because $\frac{18}{11}< 1+\frac{2}{N}$ for all $N=1,2,3$. Hence 
	\[
	\left\Vert \int_{0}^{t}\int_{\Omega} G_{d_i}(t,s,x,y)a_m(s,y)\rm{d}y\rm{d}s \right\Vert_{\mathrm{L}^9(\Omega_{0,1})} \leq \tilde{\kappa}_{1,d_i} K_{I}
	\]
	where $K_I$ as in Proposition \ref{a_i to a_m L^9 entropy}.
 Defining the constant
	\[
	\displaystyle{K_{0,1}:=\max\limits_{2\leq i\leq m}\{\tilde{\kappa}_{1,d_i}\}K_I+\max\limits_{1\leq i\leq m}\left\{\Vert a_{i,0} \Vert_{\mathrm{L}^9(\Omega_{0,1})}\right\}},
	\]
	we arrive at 
	\[
	\Vert a_i(t,x) \Vert_{\mathrm{L}^9(\Omega_{0,1})} \leq {K_{0,1}},
	\ \qquad i=2,\cdots,m-1.
	\]
\end{proof}
\vspace{.2cm}

The following result extends this $\mathrm{L}^9$ estimate to any parabolic cylinder with unit height. 

\begin{Prop}\label{a_i to a_m L^9}
	Let dimension $N\leq 3$. Let $(a_1,\cdots,a_m)$ be the smooth positive solution to the degenerate triangular reaction-diffusion system \eqref{degenerate 1}. Then, there exists a positive constant $K_{\Gamma}$, independent of $\tau$, such that
	\[
	\Vert a_i(t,x) \Vert_{\mathrm{L}^9(\Omega_{\tau,\tau+1})} \leq {K_{\Gamma}}
	\]
	for all $\tau\geq 0$ and for $i=2,\cdots,m-1$.
\end{Prop}

\begin{proof} Consider the function $\phi_{\tau}:\mathbb{R}\to\mathbb{R}$ as defined in \eqref{cut-off}. For $2\leq i\leq m-1$, the product $\phi_{\tau}a_i$ satisfies
	
	\begin{equation} \nonumber
		\left\{
		\begin{aligned}
			\partial_t \phi_{\tau}a_i-d_i \Delta \phi_{\tau}a_i = & \phi'_{\tau}a_i+\phi_{\tau}\left( a_m-a_1\prod\limits_{j=2}^{m-1}a_j^{\alpha_j} \right) \qquad \mbox{ in } \Omega_{\tau,\tau+2}\\
			\nabla_x \phi_{\tau}a_i .\gamma=& 0 \qquad \qquad \qquad \qquad \qquad \qquad \qquad \ \   \mbox { on } \partial\Omega_{\tau,\tau+2}\\
			\phi_{\tau}a_i(\tau,x)=& 0 \qquad \qquad \qquad \ \qquad \qquad \qquad \ \  \qquad \mbox{ in }\Omega.
		\end{aligned}
		\right.
	\end{equation}

	Change of variable $t\rightarrow t_1+\tau$ yields
	
	\begin{equation*}
		\left \{
		\begin{aligned}
			\partial_{t_1}(\phi_\tau(t_1+\tau)&a_i(t_1+\tau,x))-d_{i}\Delta (\phi_\tau(t_1+\tau)a_i(t_1+\tau,x))\\ &= a_{i}(t_1+\tau,x)\phi^{'}_\tau(t_1+\tau)+\phi_\tau(t_1+\tau)\left(a_m-a_1\prod\limits_{j=2}^{m-1}a_j^{\alpha_j}\right) \  \mbox{\ in \ }\Omega_{0,2} \\
			& \nabla_x (\phi_\tau(t_1+\tau)a_i(t_1+\tau,x)).\gamma=  0 \qquad   \qquad \qquad \qquad \qquad \quad  \  \mbox{ on } \partial\Omega_{0,2} \\
			& \phi_{\tau}a_i(0+\tau,x)= 0 \qquad \qquad \qquad \qquad \qquad \ \ \ \qquad  \qquad \quad \ \quad \ \ \mbox{ in } \Omega. 
		\end{aligned}
		\right .
	\end{equation*}
	
	Let $G_i(t_1,s,x,y)$ be the Green's function corresponding to the operator $\partial_{t_1}- d_i \Delta$.  Then for all $t_1\in[0,2)$, we can represent the solution as follows
	\begin{align*}
		&\phi_\tau(t_1+\tau)a_i(t_1+\tau,x)\\ &= \int_{0}^{t_1}\int_{\Omega}G_i(t_1,s,x,y)\left(a_{i}(y,s+\tau)\phi^{'}_\tau(s+\tau)+\phi_\tau(s+\tau)(a_m-a_1\prod\limits_{j=2}^{m-1}a_j^{\alpha_j})(s+\tau,t)\right).
	\end{align*}
	Thanks to the non-negativity of the Green's function, the following is immediate:
	\[
	\phi_\tau(t_1+\tau)a_i(t_1+\tau,x)\leq \int_{0}^{t_1}\int_{\Omega}G_i(t_1,s,x,y)\Big(a_{i}(s+\tau,x)\phi^{'}_\tau(s+\tau)+\phi_\tau(s+\tau)a_m(s+\tau,x)\Big),
	\]
	where we have also used the fact that the solution is non-negative.
	
Green function estimate as in \eqref{Heat kernel estimate} and \eqref{Heat kernel bound function} yields
\[
0\leq G_{d_i}(t_1,s,x,y) \leq  g_{d_i}(t_1-s,x-y) \qquad 0\leq s<t_1
\]
where the function $g_{d_i}$ as defined in  \eqref{Heat kernel bound function}. Substituting this pointwise bound for the Green's function in the earlier inequality leads to
	\[
	\phi_\tau(t_1+\tau)a_i(t_1+\tau,x)\leq \int_{0}^{t_1}\int_{\Omega}g_{d_i}(t_1-s,x,y)\Big(a_{i}(s+\tau,x)\phi^{'}_\tau(s+\tau)+\phi_\tau(s+\tau)a_m(s+\tau,x)\Big).
	\]
	 Let us consider the following two functions
\begin{equation*}
	\Tilde{g}_{d_i}(s,x):=
	\left \{
	\begin{aligned} 
		&g_{d_i}(s,x)  \qquad (s,x)\in (-1,1)\times \mathbb{R}^N,\\
		& 0  \qquad \qquad \qquad \mbox{otherwise}
	\end{aligned}
	\right .
\end{equation*}
and
\begin{equation*}
	F(s,x):=
	\left \{
	\begin{aligned} 
		a_{i}(s+\tau,x)\phi^{'}_\tau(s+\tau)&+\phi_\tau(s+\tau)a_m(s+\tau,x)  \qquad (s,x)\in (0,1)\times \Omega\\
		& 0  \qquad \qquad \qquad \qquad \qquad \qquad \ \ \ \mbox{otherwise}.
	\end{aligned}
	\right .
\end{equation*} 
Hence 
\[
\int_{0}^{t}\int_{\Omega} G_{d_i}(t,s,x,y)a_i(s,y)\rm{d}y\rm{d}s \leq \int_{\mathbb{R}}\int_{\mathbb{R}^N} \Tilde{g}_{d_i}(t-s,x-y) F(s,y) \rm{d}y \rm{d}s.
\]
We use Young's convolution inequality. It yields
\begin{align*}
	\left\Vert \int_{0}^{t}\int_{\Omega} G_{d_i}(t,s,x,y)a_i(s,y)\rm{d}y\rm{d}s \right\Vert_ {\mathrm{L}^9(\Omega_{0,1})} &\leq \left \Vert \int_{\mathbb{R}}\int_{\mathbb{R}^N} \Tilde{g}_{d_i}(t-s,x-y) F(s,y) \rm{d}y \rm{d}s \right\Vert_ {\mathrm{L}^9(\mathbb{R}\times\mathbb{R}^N)}\\ & \leq \Big\Vert \Tilde{g}_{d_i} \Big\Vert_{\mathrm{L}^{\frac{18}{11}}(\mathbb{R}\times\mathbb{R}^N)}
	\Big\Vert F \Big\Vert_{\mathrm{L}^2(\mathbb{R}\times\mathbb{R}^N)}
\end{align*}
as $\displaystyle{1+\frac{1}{9}=\frac{1}{2}+\frac{11}{18}}$. Hence
\[
\left\Vert \int_{0}^{t}\int_{\Omega} G_{d_i}(t,s,x,y)a_m(s,y)\rm{d}y\rm{d}s \right\Vert_{\mathrm{L}^9(\Omega_{0,1})} \leq \Big\Vert g_{d_i} \Big\Vert_{\mathrm{L}^{\frac{18}{11}}((-1,1)\times\mathbb{R}^N)}
\Big\Vert a_{i}(\cdot+\tau,\cdot)\phi^{'}_\tau(\cdot+\tau)+\phi_\tau(\cdot+\tau)a_m(\cdot+\tau,\cdot) \Big\Vert_{\mathrm{L}^2(\Omega_{0,1})}. 
\]	
 Furthermore \eqref{Heat kernel integral estimate} yields the following integral estimate:
	\[
	\left\Vert g_{d_i}\right\Vert_{\mathrm{L}^{z}((-2,2)\times\mathbb{R}^N)} \leq \tilde{\kappa}_{1,d_i} \qquad \forall z\in\left[1,1+\frac{2}{N}\right), N=1,2,3,
	\]
	for some constant positive $\tilde{\kappa}_{1,d_i}$ depending on $z$. We in particular choose $z=\frac{18}{11}$ which is admissible because $\frac{18}{11}< 1+\frac{2}{N}$ for all $N=1,2,3$. Thanks to the $\mathrm{L}^2$ estimate from Proposition \ref{a_i to a_m L^9 entropy}, we arrive at 
	\[
	\Vert \phi_\tau(t_1+\tau)a_i(t_1+\tau,x) \Vert_{\mathrm{L}^9(\Omega_{0,2})}\leq 2\tilde{\kappa}_{1,d_i}(1+M_{\phi})K_I, \qquad i=2,\cdots,m-1.
	\]
	
	Substituting back $t_1+\tau=t$ yields  
	\begin{align*} \label{L^9 estimate of a_i}
		\Vert a_i \Vert_{\mathrm{L}^9(\Omega_{\tau,\tau+1})} \leq 2\max\limits_{2\leq i\leq m}\{\tilde{\kappa}_{1,d_i}\}(1+M_{\phi})K_I+ \Vert a_i\Vert_{\mathrm{L}^9(\Omega_{0,1})} \leq 2\max\limits_{2\leq i\leq m}\{\tilde{\kappa}_{1,d_i}\}(1+M_{\phi})K_I+K_{0,1} \qquad i=2,\cdots,m-1,
	\end{align*}
	thanks to Proposition \ref{a_i to a_m L^9 [0,1]}. Defining the constant $K_{\Gamma}$ as
	\[
	K_{\Gamma}:= 2\max\limits_{2\leq i\leq m}\{\tilde{\kappa}_{1,d_i}\}(1+M_{\phi})K_I+ \Vert a_i\Vert_{\mathrm{L}^9(\Omega_{0,1})} \leq 2\max\limits_{2\leq i\leq m}\{\tilde{\kappa}_{1,d_i}\}(1+M_{\phi})K_I+K_{0,1},
	\]
	we have proved the result.
\end{proof}

So far we have managed to obtain uniform (in $\tau$) $\mathrm{L}^p(\Omega_{\tau,\tau+1})$ estimate on the species concentration $a_2,\cdots,a_{m-1}$. Furthermore, our results have been valid for spatial dimension $N\leq 3$. Next we prove a conditional result in arbitrary spatial that derives an $\Vert a_m\Vert_{\mathrm{L}^p(\Omega_{\tau,\tau+1})}$ estimate with $p>N$ on the species concentration provided the same $\mathrm{L}^p$ estimates were available for the species concentrations $a_2,\cdots,a_{m-1}$.  We emphasize here that we have access to such estimate for $a_2,\cdots,a_{m-1}$ only when $N\leq 3$. Hence for $N>3$ the following is a conditional result.

\vspace{.2cm}

\begin{Prop}\label{a_i to a_m L^p}
	Let $(a_1,\cdots,a_m)$ be the smooth positive solution to the degenerate triangular reaction-diffusion system \eqref{degenerate 1}. Suppose there  exists some positive constant $\texttt{K}$, independent of $\tau$, such that
	\[
	\Vert a_i(t,x) \Vert_{\mathrm{L}^p(\Omega_{\tau,\tau+1})} \leq \texttt{K},  \text{ \ for\ some\ } i\in\{2,\cdots,m-1\}, \forall \tau \geq 0,
	\]
	for some exponent $p>N$. Then there exists positive constant $\textbf{K}$, such that
	\[
	\Vert a_m(t,x) \Vert_{\mathrm{L}^p(\Omega_{\tau,\tau+1})} \leq \textbf{K} \ \qquad  \mbox{for all}\ \tau\geq0
	\]
	and for the same exponent $p>N$ as above.
\end{Prop}

\begin{proof} Let $\theta \in \mathrm L^{p'}(\Omega_{\tau,\tau+2})$ be a non-negative function with $p'$ being the H\"older conjugate of $p$. Let $\psi$ be the solution of the following backward heat equation:
	\begin{equation}\label{new label:psi}
		\left \{
		\begin{aligned}
			\partial_t \psi(t,x) +d_m\Delta \psi_(t,x)& = -\theta(t,x) \qquad \mbox{in}\ \Omega_{\tau,\tau+2} \\
			\nabla_x \psi_(t,x).\gamma & =0 \qquad \ \ \qquad \mbox{on} \ \partial \Omega_{\tau,\tau+2} 
			\\
			\psi(\tau+2,x) & =0 \qquad \qquad \ \ \  \ \mbox{in} \ \Omega. 
		\end{aligned}
		\right .
	\end{equation}
	
	From the second order regularity result (see-Appendix, Theorem \ref{estimation 1}) says
	\[
	\Vert \Delta \psi \Vert_{\mathrm{L}^{p'}(\Omega _{\tau,\tau+2})} \leq \frac{ C_{SOR}( \Omega,N,p)}{d_m} \Vert \theta \Vert_{\mathrm{L}^{p'}(\Omega_{\tau,\tau+2})}. 
	\]
	Note that $\displaystyle{p'= \frac{p}{p-1}< \frac{N+2}{2}}$. Employing integrability estimate (see-Appendix, Theorem \ref{estimation 1}) we obtain the following estimate:
	\[
	\Vert \psi_m \Vert_{\mathrm{L}^q(\Omega _{\tau,\tau+2})} \leq C_{IE}(\Omega,d_m,p,s) \Vert \theta \Vert_{\mathrm{L}^{p'}(\Omega_{\tau,\tau+2})} \qquad   \forall q< \frac{(N+2)p'}{N+2-2p'}.
	\]
	We remark that the constants  $C_{SOR}(\Omega,N,p)$ and  $C_{IE}(\Omega,d,p,s)$  are independent of  $\tau$. Consider the function $\phi_{\tau}:\mathbb{R}\to\mathbb{R}$ as in \eqref{cut-off}.	Using the equation \eqref{new label:psi}, we obtain
	\begin{align*}
		\int_{\tau}^{\tau+2} \int_{\Omega}\phi_\tau(t)a_m(t,x)\theta(t,x)= - \int_{\tau}^{\tau+2} \int_{\Omega}\phi_\tau(t)a_m(t,x)(\partial_t \psi(t,x) +d_m\Delta \psi(t,x)).
	\end{align*}
	Integration-by-parts in the term on the right hand side yields
	\begin{align*}	
		\int_{\tau}^{\tau+2} \int_{\Omega}&\phi_\tau(t)a_m(t,x)\theta(t,x)=	\int_{\tau}^{\tau+2} \int_{\Omega}\psi(t,x)\left(\partial_t(\phi_\tau(t)a_m(t,x))-d_m\Delta(\phi_\tau(t)a_m(t,x))\right)\\
		= & \int_{\tau}^{\tau+2} \int_{\Omega}\psi(t,x)\left(\phi'(t)a_m(t,x)-\phi_\tau\left(a_m(t,x)-a_1\prod\limits_{j=2}^{m-1}a_j^{\alpha_j}\right)\right)\\
		= & \int_{\tau}^{\tau+2} \int_{\Omega}\psi(t,x)\Big(\phi_\tau^{'}(t)(a_m(t,x)+a_i(t,x))-(\partial_t (\phi_\tau a_i) -d_i \Delta (\phi_\tau a_i))\Big)\\
		= & \int_{\tau}^{\tau+2} \int_{\Omega}\psi(t,x)\Big(\phi_\tau^{'}(t)(a_m(t,x)+a_i(t,x))\Big)\\
		&\qquad \quad +\int_{\tau}^{\tau+2} \int_{\Omega}\phi_\tau(t)a_i(t,x)\Big(\partial_t \psi(t,x)+d_i \Delta \psi(t,x)\Big)\\
		=& \int_{\tau}^{\tau+2} \int_{\Omega}\psi(t,x)(\phi_\tau^{'}(t)(a_i(t,x)+a_m(t,x)))+\int_{\tau}^{\tau+2} \int_{\Omega} -\phi_\tau(t)a_i(t,x)\theta(t,x)\\
		+& (-d_m+d_i)\int_{\tau}^{\tau+2} \int_{\Omega}\phi_\tau(t)a_i(t,x)  \Delta \psi(t,x)=: \textbf{I}+\textbf{II}+\textbf{III}.
	\end{align*}

	Thanks to the  H\"older inequality, we can  estimate the term  \textbf{II} as follows:
	\[
	\vert \textbf{II} \vert \leq \Vert \phi_{\tau}a_i \Vert_{\mathrm{L}^p(\Omega_{\tau,\tau+2})} \Vert \theta \Vert_{\mathrm{L}^{p'}(\Omega_{\tau,\tau+2})}\leq \Vert a_i \Vert_{\mathrm{L}^p(\Omega_{\tau,\tau+2})} \Vert \theta \Vert_{\mathrm{L}^{p'}(\Omega_{\tau,\tau+2})}\leq \texttt{K}\Vert \theta \Vert_{\mathrm{L}^{p'}(\Omega_{\tau,\tau+2})},
	\]
	where we use the hypothesis that $a_i$ is bounded in $\mathrm{L}^p$.
	Similarly we can estimate \textbf{III} using the H\"older inequality as follows:
	\begin{align*}
		\vert \textbf{III} \vert \leq & \vert d_m-d_i \vert \ \Vert \phi_{\tau}a_i \Vert_{\mathrm{L}^p(\Omega_{\tau,\tau+2})} \Vert \Delta \psi \Vert_{\mathrm{L}^{p'}(\Omega_{\tau,\tau+2})}
		\\
		\leq & \frac{C_{SOR}(\Omega,N,p)}{d_m}\vert d_m-d_i \vert\Vert a_i \Vert_{\mathrm{L}^p(\Omega_{\tau,\tau+2})} \Vert \theta \Vert_{\mathrm{L}^{p'}(\Omega_{\tau,\tau+2})}
		\\
		\leq& \frac{C_{SOR}(\Omega,N,p)}{d_m}\texttt{K}\vert d_m-d_i \vert\Vert \theta \Vert_{\mathrm{L}^{p'}(\Omega_{\tau,\tau+2})}.
	\end{align*}
	
	Observe that $\displaystyle{p^{'}<\frac{(N+2)p^{'}}{N+2-2p^{'}}}$. Hence we can choose a $q$ such that  $\displaystyle{p^{'}<q< \frac{(N+2)p^{'}}{N+2-2p^{'}}}$. Note further that the it's H\"older conjugate  $q'$ of $q$  satisfies  $1<q'<p$.

	From Theorem \ref{estimation 1} (see- Appendix), we have the integrability estimation:
	\[
	\Vert  \psi \Vert_{\mathrm{L}^q(\Omega _{\tau,\tau+2})} < C_{IE}( \Omega,d_m,p',q) \Vert \theta \Vert_{\mathrm{L}^{p'}(\Omega_{\tau,\tau+2})}.
	\]

	Note that the cut-off function satisfies  
	\[
	\phi^{'}_{\tau}(s)|_{[\tau+1,\tau+2]}=0 \mbox{\ and\ }   \phi^{'}_{\tau}(s)|_{[\tau,\tau+1]} \leq M_{\phi}.
	\]
	
	Thanks to the H\"older inequality, the following estimate holds
	
	\begin{align*}
		\vert \textbf{I}\vert \leq &  M_{\phi} \Vert \psi \Vert_{\mathrm{L}^{q}(\Omega_{\tau,\tau+2})}\Vert  a_i+a_m \Vert_{\mathrm{L}^{q'}(\Omega_{\tau,\tau+1})} %\label{2}
		\\
		\leq & M_{\phi} C_{IE}( \Omega,d_m,N,p^{'},q) \Vert \theta \Vert_{\mathrm{L}^{p'}(\Omega_{\tau,\tau+2})}\Vert  a_i+a_m \Vert_{\mathrm{L}^{q'}(\Omega_{\tau,\tau+1})} \nonumber
	\end{align*}
	
	Since $1<q'<p$, there exists a $\alpha \in (0,1)$ such that
	$$\frac{1}{q'}=\frac{1-\alpha}{1}+\frac{\alpha}{p}.$$
	
	Hence by interpolation, we obtain
	\begin{align*}
		\Vert  a_i+&a_m \Vert_{\mathrm{L}^{q'}(\Omega_{\tau,\tau+1})} \\ \leq & 
		\Vert  a_i+a_m \Vert^{1-\alpha}_{\mathrm{L}^{1}(\Omega_{\tau,\tau+1})}\Vert  a_i(t,x)+a_m(t,x) \Vert^{\alpha}_{\mathrm{L}^{p}(\Omega_{\tau,\tau+1})}.
	\end{align*}
	Thanks to the mass conservation and the hypothesis, we arrive at
	\[
	\vert\textbf{I}\vert  \leq M_{\phi} C_{IE}( \Omega,d_m,N,p^{'},q) \Vert \theta \Vert_{\mathrm{L}^{p'}(\Omega_{\tau,\tau+2})}2M_i^{1-\alpha}  (\texttt{K}^{\alpha}+\Vert  a_m \Vert^{\alpha}_{\mathrm{L}^{p}(\Omega_{\tau,\tau+1})}).
	\]
	Let's define two positive constants $K_6,K_7$ in the following way
	\begin{equation*}
		\left \{
		\begin{aligned}
			K_6= & \texttt{K}+\frac{C_{SOR}(\Omega,N,p)}{d_m}\texttt{K}\vert d_m-d_1 \vert+M_{\phi} C_{IE}( \Omega,N,p^{'},q) 2M_i^{1-\alpha}\texttt{K}^{\alpha},
			\\
			K_7= & M_{\phi} C_{IE}( \Omega,d_m,N,p^{'},q) 2M_i^{1-\alpha}.
		\end{aligned}
		\right .
	\end{equation*}
	Putting all the estimates together, we obtain
	\[
	\bigg\vert \int_{\tau}^{\tau+2} \int_{\Omega}\phi_\tau(t)a_m(t,x)\theta(t,x)\bigg\vert \leq \big(K_6+K_7\Vert  a_m(t,x) \Vert^{\alpha}_{\mathrm{L}^{p}(\Omega_{\tau,\tau+1})}\big)\Vert \theta \Vert_{L^{p'}(\Omega_{\tau,\tau+2})}.
	\]
	By duality, we have
	\[
	\Vert \phi_\tau a_m \Vert_{\mathrm{L}^{p}(\Omega_{\tau,\tau+2})}\leq K_6+K_7\Vert  a_m \Vert^{\alpha}_{\mathrm{L}^{p}(\Omega_{\tau,\tau+1})}.
	\]
	
	This further implies, thanks to the property of $\phi_{\tau}$,
	
	\[
	\Vert a_m \Vert_{\mathrm{L}^{p}(\Omega_{\tau+1,\tau+2})}\leq K_6+K_7\Vert  a_m \Vert^{\alpha}_{\mathrm{L}^{p}(\Omega_{\tau,\tau+1})}.
	\]
	
	Define a sequence $\{\beta_n\}|$ as $\beta_n=\Vert a_m \Vert_{\mathrm{L}^{p}(\Omega_{n,n+1})}$, then we have
	\begin{align}\label{d_i, L^p a_m relation 1}
		\beta_{n+1} \leq & K_7 \beta_n^{\alpha}+K_6.
	\end{align}
	
	If $\beta_{n+1}>\beta_{n}$, Young's inequality yields \[
	K_7\beta_{n+1}^{\alpha}\leq \alpha \beta_{n+1}+(1-\alpha)K_7^{\frac{1}{1-\alpha}}.
	\]
	
	Substituting this in \eqref{d_i, L^p a_m relation 1}, we obtain the following bound
	
	\[
	\beta_{n+1} \leq K_7^{\frac{1}{1-\alpha}}+\frac{K_6}{1-\alpha}.
	\]
	Hence
	\begin{align}\label{d_i, L^p a_m relation 2}
		\beta_n \leq \max\left\{ K_7^{\frac{1}{1-\alpha}}+\frac{K_6}{1-\alpha},\beta_0\right\} \qquad \forall n\in\mathbb{N}.
	\end{align}
To compute $\beta_0$ explicitly,  we start with the following integral 
\[
\int_{0}^{2} \int_{\Omega}a_m(t,x)\theta(t,x)= - \int_{0}^{2} \int_{\Omega}a_m(t,x)(\partial_t \psi(t,x) +d_m\Delta \psi(t,x))
\]
where $\psi(t,x)$ is a solution of \eqref{new label:psi} in the cylinder $\Omega_{0,2}$.  	Integration-by-parts in the term on the right hand side yields
\begin{align*}	
	\int_{0}^{2} \int_{\Omega}&a_m(t,x)\theta(t,x)=	\int_{0}^{2} \int_{\Omega}\psi(t,x)\left(\partial_ta_m(t,x)-d_m\Delta a_m(t,x)\right)\\
	= &\int_{\Omega} a_{m,0}\psi(0,x)+\int_{0}^{1} \int_{\Omega}-\psi(t,x)\left(a_m(t,x)-a_1\prod\limits_{j=2}^{m-1}a_j^{\alpha_j}\right)\\
	= & \int_{\Omega} a_{m,0}\psi(0,x)+\int_{0}^{1} \int_{\Omega}-\psi(t,x)\Big((\partial_t  a_i -d_i \Delta  a_i)\Big)\\
	= & \int_{\Omega} (a_{m,0}-a_{i,0})\psi(0,x) +\int_{\tau}^{\tau+2} \int_{\Omega}a_i(t,x)\Big(\partial_t \psi(t,x)+d_i \Delta \psi(t,x)\Big)\\
	=&\int_{\Omega} (a_{m,0}-a_{i,0})\psi(0,x)+\int_{\tau}^{\tau+2} \int_{\Omega} -a_i(t,x)\theta(t,x)\\
	+ & (-d_m+d_i)\int_{\tau}^{\tau+2} \int_{\Omega}a_i(t,x)  \Delta \psi(t,x)
\end{align*}
Thanks to H\"older inequality, we obtain
\begin{align*}
	\int_{0}^{2} \int_{\Omega}a_m(t,x)\theta(t,x)\leq &(\Vert a_{m,0}\Vert_{\mathrm{L}^p(\Omega)}+ \Vert a_{i,0}\Vert_{\mathrm{L}^p(\Omega)}) \Vert \psi(0,\cdot)\Vert_{\mathrm{L}^q(\Omega)}+ \Vert a_i\Vert_{\mathrm{L}^p(\Omega_{0,2})} \Vert \theta\Vert_{\mathrm{L}^q(\Omega_{0,2})}\\
	+&\vert d_m-d_i\vert \Vert a_i\Vert_{\mathrm{L}^p(\Omega_{0,2})} \Vert \Delta \psi\Vert_{\mathrm{L}^q(\Omega_{0,2})}.
\end{align*}  
where $q$ is the H\"older conjugate of $p$. Regularity of backward heat equation (see Appendix, Theorem \ref{regularity of backward heat}) yileds
\[
\int_{0}^{2} \int_{\Omega}a_m(t,x)\theta(t,x)\leq C_q\left( (\Vert a_{m,0}\Vert_{\mathrm{L}^p(\Omega)}+ \Vert a_{i,0}\Vert_{\mathrm{L}^p(\Omega)})+ 2\texttt{K}+2\vert d_m-d_i\vert \texttt{K} \right) \Vert \theta\Vert_{\mathrm{L}^q(\Omega_{0,2})}
\]
where $C_q$ is a positive constant depends only on $q$ and the domain. We also used the fact that 
\[
\Vert a_i(t,x) \Vert_{\mathrm{L}^p(\Omega_{\tau,\tau+1})} \leq \texttt{K},  \text{ \ for\ some\ } i\in\{2,\cdots,m-1\}, \forall \tau \geq 0.
\]
Hence, by duality we estimate $\beta_0$:
\[
\beta_0 =\Vert a_m \Vert_{\mathrm{L}^p(\Omega_{0,1})} \leq C_q\left( (\Vert a_{m,0}\Vert_{\mathrm{L}^p(\Omega)}+ \Vert a_{i,0}\Vert_{\mathrm{L}^p(\Omega)})+ 2\texttt{K}+2\vert d_m-d_i\vert \texttt{K} \right).
\]
Let us denote the following constant 
\[
\textbf{K}:= \max\left\{ K_7^{\frac{1}{1-\alpha}}+\frac{K_6}{1-\alpha},C_q\left( \left(\Vert a_{m,0}\Vert_{\mathrm{L}^p(\Omega)}+ \max\limits_{1\leq i\leq m}\{\Vert a_{i,0}\Vert_{\mathrm{L}^p(\Omega)}\}\right)+ 2\texttt{K}+2\max\limits_{1\leq i\leq m}\{\vert d_m-d_i\vert\} \texttt{K} \right)\right\}
\]
Thanks to \eqref{d_i, L^p a_m relation 2}, we conclude 
	\[
	\Vert a_m \Vert_{\mathrm{L}^{p}(\Omega_{\tau,\tau+1})}\leq \textbf{K} \qquad \forall \tau \geq 0.
	\]
\end{proof}
\vspace{.4 cm}

\begin{Rem}
	From Proposition \ref{a_i to a_m L^9}, we have the following bounds when the dimension $N\leq 3$, 
	\begin{align}\label{3 dim a_m}
		\Vert a_i \Vert_{\mathrm{L}^9(\Omega_{\tau,\tau+1})} \leq  \textbf{K} \quad \mbox{for} \ i=2,\cdots,m-1.
	\end{align}
	Hence, thanks to Proposition \ref{a_i to a_m L^p}, we deduce in dimension $N\leq 3$ that $\Vert a_m\Vert_{\mathrm{L}^9(\Omega_{\tau,\tau+1})}$ is also uniformly bounded.
\end{Rem}

\vspace{.3cm}

The following result derives an uniform $\mathrm{L}^p(\Omega_{\tau,\tau+1})$ estimate with $p>N$, in dimension $N\geq 4$. This is, however, conditional under the closeness assumption \eqref{tri closeness condition imp} being satisfied by the non-zero diffusion coefficients $d_m$ and at least one of the $d_i$ with $i\in\{2,\dots,m-1\}$.

\begin{Prop}\label{L^n estimation of a_m}
	Let $(a_1,\cdots,a_m)$ be the smooth positive solution to the degenerate triangular reaction-diffusion system \eqref{degenerate 1}. Suppose further that the closeness condition\eqref{tri closeness condition imp} is satisfied by the diffusion coefficients $d_m$ and $d_i$ for at least one $i\in\{2,\dots,m-1\}$. Let $p>N$. Then, there exists a positive constant $C_0$, such that
	\[
	\Vert a_m \Vert_{\mathrm{L}^{p}(\Omega_{\tau,\tau+1})} \leq C_0
	\]
	for all $\tau\geq 0$.
\end{Prop}

\begin{proof} Consider the following equation with diffusion coefficient $d=\frac{d_i+d_m}{2}$ where the source term $\theta$ is non-negative and belongs to $ \mathrm L^{p'}(\Omega_{\tau,\tau+2})$, with $p'$ being the H\"older conjugate to $p$.
	\begin{equation}\label{new label psi1}
		\left \{
		\begin{aligned}
			\partial_t \psi(t,x) +d\Delta \psi(t,x)& = -\theta(t,x) \qquad \mbox{in}\  \Omega_{\tau,\tau+2} \nonumber\\
			\nabla_x \psi(t,x).\gamma & =0 \qquad \ \ \ \qquad \mbox{on}\  \partial\Omega_{\tau,\tau+2}  \nonumber\\
			\psi(\tau+2,x) & =0 \qquad \qquad \ \ \  \ \mbox{in} \ \Omega. \nonumber
		\end{aligned}
		\right .
	\end{equation}
	Note that the second order regularity estimation (see Theorem \ref{estimation 1}) yields
	\[
	\Vert \Delta \psi \Vert_{\mathrm{L}^{p'}(\Omega _{\tau,\tau+2})} \leq \frac{2 C_{SOR}( \Omega,N,p')}{d_i+d_m} \Vert \theta \Vert_{\mathrm{L}^{p'}(\Omega_{\tau,\tau+2})}. 
	\]
	
	Observe that $\displaystyle{p' = \frac{p}{p-1}< \frac{N+2}{2}}$. Furthermore, the integrability estimation (see Theorem \ref{estimation 1}) yields
	\[
	\Vert \psi \Vert_{\mathrm{L}^q(\Omega _{\tau,\tau+2})} \leq C_{IE}(\Omega,d,p',s) \Vert \theta \Vert_{\mathrm{L}^{p^{'}}(\Omega_{\tau,\tau+2})} \qquad   \forall q< \frac{(N+2)p^{'}}{N+2-2p'}.
	\]
	Here the constants  $C_{SOR}(\Omega,N,p)$ and  $C_{IE}(\Omega,d,p,s)$ are independent of  $\tau$. Consider the function $\phi_{\tau}:\mathbb{R}\to\mathbb{R}$ as defined in \eqref{cut-off}. Note that the product $\phi_{\tau}(t)a_i(t,x)$ satisfies the differential equation:
	\begin{gather}
		\partial_{t}(\phi_\tau(t)a_i(t,x))-d_{i}\Delta (\phi_\tau(t)a_i(t,x)) =a_i(t,x)\phi^{'}_\tau(t)+\phi_\tau(t)(\partial_t a_i(t,x) -d_i \Delta a_{i}(t,x)) \nonumber\\
		=a_{i}(t,x)\phi^{'}_\tau(t)+\phi_\tau(t)\left(a_m-a_1\prod\limits_{j=2}^{m-1}a_j^{\alpha_j}\right). \nonumber
	\end{gather}
	Observe that, thanks to \eqref{new label psi1}, we have
	\begin{gather*}
		\int_{\tau}^{\tau+2} \int_{\Omega}\phi_\tau(t)a_i(t,x)\theta(t,x) 
		=-\int_{\tau}^{\tau+2} \int_{\Omega}\phi_\tau(t)a(t,x)(\partial_t \psi(t,x) +(d_i+d-d_i)\Delta \psi(t,x)). 
	\end{gather*}	
	Integration by parts yields
	\begin{align*}
		\int_{\tau}^{\tau+2} \int_{\Omega}\phi_\tau(t)a_i(t,x)\theta(t,x) =\int_{\tau}^{\tau+2} \int_{\Omega}& \psi(t,x)(\partial_{t}(\phi_\tau(t)a_{i}(t,x))-d_{i}\Delta (\phi_\tau(t)a_{i}(t,x)) \\ 
		&+ (d_i-d)\int_{\tau}^{\tau+2} \int_{\Omega} \phi_\tau(t)a_{i}(t,x) \Delta \psi(t,x)\\ \vspace{.3cm}
		=\int_{\tau}^{\tau+2} \int_{\Omega} & \psi(t,x)\left(a_{i}(t,x)\phi^{'}_\tau(t)+\phi_\tau(t)\left(a_m-a_1\prod\limits_{j=2}^{m-1}a_j^{\alpha_j}\right)\right)\\
		&+(d_i-d)\int_{\tau}^{\tau+2} \int_{\Omega} \phi_\tau(t)a_{i}(t,x) \Delta \psi(t,x).
	\end{align*}
	Similarly, we obtain
	\begin{align*}
		\int_{\tau}^{\tau+2} \int_{\Omega}\phi_\tau(t)a_m(t,x)\theta(t,x)
		=\int_{\tau}^{\tau+2} \int_{\Omega} & \psi(t,x)\left(a_m(x,t)\phi^{'}_\tau(t)+\phi_\tau(t)\left(a_1\prod\limits_{j=2}^{m-1}a_j^{\alpha_j}-a_m\right)\right)\\
		& +(d_m-d)\int_{\tau}^{\tau+2} \int_{\Omega} \phi_\tau(t)a_m(t,x) \Delta \psi(t,x). 
	\end{align*}
	Adding the above two equalities leads to
	\begin{align*}
		\int_{\tau}^{\tau+2} \int_{\Omega}\phi_\tau(t)(a_i(t,x)+a_m(t,x))\theta(t,x)
		\leq &\int_{\tau}^{\tau+2} \int_{\Omega} \psi(t,x)(a_{i}(t,x)+a_m(t,x))\phi^{'}_\tau(t)\\ 
		+& \frac{\vert d_i-d_m\vert}{2} \int_{\tau}^{\tau+2} \int_{\Omega} \phi_\tau(t)(a_i(t,x)+a_m(t,x)) \vert\Delta \psi(t,x)\vert\\
		=:& \textbf{I}+\textbf{II}.
	\end{align*}
	We start by estimating \textbf{II}. Thanks to the H\"older inequality, we obtain the following estimate:
	\begin{align*}
		\textbf{II} \leq & \frac{\vert d_i-d_m \vert}{2} \Vert \Delta \psi \Vert_{\mathrm{L}^{{p'}} 
			(\Omega_{\tau,\tau+2})}\Vert  \phi_{\tau}(t) (a_i(t,x)+a_m(t,x)) \Vert_{\mathrm{L}^{p}(\Omega_{\tau,\tau+2})}  %\label{1 d_1}
		\\
		\leq &\frac{\vert d_i-d_m \vert}{d_{i}+d_{m}}C_{SOR}(\Omega,N,p')\Vert \theta \Vert_{\mathrm{L}^{p^{'}}(\Omega_{\tau,\tau+2})}\Vert \phi_{\tau}(t)( a_i+a_m) \Vert_{\mathrm{L}^{p}(\Omega_{\tau,\tau+2})}. \nonumber
	\end{align*}
	Note that $\displaystyle{p'<\frac{(N+2)p'}{N+2-2p'}}$. If we choose  $q$ such that  $\displaystyle{p'<q< \frac{(N+2)p'}{N+2-2p'}}$, then it's H\"older conjugate  $q'$  satisfies  $1<q'<p$.
	
	Let us consider \textbf{I}. Note that the cut off function satisfies  
	\[
	\phi^{'}_{\tau}(s)|_{[\tau+1,\tau+2]}=0 \mbox{\ and\ }   \phi^{'}_{\tau}(s)|_{[\tau,\tau+1]} \leq M_{\phi}.
	\]
	Thanks to the H\"older inequality again, the following estimate holds:
	\begin{align*}
		\textbf{I} \leq &  M_{\phi} \Vert \psi \Vert_{\mathrm{L}^{q}(\Omega_{\tau,\tau+2})} \Vert  a_i(t,x)+a_m(t,x) \Vert_{\mathrm{L}^{q'}(\Omega_{\tau,\tau+1})}  %\label{2 d_1}
		\\
		\leq & M_{\phi} C_{IE}( \Omega,N,p^{'},q) \Vert \theta \Vert_{\mathrm{L}^{p^{'}}(\Omega_{\tau,\tau+2})}\Vert  a_i(t,x)+a_m(t,x) \Vert_{\mathrm{L}^{q'}(\Omega_{\tau,\tau+1})}. \nonumber
	\end{align*}
	Since $1<q'<p$, there exists $\alpha \in (0,1)$ such that
	$$\frac{1}{q'}=\frac{\alpha}{1}+\frac{1-\alpha}{p}.$$
	By interpolation, we have
	\begin{align*}
		\Vert  a_i+a_m \Vert_{\mathrm{L}^{q'}(\Omega_{\tau,\tau+1})} 
		\leq 
		\Vert  a_i+a_m \Vert^{\alpha}_{\mathrm{L}^{1}(\Omega_{\tau,\tau+1})}\Vert  a_i(t,x)+a_m(t,x) \Vert^{1-\alpha}_{\mathrm{L}^{p}(\Omega_{\tau,\tau+1})}.
	\end{align*}
	Thus
	\[
	\textbf{I} \leq M_{\phi} C_{IE}( \Omega,N,p^{'},q) \Vert \theta \Vert_{\mathrm{L}^{q}(\Omega_{\tau,\tau+2})}2M^{\alpha}_1  \Vert  a_i+a_m \Vert^{1-\alpha}_{\mathrm{L}^{p}(\Omega_{\tau,\tau+1})}. 
	\]
	Choose the positive time independent constant $\displaystyle{C^1:=M_{\phi} C_{IE}( \Omega,N,p^{'},q)2 M^{\alpha}_1}$.  Putting our above estimates together we arrive at
	\begin{align*}
		\left\vert\int_{\tau}^{\tau+2} \int_{\Omega}\phi_\tau(t)(a_i(t,x)+a_m(t,x))\theta(t,x)\right\vert \leq & \Vert \theta \Vert_{\mathrm{L}^{p'}(\Omega_{\tau,\tau+2})}\left(C^1 \Vert  a_i+a_m \Vert^{1-\alpha}_{\mathrm{L}^{p}(\Omega_{\tau,\tau+1})}\right)  \\
		\Vert \theta \Vert_{\mathrm{L}^{p'}(\Omega_{\tau,\tau+2})}&\left(  \frac{\vert d_i-d_m \vert}{d_{i}+d_{m}}C_{SOR}(\Omega,N,p^{'})\Vert  \phi_{\tau}(a_i+a_m) \Vert_{\mathrm{L}^{p}(\Omega_{\tau,\tau+2})}\right).
	\end{align*}
	Since $0\leq \theta \in \mathrm \mathrm{L}^{p'}(\Omega_{\tau,\tau+2})$\ was arbitrary, we conclude that
	\begin{align*}
		\Vert \phi_{\tau}(a_i+a_m) \Vert_{\mathrm{L}^{p}(\Omega_{\tau,\tau+2})} \leq & C^1 \Vert  a_i+a_m \Vert^{1-\alpha}_{\mathrm{L}^{p}(\Omega_{\tau,\tau+1})} \\
		+ \frac{\vert d_i-d_m \vert}{d_{i}+d_{m}} & C_{SOR}(\Omega,N,p')\Vert  \phi_{\tau}(a_i+a_m) \Vert_{\mathrm{L}^{p}(\Omega_{\tau,\tau+2})}.
	\end{align*}
	Thanks to the closeness condition \eqref{tri closeness condition imp} on diffusion coefficients, we arrive at the following relation 
	\begin{align*}
		\Vert \phi_{\tau}(a_i+a_m) \Vert_{\mathrm{L}^{p}(\Omega_{\tau,\tau+2})} 
		\leq \frac{C^1}{1-\frac{\vert d_i-d_m \vert}{d_{i}+d_{m}}C_{SOR}(\Omega,N,p')} \Vert  a_i+a_m \Vert^{1-\alpha}_{\mathrm{L}^{p}(\Omega_{\tau,\tau+1})}. 
	\end{align*}
	Let's choose the constant $\displaystyle{\texttt{C}=\frac{C^1}{1-\frac{\vert d_i-d_m \vert}{d_{i}+d_{m}}C_{SOR}(\Omega,N,p^{'})}}$. Let's further define a sequence 
	\[
	\displaystyle{{\beta}_n=\Vert  a_i+a_m \Vert_{\mathrm{L}^{p}(\Omega_{n,n+1})}} \qquad \forall n\in \mathbb{N}.
	\]  
	The sequence satisfies 
	\[ 
	{\beta}_{n+1} \leq \texttt{C} {\beta}_n^{1-\alpha}. 
	\]
	Consider the set  $\Lambda=\{n\in \mathbb{N} \ \text{such that}\ {\beta}_n \leq {\beta}_{n+1}\}$. This implies
	\[
	{\beta}_n \leq \ {\texttt{C}}^{\frac{1}{\alpha}} \qquad \forall n\in \Lambda,
	\]
	which further yields
	\[
	{\beta}_{n+1} \leq \max \{{\beta}_n,{\texttt{C}}^{\frac{1+\alpha}{\alpha}} \} \implies
	{\beta}_{n} \leq \max \{{\beta}_0,{\texttt{C}}^{\frac{1}{\alpha}},{\texttt{C}}^{\frac{1+\alpha}{\alpha}}\}, \quad \forall n\in\mathbb{N}.
	\]
	We now move on to estimate ${\beta}_0$. The sum $(a_i+a_m)$ satisfy the following differential equation
	\begin{equation*}
		\left \{
		\begin{aligned}
			\partial_{t}(a_i+a_m)- \Delta(M(a_i+a_m))= & 0  \qquad \qquad \qquad \qquad \quad \ \mbox{in}\ \Omega_{T} \nonumber \\
			\nabla_{x}(a_i+a_m).\gamma=& 0  \qquad \qquad \qquad \qquad \  \mbox{on} \ (0,T)\times \partial\Omega \nonumber \\
			(a_i+a_m)(0,x)=& a_{i,0}+a_{m,0}\in {\mathrm{L}^p(\Omega)}  \qquad \mbox{in}\ \Omega, \nonumber 
		\end{aligned}
		\right .
	\end{equation*}
	where  $M$  is the coefficient function  $\displaystyle{\frac{d_i a_i+d_m a_m}{a_i+a_m}}$, that satisfies the following bound:
	\[
	0<\min\{d_i,d_m\}\leq M \leq \max\{d_i,d_m\}.
	\]
	Thanks to the closeness condition \eqref{tri closeness condition imp} satisfied by $d_i$ and $d_m$, the parabolic integrability estimation  (see Theorem \ref{PE}) yields
	\begin{align*}
		{\beta}_0=\Vert  a_i(t,x)+ a_m(t,x) \Vert_{\mathrm{L}^{p}(\Omega_{0,1})} 
		\leq  \left (1+\max\{d_i,d_m\} \frac{C^{PRC}_{{\frac{d_i+d_m}{2}},p^{'}}\vert{d_i-d_m}\vert}{1-C^{PRC}_{{\frac{d_i+d_m}{2}},p^{'}}\vert{d_i-d_m}\vert} \right) \ \Vert  a_{i,0}+a_{m,0} \Vert_{\mathrm{L}^{p}(\Omega)}. 
	\end{align*}
	Thus
	\[
	{\beta}_{n} \leq \max \Bigg \{ \bigg( 1+\max\{d_i,d_m\} \frac{C^{PRC}_{{\frac{d_i+d_m}{2}},p^{'}}\vert {d_i-d_m}\vert}{1-C^{PRC}_{{\frac{d_i+d_m}{2}},p^{'}}\vert {d_i-d_m}\vert}\bigg) \ \Vert  a_{i,0}+a_{m,0} \Vert_{\mathrm{L}^{p}(\Omega)}  ,{\texttt{C}}^{\frac{1}{\alpha}},{\texttt{C}}^{\frac{1+\alpha}{\alpha}}\Bigg \}=: C_0. 
	\]
	Hence  for  $p>N$, there exist the above constant $C_0$, independent of $\tau$, such that  
	\[
	\Vert a_m \Vert_{\mathrm{L}^{p}(\Omega_{\tau,\tau+1})} \leq C_0 \qquad \forall \tau\geq0.
	\]
\end{proof}
\vspace{.3cm}

Our next objective is to obtain an estimate on the supremum norm of the species concentrations. We also emphasize here that, so far, we have been deriving the integrability estimate for the diffusive species' concentrations.

Our next result derives a supremum norm estimate on the species concentrations $a_2,\cdots,a_{m-1}$ in the parabolic cylinder $\Omega_{0,1}$ under the assumption that an $\mathrm{L}^p$ norm estimate is available for $a_m$ in $\Omega_{0,1}$ for some $p>N$.

\begin{Prop}\label{L^infty estimation of a_i [0,1]}
	Let $(a_1,\cdots,a_m)$ be the positive smooth solution to the degenerate triangular reaction-diffusion system \eqref{degenerate 1}. Suppose there exists a constant $C_0>0$ such that   \[
	\displaystyle{\Vert a_m \Vert_{\mathrm{L}^p(\Omega_{0,1})}\leq C_0}
	\]
	for some exponent $p>N$. Then, there exists a positive constant $K_{\infty,0,1}$ such that
	\[
	\Vert a_i \Vert_{\mathrm{L}^{\infty}(\Omega_{0,1})} \leq K_{\infty,0,1}, \qquad \forall i=2,\cdots,m-1.
	\]
\end{Prop}

\begin{proof} The species $a_i$ is a positive subsolution of the following equation in the time interval $(0,2)$, for $2\leq m-1$:
	\begin{equation}\label{eq:tilde ai}
		\left \{
		\begin{aligned}
			\partial_{t}a_i(t,x)-d_i\Delta a_i(t,x)
			\leq &a_m  \qquad \qquad \qquad  \mbox{ in }\Omega_{0,2}
			\\
			\nabla_x a_i(t,x).\gamma =&  0 \qquad   \qquad \quad   \ \  \mbox{ on } \partial\Omega_{0,2} 
			\\
			a_i(0,x) =& a_{i,0} \qquad   \quad  \qquad \ \ \mbox{ in } \Omega.
		\end{aligned}
		\right.
	\end{equation}
	The solution corresponding to the equation \eqref{eq:tilde ai} can be expressed as:
	\[
	0\leq a_i\leq \int_{\Omega}G_{d_i}(t,0,x,y)a_{i,0}(y) \rm{d}y+ \int_{0}^{t}\int_{\Omega} G_{d_i}(t,s,x,y)a_m(s,y)\rm{d}y\rm{d}s \quad (t,x)\in\Omega_{0,2}.
	\]
	where $G_{d_i}$ denotes the Green's function associated with the operator $\partial_t-d_i\Delta$ with Neumann boundary condition. We use the fact that $\displaystyle{a_{i,0}(y)\in\mathrm{L}^{\infty}(\Omega)}$ and $\displaystyle{\int_{\Omega}G_{d_i}(t,0,x,y) \rm{d}y \leq 1}$ for all $t\in(0,2)$. It yields
	\[
	\left\vert \int_{\Omega}G_{d_i}(t,0,x,y)a_{i,0}(y) \rm{d}y \right\vert \leq \Vert a_{i,0}\Vert_{\mathrm{L}^{\infty}(\Omega)}.
	\]
Green function estimate as in \eqref{Heat kernel estimate} and \eqref{Heat kernel bound function} yields
\[
0\leq G_{d_i}(t_1,s,x,y) \leq  g_{d_i}(t_1-s,x-y) \qquad 0\leq s<t_1
\]
where the function $g_{d_i}$ as defined in  \eqref{Heat kernel bound function}. Furthermore \eqref{Heat kernel integral estimate} yields the following integral estimate:
\[
\left\Vert g_{d_i}\right\Vert_{\mathrm{L}^{z}((-2,2)\times\mathbb{R}^N)} \leq \tilde{\kappa}_{1,d_i} \qquad \forall z\in\left[1,1+\frac{2}{N}\right)
\]
for some constant positive $\tilde{\kappa}_{1,d_i}$ depending on $z$. We in particular choose $\displaystyle{\frac{1}{z}=1-\frac{1}{p}}$ which is admissible because $\displaystyle{p>\frac{N+2}{2}}$. Hence, for $t\leq 1$:
	\begin{align*}
	\left \vert \int_{0}^{t}\int_{\Omega} G_{d_i}(t,s,x,y)a_m(s,y)\rm{d}y\rm{d}s\right\vert \leq & \int_{0}^{t}\int_{\Omega} g_{d_i}(t-s,x-y)a_m(s,y)\rm{d}y\rm{d}s\\
	\leq & \Big\Vert g_{d_i}(t-\cdot,x-\cdot)\Big\Vert_{\mathrm{L}^z((0,t)\times\Omega)} \Big\Vert a_m\Big\Vert_{\mathrm{L}^p((0,t)\times\Omega)}\\
	\leq & \Big\Vert g_{d_i}(\cdot,\cdot)\Big\Vert_{\mathrm{L}^z((0,t)\times\mathbb{R}^N)} \Big\Vert a_m\Big\Vert_{\mathrm{L}^p((0,1)\times\Omega)}\leq \tilde{\kappa}_{1,d_i} C_0.
	\end{align*}
	It further yields the pointwise bound for the species $a$ in the unit parabolic cylinder in initial time. More precisely, we have
	\[
		\Vert a_i\Vert_{\mathrm{L}^{\infty}(\Omega_{0,1})}\leq  \max\limits_{1\leq i\leq m}\left\{\Vert a_{i,0}\Vert_{\mathrm{L}^{\infty}(\Omega)}\right\} +C_0\max\limits_{1\leq i\leq m}\{\tilde{\kappa}_{1,d_i}\}=:K_{\infty,0,1}
	\]
\end{proof}
\vspace{.2cm}

The following result proves the analogue of Proposition \ref{a_i to a_m L^p} in the parabolic cylinder $\Omega_{\tau,\tau+1}$.

\vspace{.3cm}

\begin{Thm} \label{L^infty estimation of a_i}
	Let $(a_1,\cdots,a_m)$ be the positive smooth solution to the degenerate triangular reaction-diffusion system \eqref{degenerate 1}. Suppose there exists a positive constant $C_0$ (independent of $\tau$) such that
	\[
	\Vert a_m \Vert_{\mathrm{L}^p(\Omega_{\tau,\tau+1})}\leq C_0
	\]
	for all  $\tau\geq 0$ and for some exponent $p>N$.  Then, there exists a positive constant $K_{\infty}$, independent of $\tau$, such that
	\[
	\Vert a_i \Vert_{\mathrm{L}^{\infty}(\Omega)} \leq K_{\infty} \qquad \forall \tau \geq 1, \ i=2,\cdots,m-1.
	\]
\end{Thm}

\begin{proof}  Consider the function $\phi_{\tau}:\mathbb{R}\to\mathbb{R}$ as in \eqref{cut-off}. The product $\phi_\tau(t)a_i(t,x)$ satisfies for $2\leq i\leq m-1$:
	
	\begin{equation}\label{equation:L^infty}
		\left \{
		\begin{aligned}
			\partial_{t}(\phi_\tau(t)&a_i(t,x))-d_{i}\Delta (\phi_\tau(t)a_i(t,x))\\
			=& a_{i}(x,t)\phi^{'}_\tau(t)+\phi_\tau(t)\left(a_m-a_1\prod\limits_{j=2}^{m-1}a_j^{\alpha_j}\right) \qquad  \ \ \mbox{ in } \Omega_{\tau,\tau+2}\\
			\nabla_x &(\phi_\tau(t)a_i(t,x)).\gamma=  0 \qquad   \qquad \ \  \qquad \qquad \quad \qquad \mbox{ on } \partial\Omega_{\tau,\tau+2} \\
			&\phi_{\tau}a_i(\tau,x)= 0 \qquad \qquad \qquad \qquad \qquad \qquad \qquad \ \ \mbox{ in } \Omega.
		\end{aligned}
		\right.
	\end{equation}
	Change of variable $t\rightarrow t_1+\tau$ yields
	
	\begin{equation*}
		\left \{
		\begin{aligned}
			\partial_{t_1}(\phi_\tau(t_1+\tau)&a_i(t_1+\tau,x))-d_{i}\Delta (\phi_\tau(t_1+\tau)a_i(t_1+\tau,x))\\ &= a_{i}(t_1+\tau,x)\phi^{'}_\tau(t_1+\tau)+\phi_\tau(t_1+\tau)\left(a_m-a_1\prod\limits_{j=2}^{m-1}a_j^{\alpha_j}\right) \quad \  \ \mbox{\ in \ }\Omega_{0,2} \\
			& \nabla_x (\phi_\tau(t_1+\tau)a_i(t_1+\tau,x)).\gamma=  0 \qquad   \qquad \qquad \qquad \qquad \quad  \qquad \ \mbox{ on } \partial\Omega_{0,2} \\
			& \phi_{\tau}a_i(0,x)= 0 \qquad \qquad \qquad \qquad \qquad \ \ \ \qquad  \qquad \qquad \quad \qquad \ \quad \  \mbox{ in } \Omega. 
		\end{aligned}
		\right .
	\end{equation*}
	Let $G_{d_i}(t_1,s,x,y)$ be the Green's function corresponding to the heat operator $\displaystyle{\partial_{t_1}-d_i \Delta}$.  Then for all $t_1\in[0,2]$, we can represent the solution as
	\begin{align*}
		\phi_\tau(t_1+\tau)&a_i(t_1+\tau,x)= 
		\\
		&\int_{0}^{t_1}\int_{\Omega}G_{d_i}(t_1,s,x,y)\left(a_{i}(y,s+\tau)\phi^{'}_\tau(s+\tau)+\phi_\tau(s+\tau)\left(a_m-a_1\prod\limits_{j=2}^{m-1}a_j^{\alpha_j}\right)(s+\tau,t)\right).
	\end{align*}
	Thanks to the non-negativity of the Green's function and the species concentrations, an immediate pointwise bound follows.
	\begin{align*}
		\phi_\tau(t_1+\tau) a_i(t_1+\tau,x)
		\int_{0}^{t_1}\int_{\Omega}G_{d_i}(t_1,s,x,y)\Big(a_{i}(s+\tau,x)\phi^{'}_\tau(s+\tau)+\phi_\tau(s+\tau)a_m(s+\tau,x)\Big).
	\end{align*}
	Green function estimate as in \eqref{Heat kernel estimate} and \eqref{Heat kernel bound function} yields
	\[
	0\leq G_{d_i}(t_1,s,x,y) \leq  g_{d_i}(t_1-s,x-y) \qquad 0\leq s<t_1
	\]
	where the function $g_{d_i}$ as defined in  \eqref{Heat kernel bound function}.
	Substituting the above pointwise bound on the Green's function, the earlier inequality reads 
	\begin{align*}
		\phi_\tau(t_1+\tau)a_i(t_1+\tau,x)
		\leq  \int_{0}^{t_1}\int_{\Omega}g_{d_i}(t_1-s,x,y)\Big(a_{i}(s+\tau,x)\phi^{'}_\tau(s+\tau)+\phi_\tau(s+\tau)a_m(s+\tau,x)\Big).
	\end{align*}
  Substituting this pointwise bound for the Green's function in the earlier inequality leads to
  \[
  \phi_\tau(t_1+\tau)a_i(t_1+\tau,x)\leq \int_{0}^{t_1}\int_{\Omega}g_{d_i}(t_1-s,x,y)\Big(a_{i}(s+\tau,x)\phi^{'}_\tau(s+\tau)+\phi_\tau(s+\tau)a_m(s+\tau,x)\Big).
  \]
  Let us consider the following two functions
  \begin{equation*}
  	\Tilde{g}_{d_i}(s,x):=
  	\left \{
  	\begin{aligned} 
  		&g_{d_i}(s,x)  \qquad (s,x)\in (-1,1)\times \mathbb{R}^N,\\
  		& 0  \qquad \qquad \qquad \mbox{otherwise}
  	\end{aligned}
  	\right .
  \end{equation*}
  and
  \begin{equation*}
  	F(s,x):=
  	\left \{
  	\begin{aligned} 
  		a_{i}(s+\tau,x)\phi^{'}_\tau(s+\tau)&+\phi_\tau(s+\tau)a_m(s+\tau,x)  \qquad (s,x)\in (0,1)\times \Omega\\
  		& 0  \qquad \qquad \qquad \qquad \qquad \qquad \ \ \ \mbox{otherwise}.
  	\end{aligned}
  	\right .
  \end{equation*} 
  Hence 
  \[
  \int_{0}^{t}\int_{\Omega} G_{d_i}(t,s,x,y)a_i(s,y)\rm{d}y\rm{d}s \leq \int_{\mathbb{R}}\int_{\mathbb{R}^N} \Tilde{g}_{d_i}(t-s,x-y) F(s,y) \rm{d}y \rm{d}s.
  \]
  We use Young's convolution inequality. As $p>N$, there exists a $1\leq q <p$ such that $\displaystyle{1+\frac{1}{p}=\frac{1}{1+\frac{1}{N}}+\frac{1}{q}}$. It yields
  \begin{align*}
  	\left\Vert \int_{0}^{t}\int_{\Omega} G_{d_i}(t,s,x,y)a_i(s,y)\rm{d}y\rm{d}s \right\Vert_ {\mathrm{L}^p(\Omega_{0,1})} &\leq \left \Vert \int_{\mathbb{R}}\int_{\mathbb{R}^N} \Tilde{g}_{d_i}(t-s,x-y) F(s,y) \rm{d}y \rm{d}s \right\Vert_ {\mathrm{L}^p(\mathbb{R}\times\mathbb{R}^N)}\\ & \leq \Big\Vert \Tilde{g}_{d_i} \Big\Vert_{\mathrm{L}^{1+\frac{1}{N}}(\mathbb{R}\times\mathbb{R}^N)}
  	\Big\Vert F \Big\Vert_{\mathrm{L}^q(\mathbb{R}\times\mathbb{R}^N)}
  \end{align*}
   Hence
  \[
  \left\Vert \int_{0}^{t}\int_{\Omega} G_{d_i}(t,s,x,y)a_m(s,y)\rm{d}y\rm{d}s \right\Vert_{\mathrm{L}^p(\Omega_{0,1})} \leq \Big\Vert g_{d_i} \Big\Vert_{\mathrm{L}^{1+\frac{1}{N}}((-1,1)\times\mathbb{R}^N)}
  \Big\Vert a_{i}(\cdot+\tau,\cdot)\phi^{'}_\tau(\cdot+\tau)+\phi_\tau(\cdot+\tau)a_m(\cdot+\tau,\cdot) \Big\Vert_{\mathrm{L}^q(\Omega_{0,1})}. 
  \]	
Furthermore \eqref{Heat kernel integral estimate} yields the following integral estimate:
\[
\left\Vert g_{d_i}\right\Vert_{\mathrm{L}^{z}((-2,2)\times\mathbb{R}^N)} \leq \tilde{\kappa}_{1,d_i} \qquad \forall z\in\left[1,1+\frac{2}{N}\right)
\]
for some constant positive $\tilde{\kappa}_{1,d_i}$ depending on $z$. We in particular choose $\displaystyle{z=1+\frac{1}{N}}$. Change of variable $t_1+\tau=t$ yields
	
	\begin{align*}
		\Vert \phi_\tau(t) a_i(t,x) \Vert_{\mathrm{L}^p(\Omega_{\tau,\tau+2})} 
		\leq  \tilde{\kappa}_{1,d_i}\Big( \Vert a_{i}(t,x)\phi^{'}_\tau(t)\Vert_{\mathrm{L}^q(\Omega_{\tau,\tau+2})}+\Vert\phi_\tau(t)a_m(t,x)\Vert_{\mathrm{L}^q(\Omega_{\tau,\tau+2})}\Big).
	\end{align*}
	As $q<p$, there exists $\alpha\in[0,1)$, such that $\displaystyle{\frac{1}{q}=\frac{1-\alpha}{1}+\frac{\alpha}{p}}$. Interpolation thus yields
	
	\begin{align*}
		\Vert \phi_\tau(t)a_i(t,x) \Vert_{\mathrm{L}^p(\Omega_{\tau,\tau+2})} \leq \tilde{\kappa}_{1,d_i}&\left(  \Vert a_{i}(t,x)\phi^{'}_\tau(t)\Vert^{\alpha}_{\mathrm{L}^p(\Omega_{\tau,\tau+2})}\Vert a_{i}(t,x)\phi^{'}_\tau(t)\Vert^{1-\alpha}_{\mathrm{L}^1(\Omega_{\tau,\tau+2})} \right)\\
		& + \tilde{\kappa}_{1,d_i}\left(\Vert a_m(t,x)\Vert^{\alpha}_{\mathrm{L}^p(\Omega_{\tau,\tau+2})}\Vert a_m(t,x)\Vert^{1-\alpha}_{\mathrm{L}^1(\Omega_{\tau,\tau+2})}\right).
	\end{align*}
	
	Using the $\mathrm{L}^p(\Omega_{\tau,\tau+1})$ bound on $a_m$, we get 
	\begin{align*}
		\Vert \phi_\tau(t)a_i \Vert_{\mathrm{L}^p(\Omega_{\tau,\tau+2})} 
		\leq  \tilde{\kappa}_{1,d_i}(M_{\phi}M_1)^{1-\alpha}\Vert a_{i}\phi^{'}_\tau(t)\Vert^{1-\alpha}_{\mathrm{L}^p(\Omega_{\tau,\tau+2})}+C_1M_1^{1-\alpha}C_{0}^{\alpha},
	\end{align*}
	where we also use the fact that $\Vert a_i\Vert_{\mathrm{L}^1(\Omega)}\leq M_1$, for all $1\leq i\leq m$.
	\newline
	Consider the following positive constants  
	\[
	C_1:=1+\max\limits_{2\leq i\leq m}\{\tilde{\kappa}_{1,d_i}\}(M_{\phi}M_1)^{1-\alpha} \mbox{\ and\ } C_2:=\max\limits_{2\leq i\leq m}\{\tilde{\kappa}_{1,d_i}\}M_1^{1-\alpha}C_{0}^{\alpha}.
	\]
	Thanks to the property of $\phi_{\tau}$, we deduce
	\[
	\Vert a_i(t,x) \Vert_{\mathrm{L}^p(\Omega_{\tau+1,\tau+2})} \leq C_1 \Vert a_i(t,x) \Vert^{\alpha}_{\mathrm{L}^p(\Omega_{\tau,\tau+1})}+C_2 \qquad \alpha \in [0,1).
	\]
	
	Define the sequence ${\beta}_n=\{\Vert a_i \Vert_{L^p(\Omega_{n,n+1})}\}$ for $n\in\mathbb{N}$. Then $\{{\beta}_n\}$ satisfies the following relation
	\begin{align}\label{d_i, L^p a_i relation 1}
		{\beta}_{n+1} \leq & C_1 {\beta}_n^{\alpha}+C_2.
	\end{align}
	
	If ${\beta}_{n+1}>{\beta}_{n}$, Young's inequality yields
	\[
	C_1{\beta}_{n+1}^{\alpha}\leq \alpha {\beta}_{n+1}+(1-\alpha)C_1^{\frac{1}{1-\alpha}}.
	\]
	Substituting this in \eqref{d_i, L^p a_i relation 1}, we arrive at the following estimate:
	\[
	{\beta}_{n+1} \leq C_1^{\frac{1}{1-\alpha}}+\frac{C_2}{1-\alpha}.
	\]
	Hence,
	\begin{align}\label{d_i, L^p a_i relation 2}
		{\beta}_n \leq \max\{ C_1^{\frac{1}{1-\alpha}}+\frac{C_2}{1-\alpha},\beta_0\}=C_3.
	\end{align}
	
	Note that ${\beta}_0$ can be estimated thanks to \ref{L^infty estimation of a_i [0,1]}: 
	\[
	\displaystyle{{\beta}_0\leq K_{\infty,0,1}\vert \Omega\vert^{\frac{1}{p}}}
	\]

	Let $\psi(t,x)$ be the solution to
	\begin{equation*}
		\left \{
		\begin{aligned}
			\partial_{t}\psi(t,x)-d_{i}\Delta \psi(t,x)=& a_{i}(t,x)\phi'_\tau(t)+\phi_\tau(t)a_m \qquad \mbox{ in } \Omega_{\tau,\tau+2} \nonumber \\
			\nabla_x \psi(t,x).\gamma= & 0 \qquad  \qquad \qquad \qquad \qquad \quad  \ \mbox{ on } \partial\Omega_{\tau,\tau+2} \nonumber \\
			\psi(\tau,x)=& 0 \qquad \qquad \qquad \qquad \qquad  \quad \ \mbox{ in } \Omega. \nonumber
		\end{aligned}
		\right .
	\end{equation*}
	From the maximum principle  we have $\displaystyle{ \phi_{\tau}a_i(t,x)\leq \psi_i(t,x)}$ in  $\Omega_{\tau,\tau+2}$.   Furthermore the integrability estimation (see Appendix, Theorem \ref{estimation 1}) yields
	\begin{align*}
		\Vert \psi_i \Vert_{\mathrm{L}^{\infty}(\Omega_{\tau,\tau+2})}  \leq  C_{IE}(\Omega,d_i,N,p) \Vert a_{i}\phi^{'}_\tau +\phi_\tau a_m \Vert_{\mathrm{L}^p(\Omega_{\tau,\tau+2})} 
		\leq  C_{IE}(\Omega,d_i,N,p)(M_{\phi}C_3+C_0).
	\end{align*}
	
	Non-negativity of $a_i$ and property of $\phi_{\tau}$ lead to
	\begin{align*}
		\Vert a_i(t,x) \Vert_{\mathrm{L}^{\infty}(\Omega_{\tau,\tau+1})} 
		\leq \max\limits_{i=2,\cdots,m-1}&\Big\{C_{IE}(\Omega,d_i,N,p)(M_{\phi}C_3+C_0)+\Vert a_i(t,x) \Vert_{\mathrm{L}^{\infty}(\Omega_{0,1})}\Big\}\\
		\leq & \max\limits_{i=2,\cdots,m-1}\{C_{IE}(\Omega,d_i,N,p)\}(M_{\phi}C_3+C_0)+K_{\infty,0,1}.
	\end{align*}
	Let us define 
	\[
	K_{\infty}:=\max\limits_{2\leq i\leq m-1}C_{IE}(\Omega,d_i,N,p)(M_{\phi}C_4+C_0)+K_{\infty,0,1}
	\]
	where $K_{\infty,0,1}$ is the bound of $\Vert a_i(t,x) \Vert_{\mathrm{L}^{\infty}(\Omega_{0,1})}$ from Proposition \ref{L^infty estimation of a_i [0,1]}.
	
	Since $\tau\geq 0$ is arbitrary, we arrive at our conclusion.
	\[
	\sup\limits_{t>0}\Vert a_i \Vert_{\mathrm{L}^{\infty}(\Omega)} \leq K_{\infty}   \qquad \forall i=2,\cdots,m-1.
	\]
\end{proof}

\begin{Prop} \label{a_1 estimate}
	Let the dimension $N\geq 4$. Let $(a_1,\cdots,a_m)$ be the smooth positive solution to the degenerate triangular reaction-diffusion system \eqref{degenerate 1}. Suppose further that there exists a constant $C_0>0$ such that
	\[
	\Vert a_m \Vert_{\mathrm{L}^p(\Omega_{\tau,\tau+1})} \leq C_0 \quad \forall \tau\geq 0
	\]
	and for some $p>N$. Then, there exists a positive constant $\Tilde{K}$, independent of time, such that
	\[ 
	\Vert a_1 \Vert_{\mathrm{L}^{\frac{N}{2}}(\Omega)} \leq \Tilde{K}(1+t)^{\frac{N-2}{N-1}}.
	\]
\end{Prop}

\begin{proof} An application of the H\"older inequality yields
	\begin{gather}\label{1 d_1 new} 
		\Vert a_m \Vert_{\mathrm{L}^{N}(\Omega_{\tau,\tau+1})} \leq 
		\Vert a_m \Vert_{L^{p}(\Omega_{\tau,\tau+1})} {\vert \Omega \vert}^{\frac{p-N}{p}}.
	\end{gather}
	
	Thanks to the non-negativity of the species concentrations, we deduce the following differential inequality  
	\[
	\displaystyle{\partial_t a_1 \leq a_m}.
	\] 
	Integrating the above inequality in time yields
	\[
	0 \leq a_1 \leq a_{1,0}(x) +\int_0^t a_m(s,x) \rm{d}s.
	\]
	Thnaks to Minkowski's integral inequality, we arrive at
	\[
	\Vert a_1(t,\cdot) \Vert_{\mathrm{L}^N(\Omega)} \leq \Vert a_{1,0} \Vert_{\mathrm{L}^N(\Omega)} +\int_{0}^{t}\Vert a_m(s,\cdot) \Vert_{\mathrm{L}^N(\Omega)} \leq \Vert a_{1,0} \Vert_{\mathrm{L}^N(\Omega)} +C_0 \vert \Omega \vert^{\frac{(p-N)}{p}}(1+t). 
	\]
	Observe that
	\[
	\displaystyle{\frac{2}{N}= \frac{\alpha}{N}+\frac{1-\alpha}{1}} \quad \mbox{where}\ \displaystyle{\alpha = \frac{N-2}{N-1}}.
	\]
	Hence by interpolation, we have
	\[
	\Vert a_1(t,\cdot) \Vert_{\mathrm{L}^{\frac{N}{2}}(\Omega)} \leq \Vert a_1(t,\cdot) \Vert_{\mathrm{L}^{N}(\Omega)}^{\frac{N-2}{N-1}} \Vert a_1(t,\cdot) \Vert_{\mathrm{L}^1(\Omega)}^{\frac{1}{N-1}}.
	\]
	
	Using mass conservation and earlier bound, we get
	\[
	\Vert a_1(t,\cdot) \Vert_{\mathrm{L}^{\frac{N}{2}}(\Omega)} \leq \left( \Vert a_{1,0} \Vert_{L^{N}(\Omega)}+C_0\vert \Omega \vert^{\frac{(p-N)}{p}} (1+t)\right)^{\frac{N-2}{N-1}}M_1^{\frac{1}{N-1}}.
	\]
	Defining the constant $\Tilde{K}$ as 
	\[
	\Tilde{K} =\Big( \Vert a_{1,0} \Vert_{L^{N}(\Omega)} +C_0 \vert \Omega \vert^{\frac{(p-N)}{p}} \Big)^{\frac{N-2}{N-1}}M_1^{\frac{1}{N-1}}
	\]
	we have derived the required estimate.
\end{proof}

\vspace{.2cm}

We have an analogous result to the above Lemma \ref{a_1 estimate} in dimension $N\leq 3$ which is unconditional as we have an $\mathrm{L}^9(\Omega_{\tau,\tau+1})$ estimate on $a_m$ (see Proposition \ref{a_i to a_m L^9 [0,1]}). We skip the proof of the following result as it goes similar to the proof of Lemma \ref{a_1 estimate}.

So far, we have not addressed the integrability of the species concentration $a_1$. Our result is indeed towards filling that gap. This is again a conditional result which assumes an $\mathrm{L}^p$ bound on $a_m$ for some $p>N$.

\vspace{.2cm}

\begin{Lem}\label{dim 1 a_1 L^n}
	Let the dimension $N\leq 3$. Let  $(a_1,\cdots,a_m)$ be the smooth positive solution to the degenerate triangular reaction-diffusion system \eqref{degenerate 1}. Then ,there exists a constant $\Tilde{K}$, independent of $t$, such that
	\[
	\Vert a_1 \Vert_{\mathrm{L}^{\frac{3}{2}}(\Omega)} \leq \Tilde{K}(1+t)^{\frac{1}{2}}.
	\]
\end{Lem}

\vspace{.3cm}

Recall that we have obtained a uniform (with respect to $t$) $\mathrm{L}^{\infty}(\Omega_t)$ bound on the species concentrations $a_2,\cdots,a_{m-1}$ in Proposition \ref{L^infty estimation of a_i} under the assumption that $\mathrm{L}^{p}(\Omega_{\tau,\tau+1})$ norm of $a_m$ is bounded for some $p>N$. Recall further that such an $\mathrm{L}^p$ estimate for $a_m$ is always available in dimensions $N\leq 3$ and is available under the closeness condition\eqref{tri closeness condition imp} on certain diffusion coefficients in dimensions $N\geq 4$ (see Proposition \ref{L^n estimation of a_m}). The following result derives an estimate for the species concentrations $a_1$ and $a_m$ in the $\mathrm{L}^{\infty}(\Omega_t)$ norm. 

\vspace{.2cm}

\begin{Thm}\label{polynomial on t}
	Let $(a_1,\cdots,a_m)$ be the smooth positive solution to the degenerate triangular reaction-diffusion system  \eqref{degenerate 1}. Suppose there exists a positive constant $C_0$, such that 
	\[
	\Vert a_m \Vert_{\mathrm{L}^p(\Omega_{\tau,\tau+1})}\leq C_0 \qquad  \forall \tau\geq 0
	\] 
	and for some $p>N$. Then there exists  positive constants $K_{pol}$ and $\mu$, independent of $t$, such that
	\[
	\max\left\{\Vert a_1 \Vert_{\mathrm{L}^{\infty}(\Omega_{0,t})}, \Vert a_m \Vert_{\mathrm{L}^{\infty}(\Omega_{0,t})}\right\} \leq K_{pol}(1+t)^{\mu}.
	\]
\end{Thm}

\begin{proof} %Thanks to  theorem \ref{L^infty estimation of a_i}, there exists a time independent constant $K_{\infty}$, such that   %\[
	%\Vert a_i(t,x) \Vert_{L^{\infty}(\Omega_T)} \leq K_{\infty}   \qquad \forall  i=2,\cdots,m-1.
	%\]
	
	Recall that we have $\displaystyle{\partial_{t}a_1 \leq a_m}$. Minkowski's integral inequality yields
	
	\[
	\Vert a_1(t,\cdot) \Vert_{\mathrm{L}^p(\Omega)} \leq \Vert a_{1,0} \Vert_{\mathrm{L}^p(\Omega)}+\int_{0}^{t}\Vert a_m(s,\cdot) \Vert_{\mathrm{L}^p(\Omega)} \rm{d}s\leq \Vert a_{1,0} \Vert_{\mathrm{L}^p(\Omega)}+C_{0}(1+t).
	\]
	Raising it to the power $p$, and integrating over time from $\tau$ to $\tau+2$, we arrive at 
	\begin{align}\label{polynomial on t 1}
		\Vert a_1 \Vert_{\mathrm{L}^p(\Omega_{\tau,\tau+2})} \leq C_1(1+\tau) 
	\end{align}
	where $\displaystyle{C_1:=4\left(\Vert a_{1,0} \Vert_{\mathrm{L}^p(\Omega)}+C_{0}\right)}$ and $\displaystyle{p>N\geq \frac{N+2}{2}}$.
	
	Consider the function $\phi_{\tau}:\mathbb{R}\to\mathbb{R}$ as in \eqref{cut-off}.	The product $\phi_\tau(t)a_m(t,x)$ satisfies
	
	\begin{equation}\label{equation:L^infty}
		\left \{
		\begin{aligned}
			\partial_{t}(\phi_\tau(t)&a_m(t,x))-d_{i}\Delta (\phi_\tau(t)a_m(t,x))\\
			=& a_{m}(x,t)\phi'_\tau(t)+\phi_\tau(t)\left(-a_m+a_1\prod\limits_{j=2}^{m-1}a_j^{\alpha_j}\right) \qquad   \mbox{ in } \Omega_{\tau,\tau+2}\\
			\nabla_x &(\phi_\tau(t)a_m(t,x)).\gamma=  0 \qquad   \qquad \ \  \qquad \qquad \quad \ \mbox{ on } \partial\Omega_{\tau,\tau+2} \\
			&\phi_{\tau}a_m(\tau,x)= 0 \qquad \qquad \qquad \qquad \qquad \qquad \ \ \ \mbox{ in } \Omega.
		\end{aligned}
		\right.
	\end{equation}
	Let $\psi(t,x)$ satisfying the following pde
	\begin{equation*}
		\left \{
		\begin{aligned}
			\partial_{t}\psi(t,x)-d_{m}\Delta \psi(t,x)=& a_{m}(t,x)\phi'_\tau(t)+\phi_\tau(t)a_1\displaystyle\prod_{j=2}^{m-1}a_j^{\alpha_j} \qquad \mbox{ in } \Omega_{\tau,\tau+2} \nonumber \\
			\nabla_x \psi(t,x).\gamma= & 0 \qquad  \qquad \qquad \qquad \qquad  \qquad \quad \ \ \ \  \ \mbox{ on } \partial\Omega_{\tau,\tau+2} \nonumber \\
			\psi (\tau,x)=& 0 \qquad \qquad \qquad  \ \ \ \ \qquad \qquad \qquad  \quad \ \mbox{ in } \Omega. \nonumber
		\end{aligned}
		\right .
	\end{equation*}

	Remark that maximum principle yields $\displaystyle{\phi_{\tau}a_i(t,x)\leq \psi_i(t,x)}$ in $t\Omega_{\tau,\tau+2}$. Furthermore integrability estimation (see Appendix, Theorem \ref{estimation 1}) yields
	
	\begin{align*}
		\Vert  \psi (t,x) \Vert_{\mathrm{L}^{\infty}(\Omega_{\tau,\tau+2})} 
		\leq  C_{IE}(\Omega,d_m,N,p) \left\Vert a_{m}(t,x)\phi'_\tau(t)+\phi_\tau(t)a_1\displaystyle\prod\limits_{j=2}^{m-1}a_j^{\alpha_j} \right\Vert_{\mathrm{L}^p(\Omega_{\tau,\tau+2})} .
	\end{align*}
	Thanks to non-negativity of the solution, the following estimate holds:
	\begin{align*}
		\Vert \phi_{\tau}a_m(t,x) \Vert_{\mathrm{L}^{\infty}(\Omega_{\tau,\tau+2})} 
		\leq  \Vert \psi_m(t,x) \Vert_{\mathrm{L}^{\infty}(\Omega_{\tau,\tau+2})} \leq C_{IE}(\Omega,d_m,N,p)\left(2C_0M_{\phi}+2C_1K_{\infty}^Q\right)(1+\tau)
	\end{align*}
	where the constant $K_{\infty}$ is the bound from Proposition \ref{L^infty estimation of a_i}, the exponent $\displaystyle{Q=\sum\limits_{j=2}^{m-1}\alpha_j}$ and the constant $C_1$ is from \eqref{1 d_1 new}.
	
	Define the following time independent constant:
	\[
	C_2:=C_{IE}(\Omega,d_m,N,p)\big(2C_0M_{\phi}+2C_1K_{\infty}^{Q}\big)
	\]
	We thus have
	\begin{align*}
		\Vert a_m \Vert_{\mathrm{L}^{\infty}(\Omega_{\tau+1,\tau+2})} \leq C_2(1+\tau). \qquad \forall \tau\geq 0.
	\end{align*}
	Recall the differential equation satisfied by  $a_m$:
	\begin{equation*}
		\left \{
		\begin{aligned}
			\partial_{t}a_m(t,x)-d_{i}\Delta a_m(t,x)&=-a_m+a_1\prod\limits_{j=2}^{m-1}a_j^{\alpha_j} \qquad \ \mbox{in}\ \Omega_{0,2}\\
			\nabla_x a_m(t,x).\gamma&=0 \qquad \qquad \qquad \qquad \ \ \ \ \mbox{on} \ \partial\Omega_{0,2}\\
			a_m(0,x)& =a_{m,0} \qquad \qquad \qquad \qquad \mbox{in}\ \Omega.
		\end{aligned}
		\right .
	\end{equation*}
	The solution corresponding to the above equation  can be expressed as:
	\[
	0\leq a_m\leq \int_{\Omega}G_{d_m}(t,0,x,y)a_{m,0}(y) \rm{d}y+ \int_{0}^{t}\int_{\Omega} G_{d_m}(t,s,x,y)\left(-a_m+a_1\prod\limits_{j=2}^{m-1}a_j^{\alpha_j}\right)(s,y)\rm{d}y\rm{d}s \quad (t,x)\in\Omega_{0,2}.
	\]
	where $G_{d_m}$ denotes the Green's function associated with the operator $\partial_t-d_m\Delta$ with Neumann boundary condition. We use the fact that $\displaystyle{a_{m,0}(y)\in\mathrm{L}^{\infty}(\Omega)}$ and $\displaystyle{\int_{\Omega}G_{d_m}(t,0,x,y) \rm{d}y \leq 1}$ for all $t\in(0,2)$. It yields
	\[
	\left\vert \int_{\Omega}G_{d_m}(t,0,x,y)a_{m,0}(y) \rm{d}y \right\vert \leq \Vert a_{m,0}\Vert_{\mathrm{L}^{\infty}(\Omega)}.
	\]
	Green function estimate as in \eqref{Heat kernel estimate} and \eqref{Heat kernel bound function} yields
	\[
	0\leq G_{d_m}(t_1,s,x,y) \leq  g_{d_m}(t_1-s,x-y) \qquad 0\leq s<t_1
	\]
	where the function $g_{d_m}$ as defined in  \eqref{Heat kernel bound function}. Furthermore \eqref{Heat kernel integral estimate} yields the following integral estimate:
	\[
	\left\Vert g_{d_m}\right\Vert_{\mathrm{L}^{z}((-2,2)\times\mathbb{R}^N)} \leq \tilde{\kappa}_{1,d_m} \qquad \forall z\in\left[1,1+\frac{2}{N}\right)
	\]
	for some constant positive $\tilde{\kappa}_{1,d_m}$ depending on $z$. We in particular choose $\displaystyle{\frac{1}{z}=1-\frac{1}{p}}$ which is admissible because $\displaystyle{p>\frac{N+2}{2}}$. Hence, for $t\leq 1$:
	\begin{align*}
		\left \vert \int_{0}^{t}\int_{\Omega} G_{d_m}(t,s,x,y)\left(-a_m+a_1\prod\limits_{j=2}^{m-1}a_j^{\alpha_j}\right)(s,y)\rm{d}y\rm{d}s\right\vert \leq & \int_{0}^{t}\int_{\Omega} g_{d_m}(t-s,x-y)\left(-a_m+a_1\prod\limits_{j=2}^{m-1}a_j^{\alpha_j}\right)(s,y)\rm{d}y\rm{d}s\\
		\leq & \Big\Vert g_{d_m}(t-\cdot,x-\cdot)\Big\Vert_{\mathrm{L}^z((0,t)\times\Omega)} \Big\Vert -a_m+a_1\prod\limits_{j=2}^{m-1}a_j^{\alpha_j}\Big\Vert_{\mathrm{L}^p((0,t)\times\Omega)}\\
		\leq & \Big\Vert g_{d_m}(\cdot,\cdot)\Big\Vert_{\mathrm{L}^z((0,t)\times\mathbb{R}^N)} \Big\Vert -a_m+a_1\prod\limits_{j=2}^{m-1}a_j^{\alpha_j}\Big\Vert_{\mathrm{L}^p((0,1)\times\Omega)}\\
		\leq & \tilde{\kappa}_{1,d_i} \left(C_0+C_1 K_{\infty}^Q\right).
	\end{align*}
Hence
	\[
	\Vert a_m\Vert_{\mathrm{L}^{\infty}(\Omega_{0,2)}}\leq \Vert a_{m,0}\Vert_{\mathrm{L}^{\infty}(\Omega)}+\tilde{\kappa}_{1,d_i} \left(C_0+C_1 K_{\infty}^Q\right). 
	\]
	Taking $\displaystyle{C_3:=\max\left\{C_2, \Vert a_{m,0}\Vert_{\mathrm{L}^{\infty}(\Omega)}+\tilde{\kappa}_{1,d_i} \left(C_0+C_1 K_{\infty}^Q\right) \right\}}$. We deduce
	\begin{align}\label{polynomial on t 2}
		0\leq a_m(t,x) \leq C_{3}(1+t)^2,
	\end{align}
	where the lower bound comes from the non-negativity of $a_m$.
	Finally, the differential  relation $\partial_t a_1 \leq a_m$, helps us conclude that
	\begin{align}\label{polynomial on t 3}
		0\leq a_1(t,x) \leq a_{1,0}(x)+C_{3}(1+t)^3.
	\end{align}
	Here again we use the fact that $a_1$ is non-negative. 
	
	Thus taking $\displaystyle{K_{pol}=\Vert a_{1,0}\Vert_{L^{\infty}(\Omega)}+C_{3}+K_{\infty}}$ and $\displaystyle{\mu=3}$, we have indeed shown that
	
	\[
	\max\left\{\Vert a_1 \Vert_{\mathrm{L}^{\infty}(\Omega_{0,t})},\Vert a_m \Vert_{\mathrm{L}^{\infty}(\Omega_{0,t})}\right\} \leq K_{pol}(1+t)^{\mu} \qquad \forall t\geq 0.
	\]
\end{proof}

\vspace{.2cm}

We are all equipped to prove the following result which obtains a lower bound for the dissipation functional corresponding to the degenerate system \eqref{degenerate 1}. The value of this result lies in the fact that the lower bound involves the term $\displaystyle{\Vert \delta_{A_1}\Vert_{\mathrm{L}^2(\Omega)}}$ even though such a term is apparently missing in the expression of the dissipation functional.

\vspace{.2cm}

\begin{Prop}\label{Dissipatation theorem d_1}
	Let the dimension $N\geq 4$. Let $(a_1,\cdots,a_m)$ be the positive smooth solution to the degenerate triangular reaction-diffusion system \eqref{degenerate 1}. Suppose further that the closeness assumption \eqref{tri closeness condition imp} is satisfied by the diffusion coefficients $d_m$ and $d_i$ for at least one $i\in\{2,\dots,m-1\}$. Then, there exists a positive constant $\hat{C}$, independent of time, such that
	\[
	D(a_1,\cdots,a_m) \geq \hat{C}(1+t)^{-\frac{N-2}{N-1}}\left(\sum\limits_{i=1}^{m}  \Vert  \delta_{A_i} \Vert_{\mathrm{L}^{2}(\Omega)}^2+\Vert A_m-A_1\displaystyle \prod\limits_{d=2}^{m-1} (A_j)^{\alpha_j}\Vert_{\mathrm{L}^2(\Omega)}^2\right).
	\]
\end{Prop}

\begin{proof} Based on the size of the deviations in the $\mathrm{L}^2$ norm we divide the proof in two cases. First case is about large deviations. Fix an $\epsilon\in(0,1)$.
	
	\vspace{.2cm}
	
	{\textbf{Case 1:}} \qquad $\max\limits_{ i=2,\cdots,m} \big \{ \Vert \delta_{A_i} \Vert^2_{\mathrm{L}^2(\Omega)}  \big \}\geq \epsilon$.
	
	An application of the H\"older inequality in \eqref{dissipation-L_2N/N-2} yields 
	\begin{align*}
		D(a,b,c) \geq & \vert \Omega \vert^{-2/N}\frac{1}{P(\Omega)}\min\limits_{i=2\cdots,m} \{ d_i\} \max\limits_{i=2,\cdots,m} \big \{ \Vert \delta_{A_i} \Vert^2_{\mathrm{L}^2(\Omega)}   \big \} \nonumber \\
		\geq & \epsilon \vert \Omega \vert^{-2/N}\frac{1}{P(\Omega)}\min\limits_{i=2\cdots,m} \{ d_i\} \nonumber \\ 
		\geq& \frac{\epsilon \vert \Omega \vert^{-2/N}}{M_1}\Vert \delta_{A_1}\Vert^2_{\mathrm{L}^2(\Omega)} \min\limits_{i=2\cdots,m} \{ d_i\}\frac{1}{P(\Omega)}
	\end{align*}
	where the last inequality is a consequence of $\displaystyle{  \Vert A_1-\overline{A_1}\Vert^2_{\mathrm{L}^2(\Omega)} \leq \Vert A_1 \Vert^2_{\mathrm{L}^2(\Omega)} \leq \Vert a_1 \Vert_{\mathrm{L}^1(\Omega)} \leq M_1}$, thanks to the mass conservation.

	\vspace{.2cm}
	
	The second case is about small deviations. Here we algebraically manipulate the last term in the dissipation functional.
	
	\vspace{.2cm}
	
	\textbf{Case 2:} \qquad $\max\limits_{ i=2,\cdots,m} \big \{ \Vert \delta_{A_i} \Vert^2_{\mathrm{L}^2(\Omega)}  \big \}\leq \epsilon$. 
	
	We rewrite the term contributed by the rate function to the dissipation functional as follows:
	\begin{align*}
		\left\Vert A_m-A_1\displaystyle \prod\limits_{j=2}^{m-1} (A_j)^{\alpha_j}\right\Vert_{\mathrm{L}^2(\Omega)}^2
		=  \left\Vert \delta_{A_m}+\overline{A_m}-A_1 \left( \displaystyle \prod\limits_{j=2}^{m-1} (\delta_{A_j}+\overline{A_j})^{\alpha_j}-\displaystyle \prod\limits_{j=2}^{m-1} (\overline{A_j})^{\alpha_j}\right)-A_1\displaystyle \prod\limits_{j=2}^{m-1} (\overline{A_j})^{\alpha_j} \right\Vert_{\mathrm{L}^2(\Omega)}^2.
	\end{align*}
	Use of the algebraic inequality $(p-q)^2\geq  \frac{1}{2}p^2-q^2$ leads to
	\begin{align*}
		&\left\Vert \delta_{A_m}+\overline{A_m}-A_1 \left( \displaystyle \prod\limits_{j=2}^{m-1} (\delta_{A_j}+\overline{A_j})^{\alpha_j}-\displaystyle \prod\limits_{j=2}^{m-1} (\overline{A_j})^{\alpha_j}\right)-A_1\displaystyle \prod\limits_{j=2}^{m-1} (\overline{A_j})^{\alpha_j} \right\Vert_{\mathrm{L}^2(\Omega)}^2\\
		&\geq  \frac{1}{2}\left\Vert \overline{A_m}-A_1\displaystyle \prod\limits_{j=2}^{m-1}  (\overline{A_j})^{\alpha_j} \right\Vert_{{L}^2(\Omega)}^2-\left\Vert \delta_{A_m}- A_1 \left( \displaystyle \prod\limits_{j=2}^{m-1} (\delta_{A_j}+\overline{A_j})^{\alpha_j}-\displaystyle \prod\limits_{j=2}^{m-1} (\overline{A_j})^{\alpha_j}\right) \right\Vert_{\mathrm{L}^2(\Omega)}^2\\
		&\geq  \frac{1}{2}\left\Vert \overline{A_m}-A_1\displaystyle \prod\limits_{j=2}^{m-1} (\overline{A_j})^{\alpha_j} \right\Vert_{\mathrm{L}^2(\Omega)}^2-\left\Vert \delta_{A_m}\right\Vert_{\mathrm{L}^2(\Omega)}^2
		- \left\Vert A_1 \left( \displaystyle \prod\limits_{j=2}^{m-1} (\delta_{A_j}+\overline{A_j})^{\alpha_j}-\displaystyle \prod\limits_{j=2}^{m-1} (\overline{A_j})^{\alpha_j}\right) \right\Vert_{\mathrm{L}^2(\Omega)}^2.
	\end{align*}
	
	Employing the mean value theorem on the polynomial $\displaystyle \prod\limits_{j=2}^{m-1} (x_j+\overline{A_j})^{\alpha_j}$, we remark that for all $2\leq i\leq m-1$, there exists $\theta_i$ with $\vert \theta_i\vert \in \big[0,\vert \delta_{A_i}\vert \big]$, such that 
	
	\[
	\displaystyle \prod\limits_{j=2}^{m-1} (\delta_{A_j}+\overline{A_j})^{\alpha_j}-\displaystyle \prod\limits_{j=2}^{m-1} (\overline{A_j})^{\alpha_j} =\displaystyle \sum\limits_{i=2}^{m-1}\alpha_i\delta_{A_i}(\theta_i+\overline{A_i})^{\alpha_i-1}\displaystyle\left( \prod_{j< i}\overline{A_j}^{\alpha_j}\right)\displaystyle\left( \prod_{j>i}^{m-1} (\delta_{A_j}+\overline{A_j})^{\alpha_j}\right).
	\]
	Using this in the earlier bound and thanks to the H\"older inequality, we obtain
	
	\begin{align}
		\left\Vert A_m-A_1\displaystyle  \prod\limits_{j=2}^{m-1} (A_j)^{\alpha_j}\right\Vert_{\mathrm{L}^2(\Omega)}^2 &\geq  \frac{1}{2}\left\Vert \overline{A_m}-A_1\displaystyle \prod\limits_{j=2}^{m-1}  (\overline{A_j})^{\alpha_j} \right\Vert_{\mathrm{L}^2(\Omega)}^2-\vert \Omega \vert^{\frac{2}{N}}\left\Vert \delta_{A_m}\right\Vert_{\mathrm{L}^{\frac{2N}{N-2}}(\Omega)}^2 \nonumber\\  - & \left\Vert a_1 \right\Vert_{\mathrm{L}^{\frac{N}{2}}} \left\Vert \displaystyle \sum\limits_{i=2}^{m-1}\alpha_i\delta_{A_i}(\theta_i+\overline{A_i})^{\alpha_i-1}\displaystyle \left(\prod_{j< i}\overline{A_j}^{\alpha_j}\right)\displaystyle \left( \prod_{j>i}^{m-1} (\delta_{A_j}+\overline{A_j})^{\alpha_j}\right) \right \Vert_{\mathrm{L}^{\frac{2N}{N-2}}(\Omega)}^2. \nonumber
	\end{align}
	This further implies
	\begin{align}
		\left\Vert A_m-A_1\displaystyle  \prod\limits_{j=2}^{m-1} (A_j)^{\alpha_j}\right\Vert_{\mathrm{L}^2(\Omega)}^2 \geq & \frac{1}{2}\left\Vert \overline{A_m}-A_1\displaystyle \prod\limits_{j=2}^{m-1}  (\overline{A_j})^{\alpha_j} \right\Vert_{\mathrm{L}^2(\Omega)}^2-\vert \Omega \vert^{\frac{2}{N}}\left\Vert \delta_{A_m}\right\Vert_{\mathrm{L}^{\frac{2N}{N-2}}(\Omega)}^2 \nonumber \\
		- \left\Vert a_1 \right\Vert_{\mathrm{L}^{\frac{N}{2}}(\Omega)} &\sum\limits_{i=2}^{m-1}\sup\left \{ \displaystyle \alpha_i^2(\theta_i+\overline{A_i})^{2\alpha_i-2}\displaystyle \left(\prod_{j< i}\overline{A_j}^{2\alpha_j}\right)\displaystyle \left(\prod_{j>i}^{m-1} (\delta_{A_j}+\overline{A_j})^{2\alpha_j}\right)\right\} \left\Vert \delta_{A_i} \right\Vert^2_{\mathrm{L}^{\frac{2N}{N-2}}(\Omega)}. \nonumber
	\end{align}
	
	Next, we find an explicit bound on the term containing the supremum using the mass bound as obtained in \eqref{mass bound} and the supremum bound as  obtained in  Proposition \ref{L^infty estimation of a_i}.
	
	\begin{align*}
		\sup&\left \{ \displaystyle\left( \alpha^2_i(\theta_i+\overline{A_i})^{2\alpha_i-2}\displaystyle  \prod_{j< i}\overline{A_j}^{2\alpha_j}\right)\displaystyle \left(\prod_{j>i}^{m-1} (\delta_{A_j}+\overline{A_j})^{2\alpha_j}\right)\right\} %\label{bigger constant} 
		\\
		&\leq  \max\limits_{i=2,\cdots,m-1}\{\alpha^2_i\}\Bigg(1+\sqrt{K_{\infty}}+2\frac{\sqrt{M_1}}{\sqrt{\vert \Omega \vert}}\Bigg)^{\max\limits_{i=2,\cdots,m-1}\{2\alpha_i-2\}} 
		\Bigg(1+{\frac{\sqrt{M_1}}{\sqrt{\vert \Omega \vert}}\Bigg)}^{2Q} \Bigg(1+\sqrt{K_{\infty}}+{\frac{\sqrt{M_1}}{\sqrt{\vert \Omega \vert}}}\Bigg)^{2Q}:=K_1,
	\end{align*} 
	where $Q:=\sum\limits_{i=2}^{m-1} \alpha_i$. Thus we have
	\begin{align}
		\left\Vert A_m-A_1\displaystyle \prod\limits_{j=2}^{m-1} (A_j)^{\alpha_j}\right\Vert_{\mathrm{L}^2(\Omega)}^2 \geq \frac{1}{2}&\left\Vert  \overline{A_m}-A_1\displaystyle \prod\limits_{j=2}^{m-1}  (\overline{A_j})^{\alpha_j} \right\Vert_{\mathrm{L}^2(\Omega)}^2\label{intermediate 1}\\
		&-\vert \Omega \vert^{\frac{2}{N}}\left\Vert \delta_{A_m}\right\Vert_{\mathrm{L}^{\frac{2N}{N-2}}(\Omega)}^2 -K_1\left\Vert a_1 \right\Vert_{\mathrm{L}^{\frac{N}{2}}(\Omega)} \sum\limits_{i=2}^{m-1}\left\Vert \delta_{A_i} \right\Vert^2_{\mathrm{L}^{\frac{2N}{N-2}}(\Omega)}.\nonumber 
	\end{align}
	
	Using the $\mathrm{L}^{\frac{N}{2}}(\Omega)$ estimate on $a_1$ from Lemma \ref{a_1 estimate} yields 
	
	\begin{align*}
		\left\Vert A_m-A_1\displaystyle \prod\limits_{j=2}^{m-1} (A_j)^{\alpha_j}\right\Vert_{\mathrm{L}^2(\Omega)}^2   \geq \frac{1}{2} & \left\Vert \overline{A_m}-A_1\displaystyle \prod_{2}^{m-1}  (\overline{A_j})^{\alpha_j} \right\Vert_{\mathrm{L}^2(\Omega)}^2\\
		&-\vert \Omega \vert^{\frac{2}{N}}\left\Vert \delta_{A_m}\right\Vert_{\mathrm{L}^{\frac{2N}{N-2}}(\Omega)}^2 -K_1 \Tilde{K}(1+t)^{\frac{N-2}{N-1}}\sum\limits_{i=2}^{m-1}\left\Vert \delta_{A_i} \right\Vert^{2}_{\mathrm{L}^{\frac{2N}{N-2}}(\Omega)}.
	\end{align*} 
	
	Note that it follows from \eqref{dissipation-L_2N/N-2} that
	
	\begin{align*}
		D(a_1,\cdots,a_m) \geq \sum\limits_{i=2}^{m} \frac{\alpha_i d_i}{P(\Omega)} \Vert  \delta_{A_i} \Vert_{\mathrm{L}^{\frac{2N}{N-2}}}^2+\eta(1+t)^{-\frac{N-2}{N-1}} \left\Vert A_m-A_1\displaystyle \prod\limits_{j=2}^{m-1} (A_j)^{\alpha_j}\right\Vert_{\mathrm{L}^2(\Omega)}^2
	\end{align*} 
	where
	\[
	\eta= \frac{\min\limits_{2,\cdots,m}\{\alpha_id_i,1\}}{1+P(\Omega)\vert \Omega \vert^{\frac{2}{N}}+P(\Omega)K_1 \Tilde{K}}.
	\]
	Substituting the lower bound we obtained for the last expression, we get
	\begin{align}
		D(a_1,\cdots,a_m) \geq \frac{\eta}{2}(1+t)^{-\frac{N-2}{N-1}}\left\Vert \overline{A_m}-A_1\displaystyle \prod\limits_{j=2}^{m-1}  (\overline{A_j})^{\alpha_j} \right\Vert_{\mathrm{L}^2(\Omega)}^2. \label{intermediate 11}
	\end{align}
	
	We relate the missing term with the above lower bound of dissipation functional. We again divide into two subcases depending on the largeness or smallness of the product term $\displaystyle{\prod\limits_{j=2}^{m-1}  (\overline{A_j})^{\alpha_j}}$.

	\textbf{Subcase 1:} \qquad $\prod\limits_{j=2}^{m-1}  (\overline{A_j})^{\alpha_j}$  is large i.e \quad  $ \prod\limits_{j=2}^{m-1}  (\overline{A_j})^{\alpha_j} \geq \epsilon$. Let
	\[
	A_1 = \frac{\overline{A_m}}{\prod\limits_{j=2}^{m-1}  (\overline{A_j})^{\alpha_j}}(1+\mu (x)).
	\]
	Then the lower bound of the dissipation functional can be expressed as
	\[
	\left\Vert \prod\limits_{j=2}^{m-1}  (\overline{A_j})^{\alpha_j}A_1-\overline{A_m} \right\Vert^2_{\mathrm{L}^2(\Omega)} = \vert \overline{A_m} \vert^2
	\Vert \mu \Vert^2_{\mathrm{L}^2(\Omega)} = \vert \overline{A_m} \vert^2 \overline{{\mu}^2} \vert \Omega \vert. 
	\]
	Furthermore, the missing term or the deviation term $\displaystyle{\Vert\delta_{A_1}\Vert_{\mathrm{L}^2(\Omega)}}$ can be expressed as
	\[
	\Vert A_1- \overline{A_1} \Vert^2_{\mathrm{L}^2(\Omega)} = \frac{\left\vert \overline{A_m} \right\vert^2}{\left\vert \prod\limits_{j=2}^{m-1}  (\overline{A_j})^{\alpha_j} \right\vert^2} \Vert \mu - \overline{\mu} \Vert^2_{\mathrm{L}^2(\Omega)} \leq \frac{\vert \overline{A_m} \vert^2}{\left\vert \prod\limits_{j=2}^{m-1}  (\overline{A_j})^{\alpha_j} \right\vert^2} \overline{{\mu}^2} \vert \Omega \vert. 
	\]
	Hence we can compare the missing term with the lower bound of dissipation functional:
	\[
	\left\Vert A_1\prod\limits_{j=2}^{m-1}  (\overline{A_j})^{\alpha_j}-\overline{A_m} \right\Vert^2_{\mathrm{L}^2(\Omega)} \geq \epsilon \frac{\vert \Bar{A_m} \vert^2}{\left\vert \prod\limits_{j=2}^{m-1}  (\overline{A_j})^{\alpha_j} \right\vert^2} \overline{{\mu}^2} \vert \Omega \vert \ \geq \epsilon  \Vert A_1- \overline{A_1} \Vert^2_{\mathrm{L}^2(\Omega)}. 
	\]
	
	\textbf{Subcase 2:} \qquad $\prod\limits_{j=2}^{m-1}  (\overline{A_j})^{\alpha_j} < \epsilon$. In other words there exists  $ l\in\{2,\cdots,m-1\}$ such that $\displaystyle{\overline{A_l}\leq \Bigg(\epsilon\Bigg)^{\frac{1}{Q}}=:\nu}$ where $Q=\sum\limits_{i=2}^{m-1} \alpha_i$.
	
	Recall that we are treating the case corresponding to
	$\displaystyle{\max\limits_{ i=2,\cdots,m} \big \{ \Vert \delta_{A_i} \Vert^2_{\mathrm{L}^2(\Omega)}  \big \}\leq \epsilon}$.  Hence 
	\[ 
	\Vert \delta_{A_l}\Vert^2_{L^2(\Omega)}= \vert \Omega \vert (\overline{{A_l}^2}-{\overline{A_l}}^2)\leq \epsilon \implies \overline{{A_l}^2}\leq \nu(1+\vert \Omega \vert^{-1}).
	\]
	Furthermore, thanks to the mass conservation property, we arrive at the following lower bound:
	\[
	{\overline{A_m}}^2=\overline{{A_m}^2}-\vert \Omega \vert^{-1}\Vert \delta_{A_m}\Vert^2_{\mathrm{L}^2(\Omega)}= \overline{{A_m}^2}+\overline{{A_l}^2}-\overline{{A_l}^2}-\vert \Omega \vert^{-1}\Vert \delta_{A_m}\Vert^2_{\mathrm{L}^2(\Omega)} \geq M_{l}-\nu (1+2 \vert \Omega \vert^{-1}).
	\]
	Using the algebraic inequality $(p-q)^2\geq  \frac{1}{2}p^2-q^2$, we arrive at
	\[
	\left\Vert  \prod\limits_{j=2}^{m-1}  (\overline{A_j})^{\alpha_j}A_1-\overline{A_m} \right\Vert^2_{\mathrm{L}^2(\Omega)} \geq \frac{\vert \Omega \vert}{2}\left( -2\prod\limits_{j=2}^{m-1}  (\overline{A_j})^{2\alpha_j}\overline{A^2_1}+{\overline{A_m}}^2 \right). 
	\]
	Thanks to the mass conservation property, we obtain
	\[
	\left\Vert  \prod\limits_{j=2}^{m-1}  (\overline{A_j})^{\alpha_j}A_1-\overline{A_m} \right\Vert^2_{\mathrm{L}^2(\Omega)} \geq 	\frac{\vert \Omega \vert}{2} \Bigg[{M_{l}}-\nu (1+2 \vert \Omega \vert^{-1}) -2\epsilon^2 M_{1} \Bigg] \geq
	\frac{\vert \Omega \vert}{2} \Bigg[{M_{l}}-\nu (1+2 \vert \Omega \vert^{-1}) -2\nu M_{1} \Bigg]. 
	\]
	
	Choose $\displaystyle{\nu =\frac{1}{2} \times \frac{1}{1+2 \vert \Omega \vert^{-1}+2M_{1}}{M_{l}}}$. Note that $\nu<1$. It yields
	\[
	\left\Vert  \prod\limits_{j=2}^{m-1}  (\overline{A_j})^{\alpha_j}A_1-\overline{A_m} \right\Vert^2_{\mathrm{L}^2(\Omega)} \geq \frac{\vert \Omega \vert M_{l} }{4} \geq \frac{\vert \Omega \vert M_{l}}{4M_{1}} \Vert A_1- \overline{A_1} \Vert^2_{\mathrm{L}^2(\Omega)}.
	\]	
	
	Define the following time independent constant
	\[ \displaystyle{C_1:=\min\left\{\frac{\epsilon}{2},\frac{\vert \Omega \vert M_{l}}{8M_{1}}, \frac{\eta}{2}, \frac{\epsilon \vert \Omega \vert^{-2/N}}{M_1} \min\limits_{i=2\cdots,m} \{ d_i\}\frac{1}{P(\Omega)}\right\}}.
	\]
	Thus we have
	\[
	D(a_1,\cdots,a_m) \geq C_1(1+t)^{-\frac{N-2}{N-1}}\Vert A_1- \overline{A_1} \Vert^2_{\mathrm{L}^2(\Omega)}.
	\]
	
	Choose $\hat{C}:=\frac{1}{2}\min\left\{ C_1,1,\frac{\alpha_id_i}{C(\Omega)}: i=2,\cdots,m\right\}$. We finally arrive at the following lower bound
	\[
	D(a_1,\cdots,a_m) \geq \hat{C}(1+t)^{-\frac{N-2}{N-1}}\left(\sum\limits_{i=1}^{m}  \Vert  \delta_{A_i} \Vert_{\mathrm{L}^{2}(\Omega)}^2+\left\Vert A_m-A_1\displaystyle \prod\limits_{j=2}^{m-1} (A_j)^{\alpha_j}\right\Vert_{\mathrm{L}^2(\Omega)}^2\right).
	\]
\end{proof}

\vspace{.2cm}

We can mimic the computations in the above proof to demonstrate an analogous result to Proposition \ref{Dissipatation theorem d_1} in dimensions $N\leq 3$. Without proof, we simply state the results below.

\begin{Thm}\label{Dissipatation theorem d_1, N=3}
	Let the dimension $N\leq 3$. Let $(a_1,\cdots,a_m)$ be the positive smooth solution to the degenerate triangular reaction-diffusion system \eqref{degenerate 1}. Then there exists a positive constant $\hat{C}$, independent of time, such that
	\[
	D(a_1,\cdots,a_m) \geq \hat{C}(1+t)^{-\frac{1}{2}}\left(\sum\limits_{i=1}^{m}  \Vert  \delta_{A_i} \Vert_{\mathrm{L}^{2}(\Omega)}^2+\left\Vert A_m-A_1\displaystyle \prod\limits_{j=2}^{m-1} (A_j)^{\alpha_j}\right\Vert_{\mathrm{L}^2(\Omega)}^2\right).
	\]
\end{Thm}

\vspace{.2cm}

The following two results give subexponential decay (in time) estimate on the relative entropy following the method of entropy. The first is for dimensions $N\geq 4$ and the second is for dimensions $N\leq 3$.

\vspace{.2cm}

\begin{Thm}\label{entropy decay }
	Let the dimension $N\geq4$. Let $(a_1,\cdots,a_m)$ be the smooth positive solution to the degenerate triangular reaction-diffusion system \eqref{degenerate 1}. Suppose further that the closeness assumption \eqref{tri closeness condition imp} is satisfied by the diffusion coefficients $d_m$ and $d_i$ for at least one $i\in \{2,\cdots,m-1\}$. Then, for any positive $\epsilon\ll 1$, there exists a finite time $T_{\epsilon}$ and two time independent positive constants $\lambda_1$ and $\lambda_2$ (depending on $\epsilon$, domain $\Omega$, dimension $N$ and initial data) such that
	\[
	E(a_1,\cdots,a_m)- E(a_{1\infty},\cdots,a_{m\infty})
	\leq \lambda_1 e^{-\lambda_2 (1+t)^{\frac{1-\epsilon}{N-1}}} \qquad \mbox{for all} \ t\geq T_{\epsilon}.
	\]
\end{Thm} 

\begin{proof} The relative entropy can be written as
	\begin{align*}
		E(a_1,\cdots,a_m)- E(a_{1\infty},\cdots,a_{m\infty})
		= \sum\limits_{i=1}^{m}\int_{\Omega}&\alpha_i(a_i \ln{a_i}-a_i-a_{i\infty} \ln{a_{i\infty}}+a_{i,\infty})\\
		=& \sum\limits_{i=1}^{m}\int_{\Omega}\alpha_i (a_i \ln{\frac{a_i}{a_{i\infty}}}-a_i+a_{i\infty}).
	\end{align*}
	
	Let us define a function $\Gamma:(0,\infty)\times(0,\infty)\to\mathbb{R}$ as follows:
	\begin{equation}
		\Gamma(x,y) :=
		\left\{
		\begin{aligned}
			&\frac{x \ln\left(\frac{x}{y}\right)-x+y}{\left(\sqrt{x}-\sqrt{y}\right)^2}  \qquad \mbox{ for }x\not=y,
			\\
			&2 \qquad  \qquad \qquad \qquad \ \ \ \mbox{ for }x=y.
		\end{aligned}\right.
	\end{equation}
	It can be shown (see \cite[Lemma 2.1, p.162]{DF06} for details) that the above defined function satisfies the following bound:
	\begin{align}\label{eq:Gamma-bound}
		\Gamma(x,y)\leq C_{\Gamma}\max\left\{1,\ln\left(\frac{x}{y}\right)\right\}
	\end{align}
	for some positive constant $C_{\Gamma}$. Note that using the function $\Gamma$ defined above, the relative entropy can be rewritten as 
	\[
	E(a_1,\cdots,a_m)-E(a_{1\infty},\cdots,a_{1\infty})= \sum\limits_{1}^{m}\int_{\Omega}\alpha_i \Gamma(a_i,a_{i\infty})(A_i-A_{i\infty})^2
	\]
	Using the aforementioned bound for $\Gamma$, we obtain
	\begin{align*}
		E(a_1,\cdots,a_m)&-E(a_{1\infty},\cdots,a_{m\infty})\\
		\leq & \max\limits_{i=1,\cdots,m}\{\alpha_i\}C_{\Gamma}\max\limits_{i=1,\cdots,m} \big\{1,\ln{(\Vert a_i\Vert_{\mathrm{L}^{\infty}(\Omega)}+1)}+\vert \ln{a_{i\infty}}\vert\big\}\sum\limits_{1}^{m}\Vert A_i-A_{i\infty}\Vert_{\mathrm{L}^2(\Omega)}^2.
	\end{align*} 
	
	Recall that $\mathrm{L}^{\infty}(\Omega_t)$ norms of $a_2,\cdots,a_{m-1}$ are uniformly bounded (Proposition \ref{L^infty estimation of a_i})  and that the $\mathrm{L}^{\infty}(\Omega_t)$ norms of $a_1$ and $a_m$ can grow atmost polynomially in time (Proposition \ref{polynomial on t}).  Consider the following constant
	\[
	C_{1}:=\max\limits_{1\leq i\leq m}\{\alpha_i\}C_{\Gamma}\max\limits_{1\leq i\leq m}\left(1+\vert \ln{a_{i\infty}}\vert+\vert \ln{(K_{pol})}\vert+\mu\right),
	\] 
	where $K_{pol}$ as in Proposition \ref{polynomial on t}. Then we get
	\begin{align}\label{entropy decay 1}
		E(a_1,\cdots,a_m)&-E(a_{1\infty},\cdots,a_{m\infty})  
		\leq   C_{1}(1+\ln{(1+t)})\sum\limits_{i=1}^{m}\Vert A_i-A_{i\infty}\Vert_{\mathrm{L}^2(\Omega)}^2. \nonumber
	\end{align}
	So, other than the logarithmic growth of the supremum norm of the solution, the growth of relative entropy depends on the $\mathrm{L}^2(\Omega)$ norm of the deviation of the $A_i$ from $A_{i\infty}$. Observe from \eqref{dissipation-L_2N/N-2} and \eqref{dissipation-L_2N/N-2 d_m} that the dissipation functional is also related to this $\mathrm{L}^2(\Omega)$ norm of $\delta_{A_i}$. Observe that
	\[
	\Vert A_i-A_{i\infty}\Vert_{\mathrm{L}^2(\Omega)}^2 \leq 3\left( \Vert A_i-\overline{A_i}\Vert_{\mathrm{L}^2(\Omega)}^2+\Vert \overline{A_i}-\sqrt{\overline{A_i^2}}\Vert_{\mathrm{L}^2(\Omega)}^2+\Vert \sqrt{\overline{A_i^2}}-A_{i\infty}\Vert_{\mathrm{L}^2(\Omega)}^2\right). 
	\]
	The following observation says that the first term on the right hand side dominate the second term:
	\begin{align*}
		\Vert \overline{A_i}-\sqrt{\overline{A_i^2}}\Vert_{\mathrm{L}^2(\Omega)}^2= &\vert\Omega\vert \big\vert
		\overline{A_i}-\sqrt{\overline{A_i^2}} \big\vert^2 = \vert\Omega\vert \left( \overline{A_i^2}+\overline{A_i}^2-2\overline{A_i}\sqrt{\overline{A_i^2}} \right)\\ &
		\leq\vert\Omega\vert \left( \overline{A_i^2}-\overline{A_i}^2\right) \leq \Vert A_i-\overline{A_i}\Vert_{\mathrm{L}^2(\Omega)}^2.
	\end{align*}
	Here we used the fact that $\displaystyle{\overline{A_i}\leq \sqrt{\overline{A_i^2}}}$, thanks to H\"older inequality. Hence
	\[
	\Vert A_i-A_{i\infty}\Vert_{\mathrm{L}^2(\Omega)}^2 \leq 6\left( \Vert A_i-\overline{A_i}\Vert_{\mathrm{L}^2(\Omega)}^2+\Vert \sqrt{\overline{A_i^2}}-A_{i\infty}\Vert_{\mathrm{L}^2(\Omega)}^2\right). 
	\]
	
	Hence we deduce
	\[
	\sum\limits_{i=1}^m\Vert A_i-A_{i\infty}\Vert_{\mathrm{L}^2(\Omega)}^2 \leq 6\left(\sum\limits_{i=1}^m \Vert A_i-\overline{A_i}\Vert_{\mathrm{L}^2(\Omega)}^2+\sum\limits_{i=1}^m\Vert \sqrt{\overline{A_i^2}}-A_{i\infty}\Vert_{\mathrm{L}^2(\Omega)}^2\right).
	\]
	
	Next we borrow a result from \cite{FLT20} (see Appendix, Theorem \ref{ED}) which says that there exists $C_{EB}>0$, depending only on the domain and the equilibrium state $(a_{1\infty},\cdots,a_{m\infty})$ such that
	
	\[
	\sum\limits_{i=1}^{m}\left\Vert \sqrt{\overline{A_i^2}}-A_{i\infty} \right\Vert_{\mathrm{L}^2(\Omega)}^2 \leq C_{EB} \left ( \sum\limits_{i=1}^{m}  \Vert  \delta_{A_i} \Vert_{\mathrm{L}^{2}(\Omega)}^2+\left\Vert A_m-\displaystyle \prod_{j=1}^{m-1} (A_j)^{\alpha_j}\right\Vert_{\mathrm{L}^2(\Omega)}^2\right).
	\]
	
	Putting it all together along with the result from Proposition \ref{Dissipatation theorem d_1}, we arrive at
	\begin{align*}
		E(a_1,\cdots,a_m)-E(a_{1\infty},\cdots,a_{m\infty}) 
		\leq 6C_{1}(1+\ln{(1+t)})\frac{1+C_{EB}}{\hat{C}}(1+t)^{\frac{N-2}{N-1}}D(a_1,\cdots,a_m).
	\end{align*}
	Let us denote a new time independent positive constant 
	\[
	\displaystyle{C_{2}:=6C_{1}\frac{1+C_{EB}}{\hat{C}}}.
	\]
	We arrive at the following entropy-entropy dissipation inequality
	\begin{align}\label{entropy decay 2 tri}
		E(a_1,\cdots,a_m)-E(a_{1\infty},\cdots,a_{m\infty}) 
		\leq  C_{2}(1+\ln{(1+t)}) (1+t)^{\frac{N-2}{N-1}}D(a_1,\cdots,a_m).
	\end{align}
	If we choose a positive $\epsilon\ll 1$, then there exists $t=T_{\epsilon}$, such that $\displaystyle{\ln{(1+t)}<(1+t)^{\frac{\epsilon}{N-1}}}$ for all $\displaystyle{t\geq T_{\epsilon}}$. Consider the  positive constant $\displaystyle{C_{3}=2C_{2}(1+\ln{(1+T_{\epsilon})})}$.  Then the above relation \eqref{entropy decay 2 tri} between the relative entropy and the entropy dissipation can be written as
	\begin{align}\label{entropy decay 3 tri}
		E(a_1,\cdots,a_m)-E(a_{1\infty},\cdots,a_{m\infty})\leq  C_{3}(1+t)^{\frac{N-2+\epsilon}{N-1}}D(a_1,\cdots,a_m). 
	\end{align}
	
	Furthermore, the following differential relation holds
	\[
	\frac{d}{dt}\big(E(a_1,\cdots,a_m)-E(a_{1\infty},\cdots,a_{m\infty})\big)=-D(a_1,\cdots,a_m).
	\]
	Using the relation between relative entropy and entropy dissipation from \eqref{entropy decay 3 tri}, yields
	\begin{align*}
		\frac{d}{dt}\big(E(a_1,\cdots,a_m)-E(a_{1\infty},\cdots,a_{m\infty})\big)
		\leq- \frac{1}{C_{3}} (1+t)^{-\frac{N-2+\epsilon}{N-1}}\big(E(a_1,\cdots,a_m)-E(a_{1\infty},\cdots,a_{m\infty})\big).
	\end{align*}
	Hence Gr\"onwall lemma yields the following sub-exponential decay of relative entropy:
	\begin{align*} 
		E(a_1,\cdots,a_m)&-E(a_{1\infty},\cdots,a_{m\infty}) \\ 
		\leq &  \big(E(a_{1,0},\cdots,a_{m,0})-E(a_{1\infty},\cdots,a_{m\infty})\big)e^{\frac{1-\epsilon}{C_{3}(N-1)}}e^{-\frac{1-\epsilon}{C_{3}(N-1)}(1+t)^{\frac{1-\epsilon}{N-1}}}.
	\end{align*}
	
	Let $\displaystyle{C_{4}:=E(a_{1,0},\cdots,a_{m,0})-E(a_{1\infty},\cdots,a_{m\infty})}$. Furthermore, denote two time independent positive constant $\lambda_1$ and $\lambda_2$, in the following way:
	
	\begin{equation}\label{main constant 1 tri}
		\left \{
		\begin{aligned}
			\lambda_1=& C_{4}e^{\frac{1-\epsilon}{C_{3}(N-1)}}
			\\
			\lambda_2=& \frac{1-\epsilon}{C_{3}(N-1)}.
		\end{aligned}
		\right .
	\end{equation}
	With these constants the required decay estimate follows.
\end{proof}

We can mimic the computations in the above proof to demonstrate an analogous result to Theorem  \ref{entropy decay } in dimensions $N\leq 3$. Without proof, we simply state the result below.

\vspace{.3cm}

\begin{Thm}\label{entropy decay N=1,2,3 }
	Let dimension $N\leq 3$. Let $(a_1,\cdots,a_m)$ be the positive smooth solution to the degenerate triangular reaction-diffusion system \eqref{degenerate 1}. Then for any positive $\epsilon\ll 1$, there exists a finite time $T_{\epsilon}$ and two time independent positive constant $\lambda_1$ and $\lambda_2$ (depending on $\epsilon$, domain $\Omega$, dimension $N$ and initial data), such that 
	\[
	E(a_1,\cdots,a_m)-E(a_{1\infty},\cdots,a_{m\infty})
	\leq  \lambda_1e^{-\lambda_2(1+t)^{\frac{1-\epsilon}{2}}} \qquad \forall t\geq T_{\epsilon}.
	\]
\end{Thm} 

\vspace{.2cm}

We are now equipped to prove our main result of this section.
\begin{proof}[Proof of Theorem \ref{theorem convergence 1 tri}:]
	We have already obtained sub-exponential decay (in time) of the relative entropy in Proposition \ref{entropy decay } (for dimension $N\ge4$) and in Proposition \ref{entropy decay N=1,2,3 } (for dimension $N<4$). Hence the sub-exponential decay in the $\mathrm L^1$-norm is a direct consequence of the following Czisz\'ar-Kullback-Pinsker type inequality that relates relative entropy and the $\mathrm L^1$-norm (see Appendix, Theorem \ref{lower bound of entropy}):
	\begin{align*}
		E(a_i:i=1,& \cdots,m)-E(a_{i\infty}:i=1,\cdots,m) \geq C_{LE}\sum_{i=1}^{m}\Vert a_i-a_{i\infty}\Vert_{\mathrm{L}^1(\Omega)}^{2}. 
	\end{align*}
\end{proof}

\section{Large time behaviour: degenerate system $d_m=0$}\label{dm=0 convergence section}

This section deals with the degenerate system \eqref{triangular d_m=0} where the diffusion coefficient $d_m$ vanish. Our ultimate goal is to employ the method of entropy to demonstrate that the solution to the degenerate system \eqref{degenerate 1} converges to the equilibrium states given by \eqref{triangular equilibrium state 1}-\eqref{triangular equibrium state 2}. We restrict ourselves to dimension $N\leq 2$ as the global-in-time positive classical solution is known for the  degenerate system \eqref{triangular d_m=0} up to dimension $N\leq 2$ [HS]. Here, we only present the case for $N=2$, the case corresponding to $N=1$ will go along the same line. We recall the differential equality
\[
\frac{dE}{dt}(a_1,\cdots,a_m)=-D(a_1,\cdots,a_m),
\]
which ensures the decay of entropy with time. Integrating the above relation in time on $[0,T]$ yields
\begin{flalign*}
	& \sum\limits_{i=1}^m \sup\limits_{t\in[0,T]}\int_{\Omega}  \alpha_i(a_i(\ln{a_i}-1)+1)+ \sum\limits_{i=1}^{m}\int_{\Omega_T} \alpha_i d_i \frac{\vert \nabla a_i \vert^2}{a_i}\nonumber \\ & +\int_{\Omega_T}  \left(a_m-\displaystyle \prod\limits_{j=1}^{m-1} (a_j)^{\alpha_j}\right)\ln{\left(\frac{a_m}{\prod\limits_{j=1}^{m-1} (a_j)^{\alpha_j}}\right)}\leq E(a_{1,0},\cdots,a_{m,0}). %\label{entropy-entropy dissipation traingular}
\end{flalign*}

The above estimates helps us conclude the following bounds
\begin{equation}\label{Uniform-integrability traingular}
	\left \{
	\begin{aligned}
		&\int_{\Omega} \sum_1^m a_i \leq e^2\vert\Omega\vert+\frac{1}{\min\limits_{1\leq i\leq m}\alpha_i}=\Tilde{M}_1 
		\\
		\int_{\Omega} \sum_1^m  \vert a_i\ln{a_i}\vert & \leq \Tilde{M}_1+\frac{1}{\min\limits_{1\leq i\leq m}\alpha_i}E(a_{i,0};i=1,2\cdots,m) )+ e^{-1}\vert\Omega\vert.
	\end{aligned}
	\right .
\end{equation}

\vspace{.2cm}

The following estimate is a particular   $\mathrm{L}^p(\Omega_{0,2})$ integral estimate of each diffusive species.

\vspace{.2cm}

\begin{Prop}\label{a_i L^p d_m [0,1]}
	Let dimension $N=2$. Let $(a_1,\cdots,a_m)$ be the positive smooth solution to the degenerate triangular reaction-diffusion system \eqref{triangular d_m=0}. Then
	there exists a constant $C_{0,p,1}>0$, depends on the domain $\Omega$, $p$ and initial condition, such that
	\[
	\sup\limits_{t\in[0,2)}\Vert a_i \Vert_{\mathrm{L}^p(\Omega)}\leq C_{0,p,1} \qquad \forall i=1,\cdot\cdot,m-1.
	\]
\end{Prop}

\begin{proof}  The species $a_i$ is a positive subsolution of the following equation in the time interval $(0,2)$, for $2\leq m-1$:
	\begin{equation}\label{eq:tilde ai-am}
		\left \{
		\begin{aligned}
			\partial_{t}a_i(t,x)-d_i\Delta a_i(t,x)
			\leq &a_m  \qquad \qquad \qquad  \mbox{ in }\Omega_{0,2}
			\\
			\nabla_x a_i(t,x).\gamma =&  0 \qquad   \qquad \quad   \ \  \mbox{ on } \partial\Omega_{0,2} 
			\\
			a_i(0,x) =& a_{i,0} \qquad   \quad  \qquad \ \ \mbox{ in } \Omega.
		\end{aligned}
		\right.
	\end{equation}
	Thanks to maximum principle, the solution corresponding to the equation \eqref{eq:tilde ai-am} can be expressed as:
	\[
	0\leq a_i\leq \int_{\Omega}G_{d_i}(t,0,x,y)a_{i,0}(y) \rm{d}y+ \int_{0}^{t}\int_{\Omega} G_{d_i}(t,s,x,y)a_m(s,y)\rm{d}y\rm{d}s \quad (t,x)\in\Omega_{0,2}.
	\]
	where $G_{d_i}$ denotes the Green's function associated with the operator $\partial_t-d_i\Delta$ with Neumann boundary condition. We use the fact that $\displaystyle{a_{i,0}(y)\in\mathrm{L}^{\infty}(\Omega)}$ and $\displaystyle{\int_{\Omega}G_{d_i}(t,0,x,y) \rm{d}y \leq 1}$ for all $t\in(0,2)$. It yields
	\[
	\left\vert \int_{\Omega}G_{d_i}(t,0,x,y)a_{i,0}(y) \rm{d}y \right\vert \leq \Vert a_{i,0}\Vert_{\mathrm{L}^{\infty}(\Omega)}.
	\]
	Green function estimate as in \eqref{Heat kernel estimate} and \eqref{Heat kernel bound function} yields
	\[
	0\leq G_{d_i}(t_1,s,x,y) \leq  g_{d_i}(t_1-s,x-y) \qquad 0\leq s<t_1
	\]
	where the function $g_{d_i}$ as defined in  \eqref{Heat kernel bound function}.
	Furthermore, in dimension $N=2$, there exists a positive constant $K_{1,d_i}$, depending on the domain and the dimension, such that
	\[
	\Vert g \Vert_{\mathrm{L}^p(\mathbb{R}^N)} \leq K_{1,d_i} t^{-\left(1-\frac{1}{p}\right)}.
	\]

	\vspace{.2cm}
	
	Minkowski's integral inequality yields
	\[
	\Vert a_i \Vert_{\mathrm{L}^p(\Omega)} \leq \Vert  a_{i,0} \Vert_{\mathrm{L}^{\infty}(\Omega)}\vert \Omega\vert^{\frac{1}{p}} + \int_{0}^{t} \left\Vert \int_{\Omega} g(t-s,x-y) a_m(s,y) \right\Vert_{\mathrm{L}^p(\Omega)}.
	\]
	
	An application of Young's convolution inequality yields 
	\[
	\Vert a_i \Vert_{\mathrm{L}^(\Omega)} \leq \Vert  a_{i,0} \Vert_{\mathrm{L}^{\infty}(\Omega)}\vert \Omega\vert^{\frac{1}{p}} + \int_{0}^{t} \Vert g (t-s)\Vert_{\mathrm{L}^p(\mathbb{R}^N)} \Vert a_m(s) \Vert_{\mathrm{L}^1(\Omega)}.
	\]
	For $t\in[0,2)$, we arrive at the following estimate
	\[
	\Vert a_i \Vert_{\mathrm{L}^p(\Omega)} \leq \max\limits_{i=1,\cdots,m}\left\{\Vert  a_{i,0} \Vert_{\mathrm{L}^{\infty}(\Omega)}\right\} \vert \Omega\vert^{\frac{1}{p}}+ \max\limits_{1\leq i\leq m-1}\{K_{1,d_i}\}\frac{1}{p}2^{\frac{1}{p}}\Tilde{M}_1=: C_{0,p,1},
	\]
	where $\Tilde{M}_1$ is the mass bound or the $\mathrm{L}^1(\Omega)$ norm bound as obtained in  \eqref{Uniform-integrability traingular}.
\end{proof}
The following result proves the analogue of Proposition \ref{a_i L^p d_m [0,1]} in the parabolic cylinder $\Omega_{\tau,\tau+1}$.
\vspace{.2cm}

\begin{Prop}\label{a_i L^p d_m}
	Let dimension $N=2$. Let $(a_1,\cdots,a_m)$ be the positive smooth solution to the degenerate triangular reaction-diffusion system  \eqref{triangular d_m=0}. Then,
	there exists a time independent constant $C_{0,p}>0$, depending on the domain, dimension, $p$ and initial condition, such that
	\[
	\sup\limits_{t>0}\Vert a_i \Vert_{\mathrm{L}^p(\Omega)}\leq C_{0,p} \qquad \forall i=1,\cdots,m-1.
	\]
\end{Prop}

\begin{proof} 
    Consider the function $\phi_{\tau}:\mathbb{R}\to\mathbb{R}$ as defined in \eqref{cut-off}.
	The product $\phi_\tau(t)a_i(t,x)$ satisfies for all $1\leq i\leq m-1$:
	\begin{equation} \nonumber
		\left\{
		\begin{aligned}
			\partial_t \phi_{\tau}a_i-d_i \Delta \phi_{\tau}a_i = & \phi'_{\tau}a_i+\phi_{\tau}\left( a_m-\prod\limits_{j=1}^{m-1}a_i^{\alpha_i} \right) \qquad \ \ \mbox{ in } \Omega_T\\
			\nabla_x \phi_{\tau}a_i .\gamma=& 0 \qquad \qquad \qquad \qquad \qquad\qquad \qquad \ \mbox { on } \partial\Omega_T\\
			\phi_{\tau}a_i(\tau,x)=& 0 \qquad \qquad \qquad \qquad \qquad \qquad \qquad \  \mbox{ in }\Omega.
		\end{aligned}
		\right.
	\end{equation}
	Change of variable $t\rightarrow t_1+\tau$ yields
	
	\begin{equation*}
		\left \{
		\begin{aligned}
			\partial_{t_1}(\phi_\tau(t_1&+\tau)a_i(t_1+\tau,x))-d_{i}\Delta (\phi_\tau(t_1+\tau)a_i(t_1+\tau,x))\\ &= a_{i}(t_1+\tau,x)\phi'_\tau(t_1+\tau)+\phi_\tau(t_1+\tau)\left(a_m-\prod\limits_{j=1}^{m-1}a_j^{\alpha_j}\right)  \ \ \ \ \ \ \mbox{\ in \ }\Omega_{0,2} \\
			& \nabla_x (\phi_\tau(t_1+\tau)a_i(t_1+\tau,x)).\gamma=  0 \qquad   \qquad \qquad \qquad \qquad  \qquad  \mbox{ on } \partial\Omega_{0,2} \\
			& \phi_{\tau}a_i(0,x)= 0 \qquad \qquad \qquad \qquad \qquad \ \ \ \qquad  \qquad \qquad \qquad \quad \ \ \mbox{ in } \Omega. 
		\end{aligned}
		\right .
	\end{equation*}
	
	Let $G_{d_i}(t_1,x,y)$ is the heat kernel corresponding to the operator $\partial_{t_1}-d_i \Delta$. Then, for all $t_1\in[0,2]$, we can represent the solution as
	
	\begin{align*}
		\phi_\tau(t_1+\tau)& a_i(t_1+\tau,x)\\
		= &\int_{0}^{t_1}\int_{\Omega}G_{d_i}(t_1-s,x,y)\left(a_{i}(y,s+\tau)\phi^{'}_\tau(s+\tau)+\phi_\tau(s+\tau)\left(a_m-\prod\limits_{j=1}^{m-1}a_j^{\alpha_j}\right)(s+\tau,t)\right).
	\end{align*}
	Thanks to the non-negativity of the Green's function and the species concentrations, and immediate pointwise bound follows
	\begin{align*}	
		\phi_\tau(t_1+\tau)a_i(t_1+\tau,x)
		\leq  \int_{0}^{t_1}\int_{\Omega}G_{d_i}(t_1-s,x,y)\Big(a_{i}(s+\tau,x)\phi^{'}_\tau(s+\tau)+\phi_\tau(s+\tau)a_m(s+\tau,x)\Big).
	\end{align*}
	Green function estimate as in \eqref{Heat kernel estimate} and \eqref{Heat kernel bound function} yields
	\[
	0\leq G_{d_i}(t_1,s,x,y) \leq  g_{d_i}(t_1-s,x-y) \qquad 0\leq s<t_1
	\]
	where the function $g_{d_i}$ as defined in  \eqref{Heat kernel bound function}.
	Furthermore, in dimension $N=2$, there exists a positive constant $K_1$, depending on the domain and the dimension, such that
	\[
	\Vert g_{d_i} \Vert_{\mathrm{L}^p(\mathbb{R}^N)} \leq K_{1,d_i} t^{-\left(1-\frac{1}{p}\right)}.
	\]
	
	Substituting the above pointwise bound on the Green function, the earlier inequality reads
	\begin{align*}
		\phi_\tau(t_1+\tau)a_i(t_1+\tau,x)
		\leq  \int_{0}^{t_1}\int_{\Omega}g_{d_i}(t_1-s,x-y)\Big(a_{i}(s+\tau,x)\phi^{'}_\tau(s+\tau)+\phi_\tau(s+\tau)a_m(s+\tau,x)\Big).
	\end{align*}
	
	Minkowski's integral inequality yields
	\begin{align*}
		\Vert  \phi_\tau(t_1+\tau)a_i(t_1+\tau,x) \Vert_{\mathrm{L}^p(\Omega)} 
		\leq \int_{0}^{t_1}\left\Vert \int_{\Omega}g_{d_i}(t_1-s,x-y)\Big(a_{i}(s+\tau,x)\phi^{'}_\tau(s+\tau)+\phi_\tau(s+\tau)a_m(s+\tau,x)\Big)\right\Vert_{\mathrm{L}^p(\Omega)}.
	\end{align*}
	
	An application of Young's convolution inequality yields
	\begin{align*}
		\Vert \phi_\tau(t_1+\tau)a_i(t_1+\tau,x) \Vert_{\mathrm{L}^p(\Omega)}
		\leq  \int_{0}^{t_1}\Vert g_{d_i}\Vert_{\mathrm{L}^p(\mathbb{R}^N)}\left\Vert a_{i}(s+\tau,x)\phi^{'}_\tau(s+\tau)+\phi_\tau(s+\tau)a_m(s+\tau,x)\right\Vert_{\mathrm{L}^1(\Omega)}.
	\end{align*}
	
	Hence for all $t_1\in[0,2]$, $a_i$ satisfy the following estimate for all $1\leq i\leq m-1$:
	\[
	\Vert \phi_\tau(t_1+\tau)a_i(t_1+\tau,x) \Vert_{L^p(\Omega)}\leq \max\limits_{1\leq i\leq m-1}\{K_{1,d_i}\}(M_{\phi}+1)\Tilde{M}_1 \int_{0}^{2} t^{-(1-\frac{1}{p})}\leq p2^{\frac{1}{p}}\max\limits_{1\leq i\leq m-1}\{K_{1,d_i}\}(M_{\phi}+1)\Tilde{M}_1,
	\]
	where $\Tilde{M}_1$ is the mass bound or the $\mathrm{L}^1(\Omega)$ norm bound as obtained in \eqref{Uniform-integrability traingular}. This further implies
	\[
	\sup\limits_{t\in[\tau+1,\tau+2]}\Vert a_i(t,x) \Vert_{\mathrm{L}^p(\Omega)}\leq p2^{\frac{1}{p}}\max\limits_{1\leq i\leq m-1}\{K_{1,d_i}\}(M_{\phi}+1)\Tilde{M}_1.
	\]
	
	Since $\tau\geq0$ is arbitrary, consider the following time independent positive constant
	\[
	C_{0.p}=: p2^{\frac{1}{p}}\max\limits_{1\leq i\leq m-1}\{K_{1,d_i}\}(M_{\phi}+1)\Tilde{M}_1+\sup\limits_{t\in[0,2]}\Vert a_i(t,x) \Vert_{\mathrm{L}^p(\Omega)}=p2^{\frac{1}{p}}\max\limits_{1\leq i\leq m-1}\{K_{1,d_i}\}(M_{\phi}+1)\Tilde{M}_1+C_{0,p,1},
	\]
	which yields
	\[
	\sup\limits_{t>0}\Vert a_i(t,x) \Vert_{\mathrm{L}^p(\Omega)}\leq C_{0,p} \qquad \forall i=1,\cdots,m-1.
	\]
\end{proof}

\vspace{.2cm}

The following result derives an estimate for the species concentrations $a_1,\cdots,a_m$ in the $\mathrm{L}^{\infty}(\Omega_t)$ norm.

\vspace{.2cm}

\begin{Thm}\label{polynomial on t d_m}
	Let dimension $N=2$. Let $(a_1,\cdots,a_m)$ be the positive smooth solution to the degenerate triangular reaction-diffusion system \eqref{triangular d_m=0}. Then there exists time independent positive constants $K_{pol}$ and $\mu$ such that
	\[
	\Vert a_i \Vert_{\mathrm{L}^{\infty}(\Omega_{0,t})} \leq K_{pol}(1+t)^{\mu} \qquad \forall i=1,\cdots,m.
	\]
\end{Thm}

Proof: Let $Q=\sum\limits_{i=1}^{m-1}\alpha_i$. Proposition \ref{a_i L^p d_m} ensures that there exists a positive constant $C_{0,3Q}$ such that 
\[
\sup\limits_{t>0}\Vert a_i(t,x) \Vert_{\mathrm{L}^{3Q}(\Omega)}\leq C_{0,3Q}\qquad \forall i=1,\cdots,m-1.
\]
From the differential relation $\partial_t a_m \leq \displaystyle\prod\limits_{j=1}^{m-1}a_j^{\alpha_j}$, we can express $a_m$ as
\[
a_m \leq a_{m,0}+\int_{0}^{t}\displaystyle\prod\limits_{j=1}^{m-1}a_j^{\alpha_j}.
\]

We apply the Minkowski's integral inequality, H\"older inequality and Young's inequality respectively. It yields

\[
\Vert a_m \Vert_{\mathrm{L}^3(\Omega)}\leq \Vert a_{m,0} \Vert_{\mathrm{L}^3(\Omega)}+\int_{0}^{t}\sum\limits_{1}^{m-1}\Vert a_i \Vert_{\mathrm{L}^{3Q}(\Omega)}^Q\leq \Vert a_{m,0} \Vert_{\mathrm{L}^3(\Omega)}+C_{0,3Q}^Qt.
\]

Choosing a constant $\displaystyle{C_1:=\Vert a_{m,0} \Vert_{L^3(\Omega)}+C_{0,3Q}^Q}$, the above estimate reads as
\begin{align}\label{L^3 bound on a_m d_m}
	\Vert a_m \Vert_{\mathrm{L}^3(\Omega)}\leq C_1(1+t).
\end{align}

The species concentration  $a_i$ satisfies, for all $1\leq i\leq m-1$:
\begin{equation*}
	\left \{
	\begin{aligned}
		\partial_{t} a_i -d_i \Delta a_i = & a_m-\displaystyle\prod\limits_{j=1}^{m-1}a_j^{\alpha_j} \qquad \qquad \ \mbox{in}\ \Omega_T\\
		\nabla_x a_i .\gamma =& 0 \qquad \qquad \qquad \qquad \ \ \ \ \mbox{on} \ \partial\Omega_T\\
		a_{i}(0,x)=& a_{i,0} \qquad \qquad \qquad \qquad \ \  \mbox{in} \ \Omega.
	\end{aligned}
	\right .
\end{equation*}
The solution corresponding to the above equation  can be expressed as:
\[
0\leq a_i\leq \int_{\Omega}G_{d_i}(t,0,x,y)a_{i,0}(y) \rm{d}y+ \int_{0}^{t}\int_{\Omega} G_{d_i}(t,s,x,y)a_m(s,y)\rm{d}y\rm{d}s \quad (t,x)\in\Omega_{T}.
\]
where $G_{d_i}$ denotes the Green's function associated with the operator $\partial_t-d_i\Delta$ with Neumann boundary condition. We use the fact that $\displaystyle{a_{i,0}(y)\in\mathrm{L}^{\infty}(\Omega)}$ and $\displaystyle{\int_{\Omega}G_{d_i}(t,0,x,y) \rm{d}y \leq 1}$ for all $t\in(0,T)$. It yields
\[
\left\vert \int_{\Omega}G_{d_i}(t,0,x,y)a_{i,0}(y) \rm{d}y \right\vert \leq \Vert a_{i,0}\Vert_{\mathrm{L}^{\infty}(\Omega)}.
\]
Green function estimate as in \eqref{Heat kernel estimate} and \eqref{Heat kernel bound function} yields
\[
0\leq G_{d_i}(t_1,s,x,y) \leq  g_{d_i}(t_1-s,x-y) \qquad 0\leq s<t_1
\]
where the function $g_{d_i}$ as defined in  \eqref{Heat kernel bound function}. Furthermore, in dimension $N=2$, there exists a positive constant $K_{1,d_i}$, depending on the domain and the dimension, such that
\[
\Vert g_{d_i} \Vert_{\mathrm{L}^{\frac23}(\mathbb{R}^N)} \leq K_{1,d_i} t^{-\left(1-\frac{2}{3}\right)}.
\]
Hence, employing Minkowski's integral inequality we obtain the following estimate
\[
\Vert a_i(t,x) \Vert_{\mathrm{L}^{\infty}(\Omega)} \leq  \max\limits_{1\leq i\leq m-1}\left\{\Vert a_{i,0} \Vert_{\mathrm{L}^{\infty}(\Omega)}\right\}+\max\limits_{1\leq i\leq m-1}\{K_{1,d_i}\} \int_{0}^{t} (t-s)^{-\frac{1}{3}} \Vert a_m(s) \Vert_{\mathrm{L}^3(\Omega)}.
\]

Recall the estimate on $\Vert a_m(s) \Vert_{\mathrm{L}^3(\Omega)}$ as described in \eqref{L^3 bound on a_m d_m}. It yields
\begin{align*}
	\Vert a_i(t,x) \Vert_{\mathrm{L}^{\infty}(\Omega)} \leq  \max\limits_{1\leq i\leq m-1}&\left\{\Vert a_{i,0} \Vert_{\mathrm{L}^{\infty}(\Omega)}\right\}+\frac{3}{2}\max\limits_{1\leq i\leq m-1}\{K_{1,d_i}\}C_1(1+t)^2(1+t)^{\frac{2}{3}}\\
	\leq &\max\limits_{1\leq i\leq m-1}\left\{\Vert a_{i,0} \Vert_{\mathrm{L}^{\infty}(\Omega)}\right\} +\frac{3}{2}\max\limits_{1\leq i\leq m-1}\{K_{1,d_i}\}C_1(1+t)^3.
\end{align*}

Choosing the constant $\displaystyle{C_2:=\max\limits_{1\leq i\leq m-1}\left\{\Vert a_{i,0} \Vert_{\mathrm{L}^{\infty}(\Omega)}\right\}+\frac{3}{2}\max\limits_{1\leq i\leq m-1}\{K_{1,d_i}\}C_1}$, the above estimate reads as
\begin{align}\label{Linfinity d_i d_m}
	\Vert a_i(t,x) \Vert_{\mathrm{L}^{\infty}(\Omega_{0,t})} \leq C_2(1+t)^3 \qquad \forall i=1,\cdots,m-1.
\end{align}

Consider the differential relation $\displaystyle{\partial_t a_m \leq \displaystyle\prod_{1}^{m-1}a_i^{\alpha_i}}$.  The pointwise bound of $a_m$ can be expressed as
\[
a_m \leq a_{m,0}+\int_{0}^{t}\displaystyle\prod\limits_{j=1}^{m-1}a_j^{\alpha_j}.
\]
Minkowski's integral inequality and  \eqref{Linfinity d_i d_m} yields 
\[
\Vert a_{m}\Vert_{\mathrm{L}^{\infty}(\Omega_{0,t})} \leq \Vert a_{m,0}\Vert_{\mathrm{L}^{\infty}(\Omega_{0,t})}+\frac{1}{3Q+1}C_2^Q(1+t)^{3Q+1}.
\]

Consider the time independent positive constants
\begin{equation*} 
	\left \{
	\begin{aligned} 
		K_{pol}:=&C_2+\Vert a_{m,0}\Vert_{L^{\infty}(\Omega_{0,t})}+\frac{1}{3Q+1}C_2^Q\\
		\mu:=& 3Q+1.
	\end{aligned}\right .
\end{equation*}
Thus, we arrive at
\[
\Vert a_i \Vert_{\mathrm{L}^{\infty}(\Omega_{0,t})} \leq K_{pol}(1+t)^{\mu} \qquad \forall i=1,\cdots,m.
\]

\vspace{.3cm}

We are all equipped to prove the following result which obtains a lower bound for the dissipation functional corresponding to the degenerate system \eqref{triangular d_m=0}. The value of this result lies in the fact that the lower bound involves the term  $\displaystyle{\Vert \nabla \delta_{A_m} \Vert_{\mathrm{L^2}(\Omega)}}$ even though such a term is apparently missing in the expression of the dissipation functional.

\vspace{.3cm}

\begin{Lem}\label{dissipation d_m}
	Let dimension $N=2$. Let $(a_1,\cdots,a_m)$ be the positive smooth solution to the degenerate triangular reaction-diffusion system \eqref{triangular d_m=0}. Then, there exists a time independent positive constant $\hat{C}$, such that
	\[
	D(a_1,\cdots,a_m) \geq \hat{C}\left(\sum\limits_{i=1}^{m}  \Vert  \delta_{A_i} \Vert_{\mathrm{L}^{2}(\Omega)}^2+\left\Vert A_m-\displaystyle \prod\limits_{j=1}^{m-1} (A_j)^{\alpha_j}\right\Vert_{\mathrm{L}^2(\Omega)}^2\right).
	\]
\end{Lem}

\begin{proof} Thanks to Poincar\'e inequality the following lower bound of dissipation functional we have obtained in \eqref{dissipation-L_6 d_m} 
	\begin{align}\label{dissipation-L_6 d_m th}
		D(a_1,\cdots,a_m) \geq \sum\limits_{i=1}^{m-1} \frac{\alpha_i d_i}{P(\Omega)} \Vert  \delta_{A_i} \Vert_{\mathrm{L}^{6}(\Omega)}^2+\left\Vert A_m-\displaystyle \prod\limits_{j=1}^{m-1} (A_j)^{\alpha_j}\right\Vert_{\mathrm{L}^2(\Omega)}^2.
	\end{align}
	We show the missing term can be recovered from the last term in the dissipation functional. We can rewrite the last term as
	\begin{align*}%\label{L^2 intermediate 1 d_m}
		\left\Vert A_m-\displaystyle \prod\limits_{j=1}^{m-1} (A_j)^{\alpha_j}\right\Vert_{\mathrm{L}^2(\Omega)}^2
		= \left\Vert A_m- \left( \displaystyle \prod\limits_{j=1}^{m-1} (\delta_{A_j}+\overline{A_j})^{\alpha_j}-\displaystyle \prod\limits_{j=1}^{m-1} (\overline{A_j})^{\alpha_j}\right)-\displaystyle \prod\limits_{j=1}^{m-1} (\overline{A_j})^{\alpha_j} \right\Vert_{\mathrm{L}^2(\Omega)}^2.
	\end{align*}
	Use of the algebraic inequality $\displaystyle{(p-q)^2\geq  \frac{1}{2}p^2-q^2}$ leads to 
	\begin{align}
		\left\Vert A_m-\displaystyle \prod\limits_{j=1}^{m-1} (A_j)^{\alpha_j}\right\Vert_{\mathrm{L}^2(\Omega)}^2	\geq \frac{1}{2}\left\Vert A_m-\displaystyle \prod\limits_{j=1}^{m-1}  (\overline{A_j})^{\alpha_j} \right\Vert_{\mathrm{L}^2(\Omega)}^2-\left \Vert  \displaystyle \prod\limits_{j=1}^{m-1} (\delta_{A_j}+\overline{A_j})^{\alpha_j}-\displaystyle \prod\limits_{j=1}^{m-1} (\overline{A_j})^{\alpha_j}\right\Vert_{\mathrm{L}^2(\Omega)}^2. \nonumber 
	\end{align}
	
	We use mean value theorem on the polynomial $\displaystyle \prod_{1}^{m-1} (x_j+\overline{A_j})^{\alpha_j}$, we remark that for all $1\leq i\leq m-1$, there exists $\theta_i$ with $\vert \theta_i\vert \in \big[0,\vert \delta_{A_i}\vert \big]$, such that 
	
	\[
	\displaystyle \prod\limits_{j=1}^{m-1} (\delta_{A_j}+\overline{A_j})^{\alpha_j}-\displaystyle \prod\limits_{j=1}^{m-1} (\overline{A_j})^{\alpha_j} =\displaystyle \sum\limits_{i=1}^{m-1}\alpha_i\delta_{A_i}(\theta_i+\overline{A_i})^{\alpha_i-1}\displaystyle \left(\prod_{j< i}\overline{A_j}^{\alpha_j}\right)\displaystyle\left( \prod_{j>i}^{m-1} (\delta_{A_j}+\overline{A_j})^{\alpha_j}\right).
	\]
	
	Let's choose a positive, time independent constant  $\displaystyle{\Lambda_1:= \max\limits_{i=1,\cdots,m-1}\Bigg\{2^{Q},\alpha_i,\displaystyle \alpha_i\prod_{j< i}\overline{A_j}^{\alpha_j}, \frac{\sqrt{\Tilde{M}_1}}{\sqrt{\vert \Omega \vert}}\Bigg\}}$. For all $1\leq i\leq m-1$, it yields the following bounds
	
	\begin{equation*}
		\left \{
		\begin{aligned}
			\overline{A_i}\leq \frac{\left(\displaystyle{\int_{\Omega}a_i}\right)^{\frac{1}{2}}}{\sqrt{\vert \Omega \vert}}\leq & \frac{\sqrt{\Tilde{M}_1}}{\sqrt{\vert \Omega \vert}}=\Lambda_1 
			\\
			\left\vert(\theta_i+\overline{A_i})\right\vert^{\alpha_i-1}\leq  2^{\alpha_i-1}(\vert \theta_i \vert^{\alpha_i-1}+& \Lambda_1^{\alpha_i-1})\leq  \Lambda_1\left(\vert \delta_{A_i}\vert^{\alpha_i-1}+\Lambda_1^{\alpha-1}\right)
			\\
			\left\vert(\delta_{A_j}+\overline{A_j})\right\vert^{\alpha_j} \leq  & (1+a_j^{\alpha_j})\leq  (\Lambda_1+a_j^{\alpha_j}) 
			\\
			\vert \delta_{A_i} \vert^{\alpha_j-1}\leq & \Lambda_1(\Lambda_1+a_i^{\alpha_j-1}+\Lambda_1^{\alpha_j-1}),
		\end{aligned}
		\right.
	\end{equation*}
	
	where $\Tilde{M}_1$ is mass bound or the $\mathrm{L}^1(\Omega)$ norm bound as obtained in \eqref{Uniform-integrability traingular}. H\"older inequality yields
	\begin{align*}
		\left\Vert \displaystyle \prod\limits_{j=1}^{m-1}  (\delta_{A_j}+\overline{A_j})^{\alpha_j}- \displaystyle \prod\limits_{j=1}^{m-1} (\overline{A_j})^{\alpha_j} \right\Vert_{\mathrm{L}^2(\Omega)}^2 \leq & \left\Vert \displaystyle \sum\limits_{i=1}^{m-1}\alpha_i\delta_{A_i}(\theta_i+\overline{A_i})^{\alpha_i-1}\displaystyle \left(\prod_{j< i}\overline{A_j}^{\alpha_j}\right)\displaystyle \left(\prod_{j>i}^{m-1} (\delta_{A_j}+\overline{A_j})^{\alpha_j}\right) \right\Vert_{\mathrm{L}^2(\Omega)}^2\\
		\leq & \sum\limits_{i=1}^{m-1} \big\Vert \delta_{A_i} \big\Vert_{\mathrm{L}^6(\Omega)}^2\left\Vert \Lambda_1^4\displaystyle \prod_{j>i}^{m-1}\left(\Lambda_1+a_j^{\alpha_j}\right)\left(\Lambda_1+a_i^{\alpha_j-1}+2\Lambda_1^{\alpha_j-1}\right)\right\Vert_{\mathrm{L}^{\frac{3}{2}}(\Omega)}\\
		\leq & \sum\limits_{i=1}^{m-1} \big\Vert \delta_{A_i} \big\Vert_{\mathrm{L}^6(\Omega)}^2  3(1+\Lambda_1)^{Q+m+4}\left\Vert
		\prod_{j>i}^{m-1} \left(1+a_j^{\alpha_j}\right)\left(1+a_i^{\alpha_j-1}\right)\right\Vert_{\mathrm{L}^{\frac{3}{2}}(\Omega)}.
	\end{align*}	
	Young's inequality yields
	\begin{align*}
		\left\Vert \displaystyle \prod\limits_{j=1}^{m-1}  (\delta_{A_j}+\overline{A_j})^{\alpha_j}- \displaystyle \prod\limits_{j=1}^{m-1} (\overline{A_j})^{\alpha_j} \right\Vert_{\mathrm{L}^2(\Omega)}^2 \leq & 3(1+\Lambda_1)^{Q+m+4}  \sum\limits_{i=1}^{m-1} \big\Vert \delta_{A_i} \big\Vert_{\mathrm{L}^6(\Omega)}^2 \left\Vert \sum\limits_{j=i}^{m-1}\left(1+a_j^{\alpha_j}\right)^m \right\Vert_{\mathrm{L}^{\frac{3}{2}}(\Omega)}
	\end{align*}
	
	We use the fact that  $\displaystyle{(x+y)^m\leq 2^m(x^m+y^m)}$ for all $x,y\geq 0$. It yields
	
	\begin{align*}
		\left\Vert \displaystyle \prod\limits_{j=1}^{m-1}  (\delta_{A_j}+\overline{A_j})^{\alpha_j}- \displaystyle \prod\limits_{j=1}^{m-1} (\overline{A_j})^{\alpha_j} \right\Vert_{\mathrm{L}^2(\Omega)}^2 \leq & 2^{m}3(1+\Lambda_1)^{Q+m+4}\left\Vert \sum\limits_{i=1}^{m-1}\left(1+a_i^{m\alpha_i}\right) \right \Vert_{\mathrm{L}^{\frac{3}{2}}(\Omega)}\sum\limits_{i=1}^{m-1} \big\Vert \delta_{A_i} \big\Vert_{\mathrm{L}^6(\Omega)}^2\\
		\leq &2^{m}3(1+\Lambda_1)^{Q+m+4}\left(m\vert \Omega \vert^{\frac{2}{3}}+\left\Vert \sum\limits_{i=1}^{m-1} a_i^{m\alpha_i}\right\Vert_{\mathrm{L}^{\frac{3}{2}}(\Omega)}\right)\sum\limits_{i=1}^{m-1} \big\Vert \delta_{A_i} \big\Vert_{\mathrm{L}^6(\Omega)}^2
		\\
		\leq & 2^{m}3(1+\Lambda_1)^{Q+m+4}(1+\vert \Omega \vert)\left(2m+\left\Vert \sum\limits_{i=1}^{m-1} a_i\right\Vert_{\mathrm{L}^{\frac{m3Q}{2}}(\Omega)}^{mQ}\right)\sum\limits_{i=1}^{m-1} \big\Vert \delta_{A_i} \big\Vert_{\mathrm{L}^6(\Omega)}^2.
	\end{align*}
	
	Thanks to Proposition \ref{a_i L^p d_m}, we have hold over the integral expression $\displaystyle{\Vert a_i\Vert_{\mathrm{L}^{\frac{m3Q}{2}}(\Omega)}}$ for all $1\leq i\leq m-1$. Thus, we arrive at 
	\[
	\left\Vert \displaystyle \prod\limits_{j=1}^{m-1}  (\delta_{A_j}+\overline{A_j})^{\alpha_j}- \displaystyle \prod\limits_{j=1}^{m-1} (\overline{A_j})^{\alpha_j} \right\Vert_{\mathrm{L}^2(\Omega)}^2 \leq 2^{m}3(1+\Lambda_1)^{Q+m+4}(1+\vert \Omega\vert)\left (2m+mC_{0,m3Q}^{mQ}\right)\sum\limits_{i=1}^{m-1} \big\Vert \delta_{A_i} \big\Vert_{\mathrm{L}^6(\Omega)}^2.
	\]
	
	Let's choose  $\displaystyle{C_1:=2^{m}3(1+\Lambda_1)^{Q+m+4}(1+\vert \Omega\vert)\left(2m+mC_{0,m3Q}^{mQ}\right)}$, where the constant $C_{0,m3Q}$ defined in the  Proposition \ref{a_i L^p d_m}. The above estimate reads as
	\begin{align}
		\left\Vert \displaystyle \prod\limits_{j=1}^{m-1}  (\delta_{A_j}+\overline{A_j})^{\alpha_j}- \displaystyle \prod\limits_{j=1}^{m-1} (\overline{A_j})^{\alpha_j} \right\Vert_{\mathrm{L}^2(\Omega)}^2\leq  C_1 \sum\limits_{i=1}^{m-1} \big\Vert \delta_{A_i} \big\Vert_{\mathrm{L}^6(\Omega)}^2. \label{intermediate 1 d_m}
	\end{align}
	
	Hence, we obtain 
	
	\begin{align}\label{intermediate 2 d_m}
		\left\Vert A_m-\displaystyle \prod\limits_{j=1}^{m-1} (A_j)^{\alpha_j}\right\Vert_{\mathrm{L}^2(\Omega)}^2 \geq \frac{1}{2}\left\Vert A_m-\displaystyle \prod\limits_{j=1}^{m-1}  (\overline{A_j})^{\alpha_j} \right\Vert_{\mathrm{L}^2(\Omega)}^2
		- C_1 \sum\limits_{i=1}^{m-1} \big\Vert \delta_{A_i} \big\Vert_{\mathrm{L}^6(\Omega)}^2. 
	\end{align}
	
	We recover the missing term $\displaystyle{\Vert \delta_{A_m}\Vert_{\mathrm{L}^2(\Omega)}^2}$ from the term $\left\Vert A_m-\displaystyle \prod\limits_{j=1}^{m-1}  (\overline{A_j})^{\alpha_j} \right\Vert_{\mathrm{L}^2(\Omega)}^2$.
	\\
	
	Let $\displaystyle{A_m(x)=\prod\limits_{j=1}^{m-1} (\overline{A_j})^{\alpha_j}(1+\mu(x))} $. Then
	\[
	\left\Vert A_m-\displaystyle \prod\limits_{j=1}^{m-1}  (\overline{A_j})^{\alpha_j} \right\Vert_{\mathrm{L}^2(\Omega)}^2 = \Big(\prod\limits_{j=1}^{m-1}  (\overline{A_j})^{\alpha_j} \Big)^2\Vert \mu  \Vert_{\mathrm{L}^2(\Omega)}^2
	\]
	and
	\[
	\Vert \delta_{A_m}\Vert_{\mathrm{L}^2(\Omega)}^2=\Big(\prod_{1}^{m-1}  (\overline{A_j})^{\alpha_j} \Big)^2 \Vert \mu-\overline{\mu}\Vert_{\mathrm{L}^2(\Omega)}^2\leq \Big(\prod_{1}^{m-1}  (\overline{A_j})^{\alpha_j} \Big)^2\Vert \mu  \Vert_{\mathrm{L}^2(\Omega)}^2.
	\]
	
	Thus, we recover the missing term $\Vert \delta_{A_m}\Vert_{\mathrm{L}^2(\Omega)}^2$ as a lower bound of the following term
	
	\[
	\left\Vert A_m-\displaystyle \prod\limits_{j=1}^{m-1}  (\overline{A_j})^{\alpha_j} \right\Vert_{\mathrm{L}^2(\Omega)}^2\geq \Vert \delta_{A_m}\Vert_{\mathrm{L}^2(\Omega)}^2.
	\]
	
	Hence  \eqref{intermediate 2 d_m} and \eqref{dissipation-L_6 d_m th}, leads to
	
	\[
	\left( 1+\max\limits_{i=1,\cdots,m-1}\left\{\frac{P(\Omega)}{\alpha_id_i}\right\}C_1\right)D(a_1,\cdots,a_m) \geq \frac{1}{2}C_1\max\limits_{i=1,\cdots,m-1}\left\{\frac{P(\Omega)}{\alpha_id_i}\right\}\Vert \delta_{A_m}\Vert_{\mathrm{L}^2(\Omega)}^2.
	\]
	
	We choose the following time independent constant
	\[
	C_2:=\frac{\frac{1}{2}C_1\max\limits_{i=1,\cdots,m-1}\left\{\frac{P(\Omega)}{\alpha_id_i}\right\}}{\left( 1+\max\limits_{i=1,\cdots,m-1}\left\{\frac{P(\Omega)}{\alpha_id_i}\right\}C_1\right)}.
	\]
	
	The above estimate reads as
	\begin{align}
		D(a_1,\cdots,a_m) \geq C_2 \Vert \delta_{A_m}\Vert_{\mathrm{L}^2(\Omega)}^2.
		\label{dissipation with the missing term}
	\end{align}
	
	Finally, choice of the time independent positive constant $\displaystyle{\hat{C}:= \frac{1}{2}\min\left\{C_2,\frac{\alpha_id_i}{P(\Omega)}: i=1,\cdots,m\right\}}$, yields 
	\[
	D(a_1,\cdots,a_m) \geq \hat{C}\left(\sum_1^{m}  \Vert  \delta_{A_i} \Vert_{\mathrm{L}^{2}(\Omega)}^2+\Vert A_m-A_1\displaystyle \prod_{2}^{m-1} (A_j)^{\alpha_j}\Vert_{\mathrm{L}^2(\Omega)}^2\right).
	\]
\end{proof}

\vspace{.2cm}

The following result gives subexponential decay (in time) estimate on the relative entropy following the method of entropy.

\vspace{.3cm}

\begin{Thm}\label{entropy decay d_m}
	Let dimension $N=2$. Let  $(a_1,\cdots,a_m)$ be the positive smooth solution to the degenerate triangular reaction-diffusion system \eqref{triangular d_m=0}. Then for any positive $\epsilon\ll 1$, there exists a finite time $T_{\epsilon}$ and two time independent positive constants $\lambda_1$ and $\lambda_2$ (depending on $\epsilon$, domain $\Omega$, dimension $N$ and initial condition) such that 
	\[
	E(a_1,\cdots,a_m)- E(a_{1\infty},\cdots,a_{m\infty})
	\leq  \lambda_1 e^{-\lambda_2(1+t)^{1-\epsilon}} \qquad \forall t\geq T_{\epsilon}.
	\]
\end{Thm} 
The proof of the above proposition is exactly similar to the proof of Proposition \ref{entropy decay }. Hence we skip the proof. Note that the constants $\lambda_1$ and $\lambda_2$ appearing in the above theorem  depend on the constant $\hat{C}$ (appearing in Proposition \ref{dissipation d_m}), the constants $K_{pol}$ and $\mu$ (appearing in Proposition \ref{polynomial on t d_m}). Finally, we are all equipped to prove our main result of this section.

\vspace{.2cm}

\begin{proof}[Proof of Theorem \ref{theorem comvergence 2 tri}]
	We have already obtained sub-exponential decay (in time) of the relative entropy  in Proposition \ref{entropy decay d_m}. Hence the sub-exponential decay in the $\mathrm L^1$-norm is a direct consequence of the following Czisz\'ar-Kullback-Pinsker type inequality that relates relative entropy and the $\mathrm L^1$-norm (see Appendix, Theorem \ref{lower bound of entropy}):
	\begin{align*}
		E(a_i:i=1,& \cdots,m)-E(a_{i\infty}:i=1,\cdots,m) \geq C_{LE}\sum_{i=1}^{m}\Vert a_i-a_{i\infty}\Vert_{\mathrm{L}^1(\Omega)}^{2}. 
	\end{align*}
\end{proof}

\bibliography{ref3.bib}

\appendix
\section{some useful results}\label{sec:app}
\begin{Lem}[Poincar\'e-Wirtinger inequality]\label{Poincare-Wirtinger}
	There exists a positive constant $P(\Omega)$, depending only on $\Omega$ and $q$, such that
	\[
	P(\Omega) \left\Vert \nabla f\right\Vert^2_{\mathrm L^2(\Omega)} \ge \left\Vert f - \overline{f} \right\Vert^2_{\mathrm L^q(\Omega)} \qquad \mbox{ for all }f\in\mathrm H^1(\Omega),
	\]
	where
	\begin{equation*}
		q = 
		\left\{
		\begin{aligned}
			\frac{2N}{N-2} & \qquad \mbox{ for }\, \, N \ge 3,
			\\
			\in [1,\infty) & \qquad \mbox{ for }\, \, N=2,
			\\
			\in [1,\infty)\cup\{\infty\} & \qquad \mbox{ for }\, \, N=1.
		\end{aligned}
		\right.
	\end{equation*}
	We refer to $P(\Omega)$ as the Poincar\'e constant.
\end{Lem}
\begin{Thm}[Second Order Regularity and Integrability estimation] \label{estimation 1}
	Let $d>0$ and let $\tau\in[0,T)$. Take $\theta \in \mathrm L^p(\Omega_{\tau,T})$ for some $1<p< +\infty$. Let $\psi$ be the solution to the backward heat equation:
	\begin{equation*}
		\left\{
		\begin{aligned}
			\partial_t \psi + d \Delta \psi & = -\theta \qquad \mbox{ for }(t,x)\in\Omega_{\tau,T},
			\\
			\nabla \psi \cdot n(x) & = 0 \qquad \ \ \mbox{ for }(t,x)\in[\tau,T]\times\partial\Omega,
			\\
			\psi(T,x) & = 0 \qquad \ \  \mbox{ for }x\in\Omega.
		\end{aligned}
		\right.
	\end{equation*}
	Then, there exists a positive constant $C_{SOR}$, depending only on the domain $\Omega$, the dimension $N$ and the exponent $p$ such that the following maximal regularity holds:
	\begin{align}\label{eq:estimate-laplacian}
		\Vert \Delta \psi \Vert_{\mathrm L^p(\Omega _{\tau,T})} \leq \frac{C_{SOR}}{d} \Vert \theta \Vert_{\mathrm L^p(\Omega_{\tau,T})}.
	\end{align}
	Moreover, if \ $\theta\geq 0$ then $\psi(t,x)\geq 0$ for almost every $(t,x)\in\Omega_{\tau,T}$. Furthermore, we have
	\begin{equation}\label{eq:Ls-estimates}
		\begin{aligned}
			\mbox{If }\, & p<\frac{N+2}{2}\, \mbox{ then } \qquad \left\Vert \psi \right\Vert_{\mathrm L^s(\Omega_{\tau,T})} \le C_{IE} \left\Vert \theta \right\Vert_{\mathrm L^p(\Omega_{\tau,T})} \qquad \mbox{ for all }\, \, s < \frac{(N+2)p}{N+2-2p}
			\\
			\mbox{If }\, & p=\frac{N+2}{2}\, \mbox{ then } \qquad \left\Vert \psi \right\Vert_{\mathrm L^s(\Omega_{\tau,T})} \le C_{IE} \left\Vert \theta \right\Vert_{\mathrm L^p(\Omega_{\tau,T})} \qquad \mbox{ for all }\, \, s < \infty
		\end{aligned}
	\end{equation}
	where the constant $C_{IE}=C_{IE}(T-\tau, \Omega,d,p,s)$ and
	\begin{align}\label{eq:Linfty-estimate}
		\mbox{if }\, & p>\frac{N+2}{2}\, \mbox{ then } \qquad\left\Vert \psi \right\Vert_{\mathrm L^\infty(\Omega _{\tau,T})} \leq  C_{IE} \left\Vert \theta \right\Vert_{\mathrm L^p(\Omega_{\tau,T})}.
	\end{align}
\end{Thm}
The proof of \eqref{eq:estimate-laplacian} can be found in \cite[Theorem 1]{Lam87}. We refer to the constant $C_{SOR}$ as the second order regularity constant. Proof of the estimates \eqref{eq:Ls-estimates} can be found in \cite[Lemma 3.3]{CDF14} and the estimate \eqref{eq:Linfty-estimate} was derived in \cite[Lemma 4.6]{Tan18}. We refer to the constant $C_{IE}$ as the integrability estimation constant.
\begin{Thm}[Regularity of backward heat equation]\label{regularity of backward heat}
	Let $\Theta\in C_c^{\infty(\Omega_{0,T})}$ and let $\nu$ satisfy the following equation:
	\begin{equation*}
		\left \{
		\begin{aligned}
			-\partial_t \nu - d \Delta \nu=& \Theta \qquad \mbox{in}\ \Omega_{0,T}\\
			\nabla_{x} \nu.\gamma=& 0 \ \ \   \mbox{on}\ (0,T)\times\partial\Omega\\
			\nu(T,\cdot)=&0 \qquad \ \mbox{in}\ \Omega.
		\end{aligned}
		\right .
	\end{equation*}
	Then, there exists a positive constant $C_{q,T}$ such that:
	\[
	\Vert \partial_t \nu \Vert_{\mathrm{L}^q(\Omega_{0,T})}+\Vert \Delta \nu \Vert_{\mathrm{L}^q(\Omega_{0,T})}+\sup\limits_{t\in[0,T]}\Vert \nu \Vert_{\mathrm{L}^p(\Omega)}\leq C_{q,T}\Vert \Theta \Vert_{\mathrm{L}^q(\Omega_{0,T})}.
	\]
\end{Thm}
Proof can be found in \cite{rothe06},\cite{Amann1985}. It can be derived from the Theorem \ref{estimation 1} also.
\begin{Thm}[$p^{\textrm{th}}$ order integrability estimation]\label{PE}
	Let $p\in(2,\infty)$ and let $p'$ be its H\"older conjugate. Let $M(t,x)$ be such that the following holds
	\[
	\theta\leq M(t,x) \leq \Theta \quad \forall (t,x)\in\Omega_T,
	\]
	for some fixed positive constants $\theta,\Theta$. Let $\psi_0\in\mathrm L^p(\Omega)$ and let $\psi$ be a weak solution to
	\begin{equation*}
		\left\{
		\begin{aligned}
			\partial_t \psi - \Delta \left( M \psi\right) & = 0 \qquad \qquad \mbox{ for }(t,x)\in\Omega_T,
			\\
			\nabla \psi \cdot n(x) & = 0 \qquad \qquad \mbox{ for }(t,x)\in[0,T]\times\partial\Omega,
			\\
			\psi(0,x) & = \psi_0(x) \qquad \mbox{ for }x\in \Omega.
		\end{aligned}
		\right.
	\end{equation*}
	Then the following estimate holds
	\[
	\left\Vert \psi \right\Vert_{\mathrm L^p(\Omega_T)} \le \left( 1 + \Theta K_{\theta,\Theta,p'}\right) T^\frac{1}{p} \left\Vert \psi_0 \right\Vert_{\mathrm L^p(\Omega)},
	\]
	where the constant $K_{\theta,\Theta,p'}$ is given by
	\[
	K_{\theta,\Theta,p'} := \frac{C^{PRC}_{\frac{\theta+\Theta}{2},p'}\left(\frac{\Theta-\theta}{2}\right)}{1 - C^{PRC}_{\frac{\theta+\Theta}{2},p'}\left(\frac{\Theta-\theta}{2}\right)}
	\qquad
	\mbox{ provided we have }
	\qquad
	C^{PRC}_{\frac{\theta+\Theta}{2},p'}\left(\frac{\Theta-\theta}{2}\right) < 1.
	\]
	Here, the constant $C^{PRC}_{r,p'}$ is the best constant in the following parabolic regularity estimate:
	\[
	\left\Vert \Delta \phi \right\Vert_{\mathrm L^{p'}(\Omega_T)} \le C^{PRC}_{r,p'} \left\Vert f \right\Vert_{\mathrm L^{p'}(\Omega_T)},
	\]
	where $\phi,f:[0,T]\times\Omega\to\mathbb{R}$ are any two functions such that $f\in \mathrm L^{p'}(\Omega_T)$ and they satisfy
	\begin{equation*}
		\left\{
		\begin{aligned}
			\partial_t \phi + r \Delta \phi & = f \qquad \mbox{ for }(t,x)\in\Omega_T,
			\\
			\nabla \phi \cdot n(x) & = 0 \qquad \mbox{ for }(t,x)\in[0,T]\times\partial\Omega,
			\\
			\phi(T,x) & = 0 \qquad \mbox{ for }x\in\Omega.
		\end{aligned}
		\right.
	\end{equation*}
\end{Thm}
It has to be noted that $C^{PRC}_{r,p^{'}}< \infty$ for $r>0$ and $C^{PRC}_{r,2} \leq \frac{1}{r}$\ and depends  only on $r,p^{'}$, the domain and on the dimension, i.e it is independent of time. Moreover, as $C^{PRC}_{r,p^{'}}< \infty$, if we take the difference between $\theta$ and $\Theta$ sufficiently small, then we have the required property that 
\[
C^{PRC}_{\frac{\theta+\Theta}{2},p'}\left(\frac{\Theta-\theta}{2}\right) < 1.
\]
Proof of the above theorem can be found in \cite[Proposition 1.1]{CDF14}.
\begin{Thm}[Entropy -Dissipation Bound Inequality] \label{ED}
	
	Let $\Omega\subset \mathbb{R}^{N}$ be a bounded domain. For $1\leq i\leq m$, let $a_i(t,\cdot):\Omega\rightarrow\mathbb{R}_{+}\cup \{0\}\in \mathrm{L}^2(\Omega)$ satisfies 
	\[
	a_{i}(t,x)+a_{m}(t,x)=M_{im}>0 \qquad \forall i=1,\cdots,m-1,\forall t\geq 0,
	\]
	where for all $1\leq i\leq m$, $M_{i}$ are positive constants.
	Furthermore for all $1\leq i\leq m$, let $a_{i\infty}\in\mathbb{R}_{+}\cup \{0\}$ satisfies
	\[ 
	a_{i\infty}+a_{m\infty}=\frac{M_{i}}{\vert \Omega\vert}\qquad \forall i=1,\cdots,m-1
	\]
	and
	\[
	a_{m\infty}=\displaystyle{\prod_{1}^{m-1}a_{j\infty}^{\alpha_j}}.
	\]
	For all $1\leq i\leq m$, let's denote $A_i=\sqrt{a_i}, A_{i\infty}=\sqrt{a_{i\infty}}$. Then there exists a constant $C_{EB}>0$, depends only on the domain and $M_{i}, \forall i=1,\cdots.m-1$, such that
	
	\[
	\sum\limits_{i=1}^{m}\left\Vert \sqrt{\overline{A_i^2}}-A_{i\infty} \right\Vert_{L^2(\Omega)}^2 \leq C_{EB} \left ( \sum\limits_{i=1}^{m}  \Vert A_i-\overline{A_i} \Vert_{L^{2}(\Omega)}^2+\left\Vert A_m-\displaystyle \prod\limits_{j=1}^{m-1} (A_j)^{\alpha_j}\right\Vert_{L^2(\Omega)}^2\right).
	\]
\end{Thm}
Proof can be found in the article \cite{FLT20}. Here $\overline{A_i^2}$ and $\overline{A_i}$ is average of $A_i^2$ and $A_i$ over the domain $\Omega$. The three species case can be found in \cite{DF06}.
\begin{Lem}[Czisz\'ar-Kullback-Pinsker Inequality]\label{CKP inequality}
	
	Let $\Omega\in\mathbb{R}^N$ be an open set. $f,g:\Omega\to\mathbb{R}$ be two non-negative functions such that  $ \Vert f \Vert_{\mathrm{L}^1(\Omega)}=\Vert g \Vert_{\mathrm{L}^1(\Omega)}=1$. Then,
	\[ 
	\frac{1}{2}\Vert f-g \Vert_{\mathrm{L}^1(\Omega)}^2 \leq \int_{\Omega}f\ln{\left(\frac{f}{g}\right)}.
	\]
\end{Lem}
\begin{Thm}[{Lower Bound of Entropy}] \label{lower bound of entropy}
	For $1\leq i\leq m$, let $a_i$ be the solution of the triangular reaction-diffusion system \eqref{triangular  model} and $a_{i\infty}$ are the equilibrium states \eqref{triangular equilibrium state 1}. Consider the following entropy functional
	\[
	E(a_i:i=1,2\cdots,m)=\sum\limits_{i=1}^m\int_{\Omega}  \alpha_i(a_i(\ln{a_i}-1)+1).
	\] 
	Then there exists a constant $C_{LE}>0$, depends only on the dimension and equilibrium point $a_{i\infty}, \ \forall i=1,\cdots,m-1$, such that
	\begin{align*}
		E(a_i:i=1,& \cdots,m)-E(a_{i\infty}:i=1,\cdots,m) \geq C_{LE}\sum_{i=1}^{m}\Vert a_i-a_{i\infty}\Vert_{\mathrm{L}^1(\Omega)}^{2}. 
	\end{align*}
\end{Thm}
This is a direct application of Czisz\'ar-Kullback-Pinsker inequality. For Three species case, proof can be found in \cite{DF06}. For more general case proof can be found in \cite{FLT20}.

\end{document}